\documentclass{article}

\usepackage[margin=1in]{geometry}
\usepackage{graphicx}
\usepackage{multirow}
\usepackage{amsmath,amssymb,amsfonts}
\usepackage{amsthm}
\usepackage{mathrsfs}
\usepackage[title]{appendix}
\usepackage{xcolor}
\usepackage{textcomp}
\usepackage{manyfoot}
\usepackage{booktabs}
\usepackage{array}
\usepackage{algorithm}
\usepackage{algorithmicx}
\usepackage{algpseudocode}
\usepackage{listings}
\usepackage{hyperref}
\usepackage{cleveref}
\usepackage{tikz}
\usepackage{makecell}
\usepackage{authblk}

\DeclareMathOperator{\ex}{ext}
\DeclareMathOperator{\NB}{NB}

\newtheorem{theorem}{Theorem}
\newtheorem{proposition}[theorem]{Proposition}
\newtheorem{lem}[theorem]{Lemma}
\newtheorem{cor}[theorem]{Corollary}

\newcommand{\RNA}{\textsc{rna}}
\newcommand{\rna}{\textsc{rna}}
\newcommand{\SRP}{\textsc{SRP}}
\newcommand{\RNaseP}{\textsc{RN}ase\textsc{P}}

\newcommand{\A}{\textsc{a}}
\newcommand{\C}{\textsc{c}}
\newcommand{\G}{\textsc{g}}
\newcommand{\U}{\textsc{u}}

\newcommand{\DEG}{{\tt DEG}}
\newcommand{\UNP}{{\tt UNP}}
\newcommand{\ETE}{{\tt ETE}}
\newcommand{\COV}{{\tt CHN}}
\newcommand{\LEN}{{\tt LEN}}
\newcommand{\FirstHelix}{{\tt HEL}}
\newcommand{\HEL}{{\tt HEL}}
\newcommand{\STM}{{\tt STM}}

\newcommand{\keywords}[1]{%
  \par\medskip\noindent\textbf{Keywords:} #1\par\medskip
}

\title{Making ends meet or just meeting at the ends?  Assessing end-to-end distance in folded RNA sequences and other branched structures}

\author[1]{Torin Greenwood}
\author[2]{Christine Heitsch}

\affil[1]{Department of Mathematics, North Dakota State University, Fargo, ND 58102, USA}
\affil[2]{School of Mathematics, Georgia Institute of Technology, Atlanta, GA 30332, USA}

\date{}

\usepackage[backend=biber,style=numeric]{biblatex}
\begin{document}

\maketitle

\begin{abstract}
Researchers have repeatedly found that the ends of an \textsc{rna} sequence are significantly closer than expected for a random linear chain.  However, we prove that the ends of a branched structure are almost certainly close.  Our results are obtained via combinatorial branching models of increasing complexity using tools from multivariate analytic combinatorics.  We completely characterize parameters tracking end-to-end distance, including means and variances.  Then, we compare to existing datasets of known \rna\ structures, as well as the minimum free-energy structures of randomized shuffles.  We find that the shuffled structures resemble our theoretical distributions while the known \rna\ structures have similar parameter values but are more concentrated.
\end{abstract}

\keywords{Dyck, Motzkin, Pfold stochastic context-free grammar, analytic combinatorics, RNA, secondary structures}

\textsc{Rna} molecules are not merely passive carriers of genetic information but play direct roles in many of the cell’s essential processes. The function of an \RNA\ molecule depends critically on the three-dimensional conformation induced by intramolecular base pairing.  These interactions generate characteristic structural motifs that govern stability, specificity, and interaction with proteins and other nucleic acids.  Identifying the three-dimensional structure can be technically challenging, time-consuming, and often infeasible.  Instead, it is more tractable to identify a two-dimensional approximation called the \emph{secondary structure}, which still captures essential functional information.  Secondary structures are a special type of branching structure, with several interpretations defined below.  

Recent papers in \rna\ folding have concluded that the ends of \rna\ molecules often end up close to one another.  Here, we ask: how common is this for a random branched structure?  To address this problem with mathematical rigor, we consider three increasingly complex distributions over branched structures and five different parameters measuring \emph{end-proximity}.  By using an analysis pipeline from analytic combinatorics, we find complete descriptions of the limiting distributions of each parameter.  We find uniformly small means and variances for each parameter, indicating that the endpoints of branched structures are ``necessarily'' close, rather than this being a feature unique to thermodynamically favorable \rna\ secondary structures~\cite{Yoffe:2010}.

Our contribution to the increasing body of work on end-proximity has three main advancements: first, we refine the analytic combinatorics pipeline to give limiting probability generating functions (PGFs) describing the distributions of each parameter.  Previous work focused on the mean of each parameter, but our analysis allows us to identify variances, which illustrate a key difference between random branched structures and known \RNA\ families.  Next, we are able to analyze the joint distribution of multiple measures of end-proximity simultaneously.  Besides being of interest in its own right, the joint distribution unlocks our ability to analyze more refined definitions of end-to-end distance~\cite{Aalberts:2010}.  Finally, we can study the distribution of branched structures output by the Pfold grammar~\cite{Knudsen:2003}.  This grammar has been shown to predict \RNA\ structures with high accuracy.

In \Cref{sec:Overview}, we give an initial outline of the branching models considered as well as our main results.  Then, in \Cref{sec:Analytic}, we describe the primary tool we develop in analytic combinatorics to extract distributions.  Following this is our mathematical analysis of each characteristic of end-proximity in \Cref{sec:ModelAnalysis}.  To investigate the pairing propensity of the ends of a sequence, we examine the lengths of the first helix and first stem of a structure in \Cref{sec:Helices}.  In \Cref{sec:Data}, we compare our theoretical distributions to known \RNA\ structures from curated datasets. Finally, in \Cref{sec:Conclusion}, we indicate implications to combinatorics and directions for future work.

\section{The distributions of end-proximity} \label{sec:Overview}

The one-dimensional structure of an \RNA\ molecule, known as its sequence,
is an oriented linear polymer composed of four nucleotides: 
\A(denine), \C(ytosine), \G(uanine), and \U(racil).
The links of this chain are formed by covalent bonds which connect the 
3'-hydroxyl group of one sugar to the 5'-phosphate group of the next 
nucleotide.
This creates a strong 5'-3' linear polymer with an alternating 
sugar-phosphate backbone and nitrogenous bases available for pairing.

The intra-sequence base pairing of an \RNA\ molecule creates  
a two-dimensional branched structure, known as its (pseudoknot-free)
secondary structure\footnote{
Two base pairs $(i,j)$ and $(k,l)$ with $i < k$ form a pseudoknot 
if $i < k < j < l$.
Although such crossed pairings are known to exist, 
and to be functionally important,
their prediction remains a challenging aspect of \RNA\ pairwise interactions.
Hence, they are typically considered part of the tertiary interactions
that, along with base triples, coaxial stacking, etc., ultimately 
determine the three dimensional molecular conformation.
}.
Hydrogen bonds create a bridge between complementary nucleotides;
the familiar Watson-Crick-Franklin pairings of \C\G/\G\C\ and \A\U/\U\A\ 
along with the wobble \G\U/\U\G\ one are known as the canonical 
base pairs\footnote{
Although non-canonical pairings are known to exist, they are 
treated as a small internal loop under the most common secondary 
structure prediction methods which are based on the nearest
neighbor thermodynamic model (NNTM)~\cite{Turner:2009}.
}.
An \RNA\ secondary structure decomposes into runs of consecutive 
base pairs, known as helices, and single-stranded regions called 
loops.
Our focus here is the \emph{exterior loop} which is the 
set of nucleotides (paired and unpaired) not contained within any 
base pair.

Specifically, we consider the ``end-to-end'' (\ETE) distance in \RNA\ 
molecules as the shortest path from the 5' to 3' ends.
In a secondary structure, this relates to one less than
the number of nucleotides in the exterior loop.
However, to compare with experimental measurements, such as 
force microscopy~\cite{Gerland:2001} or 
FRET~\cite{Leija-Martinez:2014,Lai:2018}, 
this discrete count must be 
translated into an estimate of the physical distance.
Hence, it is important to distinguish between the covalent links 
in the sugar-phosphate backbone and the hydrogen bridges created by
base pairing as the latter are $\sim2.5$ times longer than the former.

Our method allows us to explicitly consider the two different
kinds of steps, and to fully characterize their individual and 
joint distributions.
In particular, we can compute arbitrary moments and hence
not only give the asymptotic means, as others have done,
but also their variances.

\subsection{Overview of main results}

\label{sec:Models}
\begin{figure}
    \begin{center}
        \includegraphics[width=\textwidth]{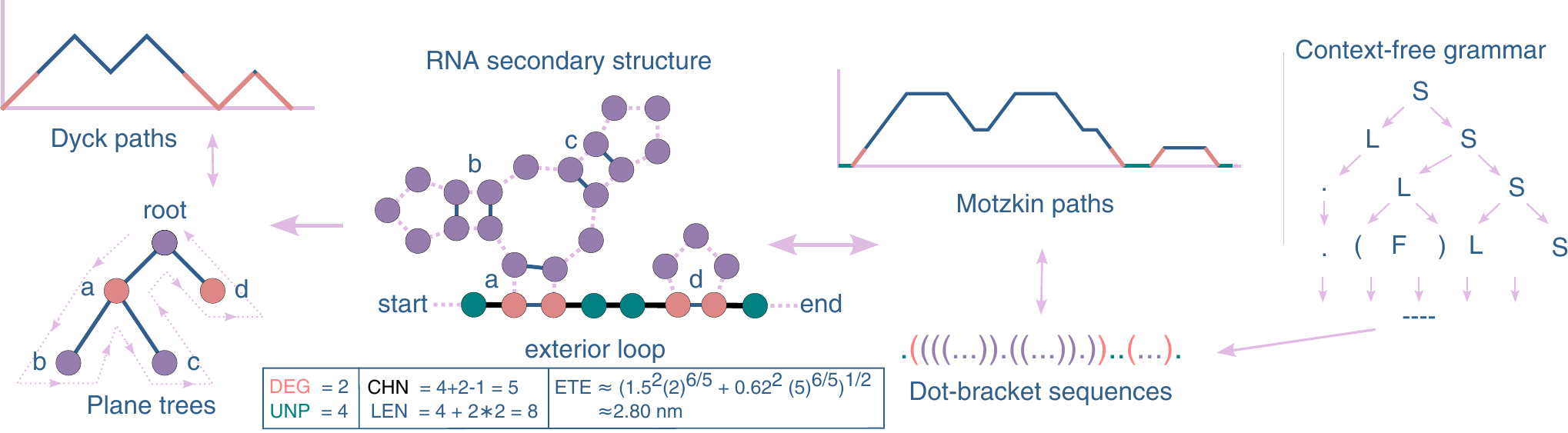}
    \end{center}
    \caption{%
An \RNA\ secondary structure, as in the radial diagram (center),
consists of runs of stacked base pairs, called helices, separated by
singled-stranded regions called loops.
Since no pairings cross, the arrangement of helices can be 
abstracted to a plane tree, which is equivalent 
to a Dyck path under a pre-order walk.
A secondary structure is written as a dot-bracket sequence (equivalently a Motzkin path) by reading from left (5') to right (3')
an unpaired base as a horizontal step (dot), 
the 5' side of a base pair as an up step (left parenthesis) and 
the 3' side as a down one (right parenthesis).
A stochastic context-free grammar, such as the one in \Cref{sec:Pfold},
generates dot-bracket sequences
under a weighting determined by the probability parameters.
In this example, 
there are 8 nucleotides in the exterior loop, with 2 base pairs and 
4 unpaired nucleotides which gives an \ETE\ distance estimate of 2.80 nm.
    }
    \label{fig:RNAReps}
\end{figure}
 
Consider a strand of $n$ ordered nodes in a line.  An \textsc{rna} secondary structure is a partial pairing of these nodes with arcs above the nodes such that no two arcs cross.  Equivalently, \RNA\ structures can be represented as dot-bracket sequences (\Cref{fig:RNAReps}), meaning that they form a \emph{language} $\mathcal{L}$ on the alphabet $\{(, ), .\}$.  We consider distributions on $\mathcal{L}$ of varying complexity.  When examining increasingly abstract distributions, we find stability in the behavior of end-proximity.  This illustrates that end-proximity is a feature of branching structures globally, and not a biological novelty of \RNA.

As the most rigorous distribution considered, the Pfold
grammar~\cite{Knudsen:2003} has been used to predict 
\RNA\ secondary structures with high accuracy.  
This stochastic context-free grammar (defined in \Cref{sec:Pfold}) 
has only three production rules, each with one probability parameter. 
Nonetheless, it captures biologically meaningful information, such as 
the stacking of base pairs and the spacing of unpaired nucleotides, in a more mathematically
tractable way than the standard thermodynamic recursions
used in most secondary structure prediction methods.

We obtain a simpler distribution on $\mathcal{L}$ if we
weight all dot-bracket sequences (or Motzkin paths) equally.
One further relaxation is if we omit unpaired bases and consider only runs of paired bases, which leads to the uniform distribution on balanced parenthesizations, equivalent to Dyck paths.  While both distributions have been used elsewhere to analyze \RNA, the ubiquity of Catalan objects like Dyck paths implies that any characteristics of end-proximity from this model are not due to biological specificities.

We consider four different end-proximity characteristics, 
summarized in Table~\ref{table:Parameters}, of an exterior loop.
The number of base pairs is the degree, \DEG,
equivalently the root degree of a plane tree.  
It is also the number of mountains in a Motzkin path while
the number of unpaired bases \UNP\ is the 
number of horizontal steps along the horizon of the path.
The number of covalent bonds \COV\ along the 5' to 3' shortest path is 
related as $\COV = \DEG + \UNP - 1$ whereas the total number of nucleotides,
denoted \LEN, is $2\DEG + \UNP$.
Hence $\LEN - 1 = \DEG + \COV$.

We also consider an approximation of the physical end-to-end distance \ETE\
based on a two-scale self-avoiding freely jointed chain 
estimate~\cite{Aalberts:2010, Lai:2018} 
from polymer theory given by 
$(b^2 \cdot \DEG^{6/5} + c^2 \cdot \COV^{6/5})^{1/2}$
with a distance of $b = 1.5$ nm for hydrogen bridges 
and $c = 6.2$ nm for covalent links.
We note that this is strictly larger than the root-mean-square estimate
of $a \sqrt{\LEN - 1}$ originally considered~\cite{Yoffe:2010}
based on the path length with an average step size of $a = 0.75$ nm.

\begin{table}[h]
\centering
\begin{tabular}{cl}
\toprule
Parameter & Definition\\
\cmidrule(l){1-2}
\DEG & Degree of the exterior loop; number of hydrogen bridges \\
\UNP & Number of unpaired bases in exterior loop\\
\COV & Number of covalent bonds in exterior loop; $\DEG + \UNP - 1$\\
\LEN & Number of nucleotides in exterior loop; $2\DEG + \UNP$\\
\ETE & Approximate end-to-end distance in nanometers; $\sqrt{1.5^2 \cdot \DEG^{6/5} + 0.62^2 \cdot \COV^{6/5}}$\\
\bottomrule
\end{tabular}
\caption{Parameters characterizing end-proximity of the exterior loop.}
\label{table:Parameters}
\end{table}

\begin{proposition}
    As the length of a dot-bracket sequence approaches infinity, the limiting distributions of end-proximity are as described in \Cref{table:MainProp}.
\end{proposition}

\begin{table}[h]
\centering
\renewcommand{\arraystretch}{1.3}
\setlength{\tabcolsep}{8pt}
\begin{tabular}{|l|c|c|c|}
\hline
\textbf{Characteristic} & \textbf{Distribution or PGF} & \textbf{Mean} & \textbf{Variance} \\
\hline
\DEG\ Dyck paths & \( 1 + \NB(2, \tfrac{1}{2}) \) & 3 & 4 \\
\hline
\DEG\ Motzkin paths & \( 1 + \NB(2, \tfrac{1}{2}) \) & 3 & 4 \\
\hline
\DEG\ Pfold output & \( 1 + \NB(2, (1 + \delta)^{-1}) \) & \( 1 + 2\delta \) & \( 2\delta(1 + \delta) \) \\
\hline
\DEG\ Pfold default $\delta$ & \( 1 + \NB(2, 0.56) \) & 2.55 & 2.76 \\
\hline\hline
\UNP\ Motzkin paths & $\NB(2, 1/2)$ & 2 & 4\\
\hline
\UNP\ Pfold output & $\NB(2, \tfrac{1 - \delta}{1 + \delta^2})$ & $2\tfrac{\delta(1 + \delta)}{1 - \delta}$ & $\tfrac{2\delta(1+\delta)(1+\delta^2)}{(1-\delta)^2}$\\
\hline
\UNP\ Pfold default $\delta$ & $\NB(2, 0.14)$ & 12.39 & 89.19\\
\hline\hline
\COV\ Motzkin paths & $ \NB(2, \tfrac{1}{3}) $ & 4 & 12 \\
\hline
\COV\ Pfold output & \( \NB\!\left(2, \tfrac{1-\delta}{1 + \delta}\right) \) & \( \tfrac{4\delta}{1 - \delta} \) & \( \tfrac{3\delta(1 + \delta)}{(1 - \delta)^2} \) \\
\hline
\COV\ Pfold default $\delta$ & \( \NB\!\left(2, 0.13\right) \) & 13.95 & 111.21 \\
\hline\hline
\texttt{LEN} Motzkin paths & \rule{0pt}{3.2ex}\( \left( \tfrac{u}{u^2 + u - 3}\right)^2 \) & 8 & 28 \\
\hline
\texttt{LEN} Pfold output & \rule{0pt}{3.2ex}\( \left( \tfrac{(1-\delta)u}{(1 + \delta) - \delta(1 + \delta)u - \delta(1-\delta)u^2}\right)^2 \) & \( 2 + 2\delta + \tfrac{4\delta}{1-\delta} \) & \( \tfrac{2\delta(\delta^3 - 3\delta^2 + \delta + 5)}{(1-\delta)^2} \) \\
\hline
\texttt{LEN} Pfold default $\delta$ & \rule{0pt}{3.2ex}\( \left( \tfrac{0.05u^2}{1.78 - 1.38u - 0.17u^2}\right) \) & 17.50 & 138.76 \\
\hline
\hline
\ETE\ Dyck paths & See \Cref{cor:DyckETE} & 2.89 & 1.42\\
\hline
\ETE\ Motzkin paths & See \Cref{cor:Motzkin} & 3.08 & 1.56\\
\hline
\ETE\ Pfold default $\delta$ & See Proposition~\ref{prop:Pfold} & 3.83 & 2.23\\
\hline
\hline
\HEL\ Dyck paths & $1 + \NB(1, 3/4)$ & $4/3$ & $4/9$\\
\hline
\HEL\ Motzkin paths & $1 + \NB(1, 8/9)$ & $9/8$ & $9/64$\\
\hline
\HEL\ Pfold output & \rule{0pt}{3.0ex}\rule[-1.7ex]{0pt}{0pt}\(1 + \NB(1, 1 - \rho^2p_3)\) &  \(\tfrac{1}{1 - \rho^2 p_3}\) & \(\tfrac{\rho^2 p_3}{(1 - \rho^2 p_3)^2}\)\\
\hline
\HEL\ Pfold default $\delta$ & $1 + \NB(1, 0.21)$ & 4.71 & 17.51\\
\hline
\end{tabular}
\caption{Limiting distributions for the five end-proximity characteristics
from Table~\ref{table:Parameters} under the different 
secondary structure models considered.
The Pfold distribution depends on $\delta$ which is determined by 
the grammar probability parameters; see Proposition~\ref{prop:Pfold}.}
\label{table:MainProp}
\end{table}

The first observation is the recurrence of a negative binomial 
distribution $\NB(r, p)$.
This is often described in terms of coin flipping, where the probability
of one head is $p$ and $\NB(r, p)$ models the number of tails (i.e.\ failures)
obtained before $r$ heads.
Its probability generating function (PGF) is $h(u) := [p/(1 - (1-p)u)]^r$ which 
has mean $r(1-p)/p$ with variance $r(1-p)/p^2$. 
The coefficient of $u^k$ in the Taylor series expansion (denoted $[u^k] h(u)$) 
is the probability that $k$ tails were flipped before obtaining $r$ heads.  
With $r = 2$ and a fair coin, the discrete distribution begins as 
$(0.25, 0.25, 0.1875, 0.125, 0.078125, \ldots)$.
In contrast to a Gaussian distribution,
89\% is within one standard deviation of the mean.
Hence, we expect a random Dyck path to have degree 1 or 2 half 
of the time. 

When considering the degree of the exterior loop, the distribution in 
Dyck and Motzkin paths is the same.
This can be explained by projecting Motzkin paths down to Dyck paths by removing all horizontal steps.  When the size of the Motzkin path and Dyck path are known in advance, the projection is uniform, meaning that any Dyck path is the projection for the same number of Motzkin paths.  Another instance of this phenomenon appears in Proposition~\ref{prop:DyckMotzkinDepth} below.

The degree under Pfold is distributed in the same way as the Dyck and Motzkin paths,
except that the probability of a success now depends on $\delta$.
Under the default parameters this increases slightly, shifting more
probability to lower degrees.
For instance, the probability that $\DEG = 1$ or $2$ increases from 0.5
under the uniform distribution to 0.5896 with all subsequent degrees
being lower.
This is reflected in both a lower mean and lower variance, and also
shifts $1 \leq \DEG \leq 5$ from 0.8906 to 0.9373.

In contrast, the probability of a `failure' (i.e.\ adding a base) 
increases substantially under Pfold when considering the unpaired bases 
whereas it is the same as a base pair under the uniform weighting.
As a consequence, this shifts the balance from paired to unpaired
in determining the \LEN\ distribution.

However, the \ETE\ approximation is largely determined by \DEG.
This is apparent in the small difference between the Dyck and Motzkin
statistics.
Moreover, although the Pfold \LEN\ average is more than twice that
of the Motzkin paths, with a substantially higher variance,
the difference in \ETE\ means is $< 1$ nm as is the variance.
Consequently, under Pfold, the most biologically realistic model considered,
the theory predicts that a random sequence which folds into a 
branched structure via paired positions would typically have an \ETE\ 
distance of about 4 nm with a standard deviation of 1.5 nm.

\subsection{Comparison with related modeling results}

Our primary goal is to provide a theoretical context for recent 
experimental results on messenger \RNA\ (m\RNA) and long noncoding \RNA\ 
(lnc\RNA) molecules~\cite{Lai:2018}.
They found that average \ETE\ distance,  
as measured by F\"orster resonance energy transfer (FRET) experiments,
was always in the range of 5--7 nm for the 10 sequences considered.
As we will discuss later,
this is not inconsistent with our models, and well below the 
expectation for unpaired polymers.

The fact that experimental results consistently indicate that
the ends of long \RNA\ molecules are ``close'' always attracts 
attention.
It is viewed as ``counterintuitive,''
``surprising,'' and even ``rather remarkable''~\cite{Ermolenko:2020}
when compared to the vanishingly small probability that the ends of an
arbitrary linear chain are near each other~\cite{Yoffe:2010}.
However, when the pairing of the chain is taken into account, 
the expected distance drops dramatically.

The first paper to explicitly consider the \ETE\ question in this 
form~\cite{Yoffe:2010} gave a theoretical argument for the end-proximity
of a randomly self-paired polymer.
Their formulation depended on the percentage
of paired nucleotides and the average helix length, which were estimated
from computational minimum free energy (MFE) predictions 
to be $\sim60\%$ and $\sim4$ respectively.
They also considered $\LEN - 1$, and calculated that it should be 
$\sim12$. 
Under a root-mean-square (RMS) \ETE\ distance estimate of $a \sqrt{\LEN - 1}$
with an estimated average step size of $a = 3/4$ nm, they concluded that 
the expected \ETE\ was 3 nm.
This is almost exactly what our simpler Dyck and Motzkin path models predict
under the two-scale distance estimate, and what our Pfold model predicts
under theirs.

Subsequently~\cite{Clote:2011}, the $\LEN - 1$ problem was considered 
based on the standard recursive formulation for \RNA\ secondary structures.

Using generating functions and analytic combinatorics, as is done here,
they concluded that the mean should be lower,
i.e.\ about 6 (depending on parameters; see below).
Using the largest mean exterior loop length reported in the 
average step distance estimate yields an \ETE\ of 2.11 nm. 
 
Although not formulated in this way, prior work~\cite{Gerland:2001} 
had recursively computed the size of the exterior loop to model 
force microscopy experiments.
They give an explicit physical distance estimate, but as a 
function of the forces pulling on a specific \RNA\ sequence, so it 
is not comparable with our results.
Others~\cite{Hofacker:1998}, as part of a comprehensive theoretical 
analysis of motifs in \RNA\ secondary structures,
used generating functions to estimate the mean number of 
base pairs and also of unpaired bases in the exterior loop.
Again, using the largest means reported, the two-state distance
estimate yields an \ETE\ of 2.71 nm.

Parameters are an important factor in these analyses.
A common one in theoretical models is the minimum hairpin length.
This is negatively correlated with \LEN, although the differences
are small~\cite{Clote:2011}, i.e.\ about a nucleotide. 
In the models we consider, 
Pfold has a minimum of length 1 but the Motzkin paths could be 0.
Another factor is that only some nucleotides can pair in a real
\RNA\ sequence.
Other analyses~\cite{Hofacker:1998, Clote:2011} have considered 
a ``stickiness'' factor which models the pairing propensity, 
which again is negatively correlated with $\LEN = 2 \DEG + \UNP$
and again the differences are small.

Finally, we note that
others have also analyzed \LEN\ theoretically~\cite{Han:2012}
using the methodology of generating functions and 
analytic combinatorics,
or shown how to efficiently compute it for a given 
sequence~\cite{Mori:2014} under the NNTM recursions.
Additionally, if crossed pairings/pseudoknots are included,
the calculation changes considerably~\cite{Fang:2011}.
In particular, the growth rate of \LEN\ with sequence length 
is no longer negligible.
As we will discuss, the effect of pseudoknots in the exterior loop
(as seen in 16S ribosomal \RNA\ in \Cref{sec:Data})
may well be to extend the \ETE\ distance beyond the expected small amount.

\section{Extracting asymptotic distributions} \label{sec:Analytic}

Here, we lay out our main tool from analytic combinatorics that allows us to extract limiting distributions of parameters from generating functions.  Our result allows us to track the multivariate distribution of several parameters simultaneously.  We view $z$ as the variable associated to the size of our objects, and $\mathbf{u}$ as the array of variables associated to the parameters. 

We define $\mathbf{u} = (u_1, u_2, \ldots, u_d)$ and $\mathbf{u}^\mathbf{k} = u_1^{k_1} \cdots u_d^{k_d}$.  We also define $[z^n \mathbf{u}^{\mathbf{k}}]F(z, u)$ to be the coefficient of $z^n \mathbf{u}^{\mathbf{k}}$ in the series expansion of $F(z, u)$.  Additionally, we let $\mathbf{0}$ and $\mathbf{1}$ be the vectors of all zeroes and all ones, respectively.

In the context of \rna, $z$ will track the length of an \rna\ sequence, while $\mathbf{u}$ will track the degree of the exterior loop and number of unpaired bases in the exterior loop. For example, if $F(z, \mathbf{u})$ enumerates \rna\ secondary structures under some model, then $[z^{100} u_1^{1} u_2^{6}] F(z, \mathbf{u})$ would correspond to the number of \rna\ secondary structures of length $100$ with an exterior loop of degree $1$ containing $6$ unpaired bases.

\begin{proposition} \label{prop:GF}
Let $F(z, \mathbf{u})$ be a multivariate GF where $[z^n\mathbf{u}^\mathbf{k}]F(z, \mathbf{u})$ counts the number of objects of size $n$ with parameters equal to $\mathbf{k} = (k_1, k_2, \ldots, k_d)$.  Suppose $F(z, \mathbf{u})$ is of the form
   
    \[
    F(z, \mathbf{u}) = \frac{-b(z, \mathbf{u}) \pm g(z, \mathbf{u}) \sqrt{c(z)}}{2a(z, \mathbf{u})}
    \]
    with everywhere analytic functions $a, b, c,$ and $g$. Define $p_{n, \mathbf{k}}$ to be the probability that an object of size $n$ has parameters equal to $\mathbf{k}$ as below:
    \[
    p_{n, \mathbf{k}} = \frac{[z^n\mathbf{u}^\mathbf{k}] F(z, \mathbf{u})}{[z^n]F(z, \mathbf{1})}.
    \]
    Additionally, suppose there exists a $\rho > 0$ such that $c(\rho) = 0$ and $c(z) \neq 0$ for any other value of $z$ with $|z| \leq \rho$.  Assume that $c'(\rho) \neq 0$, and that for $\mathbf{u} \neq 0$ in a neighborhood of zero: $g(\rho, \mathbf{u}) \neq 0, a(\rho, \mathbf{u}) \neq 0$, and there are no non-removable singularities of $F(z, \mathbf{u})$ with $|z| \leq \rho$ besides $z = \rho$.  Also, assume $F(z, \mathbf{1})$ has no non-removable singularities with $|z| \leq \rho$ besides $z = \rho$. Then,
    \[
    \lim_{n \to \infty} p_{n, \mathbf{k}} = [\mathbf{u}^\mathbf{k}] \left(\frac{g(\rho, \mathbf{u})}{g(\rho, \mathbf{1})} \cdot \frac{a(\rho, \mathbf{1})}{a(\rho, \mathbf{u})}\right).
    \]
\end{proposition}

\begin{proof}
    For any fixed $\mathbf{u} \neq 0$ with $|\mathbf{u}|$ sufficiently small, we examine the singular expansion of $F(z, \mathbf{u})$ near its closest non-removable singularity to the origin, which by assumption is $z = \rho$.  We can expand each part of $F$ separately, obtaining as $z \to \rho$:
    \[
    -\frac{b(z, \mathbf{u})}{2a(z, \mathbf{u})} = -\frac{b(\rho, \mathbf{u})}{2a(\rho, \mathbf{u})} + O\left(1 - \frac{z}{\rho}\right)
    \]
    and
    \[
    \frac{g(z, \mathbf{u})\sqrt{c(z)}}{2a(z, \mathbf{u})} = \frac{g(\rho, \mathbf{u}) \sqrt{-\rho c'(\rho)}}{2a(\rho, \mathbf{u})} \sqrt{1 - z/\rho} + O(1 - z/\rho)^{3/2}.
    \]
    Overall, for fixed $\mathbf{u} \neq 0$ sufficiently small as $z \to \rho$,
    \[
    F(z, \mathbf{u}) = -\frac{b(\rho, \mathbf{u})}{2a(\rho, \mathbf{u})} \pm \frac{g(\rho, \mathbf{u})\sqrt{-\rho c'(\rho)}}{2a(\rho, \mathbf{u})}\sqrt{1 - z/\rho} + O(1 - z/\rho).
    \]
    By standard transfer theorems of analytic combinatorics (c.f.~\cite[Corollary VI.1]{FlSe:2009}),
    \[
    [z^n] F(z, \mathbf{u}) \sim \frac{g(\rho, \mathbf{u})\sqrt{-\rho c'(\rho)}}{2a(\rho, \mathbf{u})} \cdot \frac{\rho^{-n}n^{-3/2}}{\Gamma(-1/2)}
    \]
    and
    \[
    [z^n] F(z, \mathbf{1}) \sim \frac{g(\rho, \mathbf{1}) \sqrt{-\rho c'(\rho)}}{2a(\rho, \mathbf{1})} \cdot \frac{\rho^{-n}n^{-3/2}}{\Gamma(-1/2)}.
    \]
    Dividing the two gives that pointwise for each $\mathbf{u} \neq 0$ sufficiently small,
    \[
    \lim_{n \to \infty} \frac{[z^n] F(z, \mathbf{u})}{[z^n] F(z, \mathbf{1})} = \frac{g(\rho, \mathbf{u})}{g(\rho, \mathbf{1})} \cdot \frac{a(\rho, \mathbf{1})}{a(\rho, \mathbf{u})}.
    \]
    Because this holds for a collection of small $\mathbf{u}$ values including an accumulation point within the $d$-dimensional unit disk, the Continuity Theorem~\cite[Theorem IX.1]{FlSe:2009} completes the proof.
\end{proof}

\section{End-proximity in the exterior loop} \label{sec:ModelAnalysis}

We now apply Proposition~\ref{prop:GF} to each of our branching models in order of increasing complexity.  Our calculations also appear in a supplementary {\tt SageMath} worksheet available here:
\begin{center}
    \url{https://github.com/gtDMMB/EndProximity}
\end{center}

\subsection{End-proximity in plane trees/Dyck paths} \label{sec:Plane}

Dyck paths serve as our most abstract model on $\mathcal{L}$ where the only measure of end-proximity is the degree, {\tt DEG}.  The following lemma is a recapitulation of the result in~\cite[Example IX.8]{FlSe:2009}, but provides a template for using Proposition~\ref{prop:GF} that we will mimic later.

\begin{lem} \label{lem:DyckDegree}
Consider the uniform distribution over all Dyck paths as $n \to \infty$.  ${\tt DEG}$ tends in distribution to $1 + \NB(2, 1/2)$, which has mean $3$ and variance $4$.
\end{lem}

\begin{proof}
    Let $D(z, u)$ be the GF where $[z^n u^k] D(z, u)$ is 
the number of Dyck paths of length $2n$ with $k$ mountains.  Then, by the standard recurrence for Dyck paths (see \Cref{fig:Dyck}),
    \[
    D(z, u) = 1 + z (u \cdot D(z, 1)) \cdot D(z, u)
    \]
where $D(z, 1) = 1 + zD^2(z,1)$ is the standard univariate GF.
Solving for $D(z, u)$ yields
    \[
    D(z, u) = \frac{1}{1 - zuD(z, 1)} = \frac{2}{2 - u + u\sqrt{1 - 4z}} = \frac{2((u - 2) - (u\sqrt{1 - 4z}))}{(2 - u)^2 - u^2(1-4z)},
    \]
   which meets the conditions of Proposition~\ref{prop:GF} for $u$ sufficiently small and $|z| \leq 1/4$.  
Since the discriminant depends only on $z$, its contribution to $2a(z,u)$ 
disappears at the singularity $z = 1/4$.  
Thus, the distribution of peaks approaches the one encoded by 
    \[
    u \cdot \frac{a(1/4, 1)}{a(1/4, u)} = \frac{u}{(2-u)^2}.
    \]
This corresponds to the sum of two independent 
geometric random variables, each with parameter $1/2$, shifted by 
one position.
\end{proof}

The limiting distribution provides complete asymptotic information about the probability of a Dyck path having $\DEG = k$ for any $k$.  With this, we can compute the moments of $\ETE$ to arbitrary precision.  For any Dyck path, we interpret $\UNP = 0$.

\begin{cor}\label{cor:DyckETE}
    For the uniform distribution over all Dyck paths, the limiting distribution of \ETE\ has mean approximately $2.89$ and variance approximately $1.42$.
\end{cor}

\begin{proof}
    We can find the probabilities in the limiting distribution of \DEG\ by expanding the GF in \Cref{lem:DyckDegree}.  To expand, note that the function $1/(1-x)^2$ is the derivative of $1/(1-x)$, and hence the series for $1/(1-x)^2$ is the derivative of a geometric series.  We obtain:
    \[
    \frac{u}{(2-u)^2} = \frac{1}{2} \frac{\frac{u}{2}}{\left(1 - \frac{u}{2}\right)^2} = \sum_{n = 0}^\infty \frac{n}{2^{n + 1}} u^n.
    \]
    By the definition of a PGF, this implies that in the limiting distribution,
    \[
    \mathbb{P}(\DEG = n) = \frac{n}{2^{n + 1}}.
    \]
    Note that this could also have been derived by using the definition of the negative binomial, but these sorts of expansions are useful for the more general bivariate distributions below.

    Since we interpret $\UNP = 0$ for Dyck paths, our distance formula becomes
    \[
    \ETE = \sqrt{(1.5)^2 \DEG^{6/5} + (0.62)^2 (\DEG - 1)^{6/5}}
    \]
    Thus,
    \[
    \mathbb{E}(\ETE) = \sum_{n = 0}^\infty \frac{n}{2^{n + 1}}\sqrt{(1.5)^2 n^{6/5} + (0.62)^2 (n - 1)^{6/5}}
    \]
    To compute this sum to within $0.001$, we must bound the tails of the sum.  Note:
    \[
    T_k := \sum_{n = k}^\infty \frac{n}{2^{n + 1}}\sqrt{(1.5)^2 n^{6/5} + (0.62)^2 (n - 1)^{6/5}} \leq \sum_{n = k}^\infty \frac{n}{2^{n + 1}} \cdot 1.63 n^{3/5} \leq \frac{1.63}{2}\sum_{n = k}^\infty \frac{n^2}{2^n},
    \]
    where we bounded the term under the square root by replacing $n - 1$ with $n$ and combining the two terms.

    It turns out that the latter sum may be computed explicitly.  To do so, we first need some well-known formulas for geometric series and their derivatives when $|x| < 1$:
    \begin{equation}\label{eq:geometric}
        \frac{1}{1 - x} = \sum_{n = 0}^\infty x^n, \ \ \frac{x}{(1-x)^2} = \sum_{n = 0}^\infty nx^n, \ \ \frac{x(1 + x)}{(1-x)^3} = \sum_{n = 0}^\infty n^2 x^n
    \end{equation}
    Replacing $x$ with $1/2$ yields for $k \geq 6$:
    \begin{align*}
    \sum_{n = k}^\infty \frac{n^2}{2^n} &= \frac{1}{2^k}\sum_{n = 0}^\infty \frac{(n + k)^2}{2^n} = \frac{1}{2^k} \left(\sum_{n = 0}^\infty \frac{n^2}{2^n} + 2k \sum_{n = 0}^\infty \frac{n}{2^n} + k^2 \sum_{n = 0}^\infty \frac{1}{2^n}\right)\\
    &= \frac{1}{2^k} \left(6 + 4k + 2k^2 \right) \leq \frac{3k^2}{2^{k}},
    \end{align*}
    where the $k \geq 6$ condition was used only in the last inequality.  Hence, $T_k \leq 1.63/2 \cdot 3k^2/2^{k}$, and the right side is a decreasing function for $k \geq 3$.  That means that for all $k \geq 20$, $T_k \leq 0.001$.  Hence, it suffices to approximate $\mathbb{E}(\ETE)$ to within two decimal places by computing
    \[
    \sum_{n = 0}^{19} \frac{n}{2^{n + 1}}\sqrt{(1.5)^2 n^{6/5} + (0.62)^2 (n - 1)^{6/5}},
    \]
    the value of which is $2.893$.  Similarly, when computing $\mathbb{E}(\ETE)^2$, the tail is now
    \[
    T_k = \sum_{n = k}^\infty \frac{n}{2^{n + 1}} \left((1.5)^2 n^{6/5} + (0.62)^2 (n - 1)^{6/5}\right) \leq \frac{2.64}{2} \sum_{n = k}^\infty \frac{n^3}{2^n}
    \]
    Using the same steps with \Cref{eq:geometric} and the additional relation $x(1 + 4x + x^2)/(1-x)^4 = \sum_{n = 0}^\infty n^3x^n$, we obtain for $k \geq 9$:
    \[
    T_k \leq \frac{2.64}{2} \cdot \frac{1}{2^k} \cdot\left(2k^3 + 6k^2 + 18k + 26 \right) \leq \frac{2.64}{2} \cdot \frac{3k^3}{2^{k}},
    \]
    where the rightmost function is decreasing for $k \geq 5$.  Hence, we find that the tail is smaller than $0.001$ for $k \geq 27$, meaning we can estimate $\mathbb{E}(\ETE)^2$ to two decimal places with the sum,
    \[
    \sum_{n = 0}^{26} \frac{n}{2^{n + 1}} \left((1.5)^2 n^{6/5} + (0.62)^2 (n - 1)^{6/5}\right),
    \]
    and we obtain $\mathbb{E}(\ETE)^2 \approx 9.795$.  Then, the variance can be computed with the shortcut formula Var$(\ETE) = \mathbb{E}(\ETE)^2 - (\mathbb{E}(\ETE))^2 \approx 1.42$, where we note that the precision of $\mathbb{E}(\ETE)$ needs to be increased in order to guarantee that $(\mathbb{E}(\ETE))^2$ has error at most $0.001$.
    
\end{proof}

\subsection{End-proximity in dot-bracket sequences/Motzkin paths} \label{sec:Motzkin}
Dot-bracket sequences provide a more nuanced model of branching.
The distributional results for the sum {\tt UNP + DEG} appeared with a different proof method in~\cite[Theorem 3.1]{Banderier:2018}.

\begin{proposition} \label{lem:MotzkinPaths}
Consider the uniform distribution over all Motzkin paths of length $n$ as $n \to \infty$.  The limiting PGF encoding \UNP\ $(u)$ and \DEG\ $(v)$ as the length of the sequence tends towards infinity is given by
\[
M_p(u, v) = \frac{v}{(u + v - 3)^2}
\]
corresponding to a joint negative binomial distribution. 
\end{proposition}

\begin{proof}
    Let $M(z, u, v)$ be the GF tracking Motzkin paths according to length $(z)$, \UNP\ $(u)$, and \DEG\ $(v)$.  The univariate GF $M(z, 1, 1)$ can be specified using the symbolic method as follows:
        \begin{equation*} \label{eq:MotzkinSpecification}
    \mathcal{M} = \mbox{SEQ}\bigg(\rightarrow \mbox{ or } \nearrow \mathcal{M} \searrow\bigg)
    \end{equation*}
    In words, this states that any Motzkin path can be decomposed uniquely into a sequence of the following: a horizontal step; or, a single step up followed by a Motzkin path followed by a single step down.  Additionally, it is simple to track both \DEG\ and \UNP\ in this specification: each horizontal step corresponds to an unpaired nucleotide, and each occurrence of $\nearrow \mathcal{M} \searrow$ adds one to \DEG.  In total,
    \begin{equation} \label{eq:MetaMotzkin}
    M(z, u, v) = \frac{1}{1 - zu - z^2v M(z, 1, 1)}.
    \end{equation}
  Critically, the discriminant of the univariate GF still determines the contributing singularity of the generating function and hence the distribution.  After rationalizing \eqref{eq:MetaMotzkin}, we obtain:
    \[
    M(z, u, v) = 2\frac{2 - 2uv - v + vz - v\sqrt{(1-3z)(1+z)}}{(2-2uz-v+vz)^2 + v^2(1-3z)(1+z)}
    \]
    This GF satisfies the conditions of Proposition~\ref{prop:GF}: first, $g(z, u, v) = v$, and all functions are analytic.  Consider the zeroes of the denominator.  For $|z| \leq 1/3$ and $u$ and $v$ sufficiently small, $(2-2uz-v+vz)$ is bounded away from zero, implying that the denominator overall is bounded away from zero for $|z| \leq 1/3$ and $u$ and $v$ sufficiently small. Hence, the algebraic singularity $z = 1/3$ is indeed the minimal singularity of the GF in all cases.

   Substituting $z = 1/3$ into the denominator causes the $(1-3z)$ term to disappear, leaving us with $a(1/3, u, v) = 4/9(u + v - 3)^2$ in the terms of Proposition~\ref{prop:GF}.  Then, substituting into the formula of Proposition~\ref{prop:GF} completes the proof.

\end{proof}

\begin{cor} \label{cor:Motzkin}
    For the uniform distribution over Motzkin paths, the limiting distributions of \DEG, \UNP, \COV, \LEN, and \ETE\ are as in \Cref{table:MainProp}.
\end{cor}

\begin{proof}
    With the limiting bivariate GF from \Cref{lem:MotzkinPaths}, the distributions for \DEG, \UNP, \COV, and \LEN\ are easily recovered.  For \DEG, note that substituting $1$ into $u$ eliminates tracking \UNP\ by combining all terms with the same value of \DEG\ but different values of \UNP\ in the bivariate sum.  Hence, \DEG\ has limiting PGF $M_p(1, v) = v/(v-2)^2$. The negative binomial distribution $\NB(2, 1/2)$ has PGF $1/(v-2)^2$, so the additional $v$ in the numerator shifts the distribution up by one.  Similarly, \UNP\ has limiting PGF $M_p(u, 1) = 1/(u-2)^2$, corresponding exactly to $\NB(2, 1/2)$.

    For $\COV = \DEG + \UNP - 1$, we first look for the PGF for $\DEG + \UNP$, obtained by treating \DEG\ and \UNP\ the same in the bivariate GF.  Equivalently, we make the substitution $u = v$, which means that the coefficient $[u^k]M_p(u, u)$ is the sum of all coefficients $[u^m v^n] M_p(u, v)$ where $m + n = k$.  This yields $M_p(u, u) = u/(2u-3)^2$, the PGF for $1 + \NB(2, 1/3)$.  We subtract one from this distribution to find \COV.

    Next, for \LEN, to account for the fact that the degree adds twice as much as any unpaired base, we make the substitution $v = u^2$ to obtain $L(u) := M_p(u, u^2)$.  This yields the limiting PGF in the statement.  To find the mean, compute $L'(1)$, and to find the variance, compute $L''(1) + L'(1) - (L'(1)^2).$

    No substitution will give the limiting PGF for \ETE, but we can still compute statistics on this distribution to arbitrary precision by first identifying the joint distribution of \UNP\ and \DEG.  This is a generalization of the proof of the mean and variance of \ETE\ for the Dyck path case, and we give details in the appendix, \Cref{sec:MotzkinETE}.

\end{proof}

\subsection{End-proximity in the Pfold stochastic context-free grammar} \label{sec:Pfold}

Successfully used to predict \rna\ secondary structures, stochastic context-free grammars assign a distribution to dot-bracket sequences $\mathcal{L}$.  Grammars are defined via a set of recursive rules that begin with a specified start symbol, $S$, and evolve until only the terminal symbols of parentheses and dots remain.  The Pfold grammar~\cite{Knudsen:2003} is a particularly successful grammar that uses relatively few rules to achieve a high prediction accuracy~\cite{Dowell:2004}.  This grammar is defined with the following rules, where each rule is chosen with a corresponding probability $p_i$ or $q_i$ with $p_i + q_i = 1$:
\[
[1]\ S \to LS\ (p_1)\ \big|\ L\ (q_1) \hspace{.4 in}[2]\ L \to (F) \ (p_2)\ \big| \ . \ (q_2) \hspace{.4 in} [3]\ F \to (F)\ (p_3) \ \big|\ LS\ (q_3)
\]
Here, the parameters $p_i$ and $q_i$ can be chosen to fit different families of \RNA, but default values are given in the original paper~\cite{Knudsen:2003}.

In~\cite{Poznanovic:2014}, the authors found that in the output of the Pfold grammar, many common structural features of \RNA\ are normally distributed.  Curiously, the distributions of some features were connected in ways that are independent of the parameters of the grammar.  In contrast, we identify that even in the Pfold grammar, the measures of end-proximity have limiting distributions that are again negative binomials.  This is again due to a factorization of the discriminant of the corresponding GF.  As in the Dyck and Motzkin analyses above, the singularities of the corresponding univariate GF play an important role.  In~\cite{Poznanovic:2014}, the authors studied this very GF, so we cite those results when necessary here.  

\begin{proposition}
\label{prop:Pfold}
    Define $\rho$ as the unique smallest positive solution to $R(z) = (1 - p_1q_2z)^2(1-p_3z^2) - 4p_2q_1q_2q_3z^3$, and define $\delta = p_1q_2\rho$.  Consider the distribution assigned by the Pfold grammar on dot-bracket sequences $\mathcal{L}$ as their length $n \to \infty$.  Then, the joint limiting distribution of \DEG\ and \UNP\ is given by the bivariate PGF,
    \[
    \frac{(1-\delta)^2v}{(1 + \delta + (\delta^2 - \delta)v - (\delta^2 + \delta)u)^2}.
    \]
    From this, the limiting distributions of \DEG, \UNP, \COV, \LEN, and \ETE\ are as in \Cref{table:MainProp}.
\end{proposition}

\begin{proof}

We first need to rewrite the Pfold grammar so that \DEG\ and \UNP\ can be more easily tracked.  To do this, we add two additional symbols, $S_{\ex}$ and $L_{\ex}$, corresponding to when the grammar is still producing part of the exterior loop.  Then, $S_{\ex}$ becomes the new start symbol for the grammar, and we can mark when the grammar transitions from the exterior loop to other nucleotides by dropping the ``$\ex$'' tag.  We add the following two rules:
\[
[1']\ S_{\ex} \to\ L_{\ex}S_{\ex} \ (p_1)\ \big|\ L_{\ex} \ (q_1) \hspace{0.4 in} [2']\ L_{\ex} \to (F) \ (p_2) \ \big|\ .\ (q_2)
\]
Importantly, these rules still use the same probabilities as rules $[1]$ and $[2]$ in the original grammar so that the distribution of structures remains unchanged.

By using the Chomsky-Sch\"{u}tzenberger theorem, we can convert this recursive definition of the grammar's symbols into a system of equations for the trivariate generating function $S_{\ex}(z, u, v) = \sum p(n, k, \ell) z^n u^k v^\ell$ where $p(n, k, \ell)$ represents the probability that the grammar outputs a structure of length $n$ with $k$ unpaired nucleotides and degree $\ell$:
\[
[1] \ S(z, u, v) = p_1 LS + q_1L \hspace{.2 in} [2]\ L(z, u, v) = p_2z^2F + q_2z \hspace{.2 in} [3]\ F(z, u, v) = p_3z^2F + q_3LS
\]
\[
[1']\ S_{\ex}(z, u, v) = p_1 L_{\ex}S_{\ex} + q_1L_{\ex} \hspace{.4 in} [2']\ L_{\ex}(z, u, v) = p_2vz^2F + q_2uz
\]
The functions $S, L, F,$ and $F_{\ex}$ are auxiliary GFs that correspond to the output of the grammar had it started with the corresponding symbol instead of $S_{\ex}$, but they are useful in solving for $S_{\ex}$.  One way to find the minimal polynomial for $S_{\ex}$ is to use Gr\"{o}bner bases, as illustrated in our companion {\tt SageMath} worksheet.  Overall, we obtain a function of the form
\begin{equation}
S_{\ex}(z, u, v) = \frac{-b(z, u, v) - q_1v\sqrt{(1-z^2p_3)R(z)}}{2a(z, u, v)} \label{eq:Sext}
\end{equation}
Recall that $R(z) = (1 - p_1q_2z)^2(1-p_3z^2) - 4p_2q_1q_2q_3z^3$, which is the function determining the singularities of the univariate GF tracking only the length of the output of the grammar. From \cite[Lemma 3.3]{Poznanovic:2014}, the unique minimal zero $\rho$ of $R(z)$ has $\rho < 1/\sqrt{p_3}$, so $z = \rho$ is indeed the minimal singularity of the function underneath the square root.

There are two technically challenging components remaining to finish deriving the limiting PGF.  First, we must show that $z = \rho$ is the unique minimal singularity for $S_{\ex}(z, u, v)$ for all $u$ and $v$ sufficiently small.  Then, the expression for $a(z, u, v)$ is exceedingly messy.  To recognize the limiting distribution as related to negative binomials, we must find a way to rewrite it.

For the first challenge, we find that the singularities of $S_{\ex}(z, u, v)$ either come from the zeroes of $a(z, u, v)$ or $(1-z^2p_3)R(z)$.  Having analyzed the latter already, we look at the zeroes of $a$ when $u$ and $v$ are small.  The zeroes of $a(z, 0, 0)$ are of magnitude $z = \sqrt{p_1/(p_1p_3 - p_2q_1q_3)}$, as verified in the {\tt SageMath} worksheet.  It is easy to see that $\sqrt{1/p_3}$ is strictly smaller than this, and \cite[Lemma 3.3]{Poznanovic:2014} again verifies that $\rho < \sqrt{1/p_3}$.  Hence, $|a(z, 0, 0)|$ is bounded away from zero for all $|z| \leq \rho$.  Because $a(z, u, v)$ is a polynomial, for $u$ and $v$ sufficiently small, $|a(z, u, v)|$ is also bounded away from zero for all $|z| \leq \rho$.  Hence, we have satisfied the conditions of Proposition~\ref{prop:GF} and can compute the limiting probability distribution.

Now, to rewrite $a(z, u, v)$ in a more convenient form, it is useful to note that $a(z, u, v)$ will only be evaluated when $z = \rho$.  So, we are free to add and subtract multiples of $R(z)$ without changing the output of the proposition.  By matching terms by hand, we find:
\[
4q_2z \cdot a(z, u, v) = (p_1^2q_2^2uz^2 - p_1^2q_2^2vz^2 + p_1q_2uz + p_1q2vz - p_1q_2z - 1)^2(p_3z^2 - 1) + C(z, u, v)R(z)
\]
for an appropriate choice of $C(z, u, v)$ given in the {\tt SageMath} worksheet.  Substituting this form of $a(z, u, v)$ into Proposition~\ref{prop:GF} with $g(u, v) = q_1v$ yields the limiting PGF given in the proposition statement.

The default values of the $p_i$ are given in the original paper~\cite{Knudsen:2003} as $p1 = 0.868534, p2 = 0.105397,$ and $p3 = 0.787640$.  Following the same steps as for the Dyck and Motzkin paths, we can find the distributions of each parameter in the proposition with appropriate substitutions or series expansions of this PGF.
\end{proof}

\section{Exterior helices} \label{sec:Helices}

\begin{table}[h]
\centering
\begin{tabular}{cl}
\toprule
Parameter & Definition\\
\cmidrule(l){1-2}
\HEL & Number of consecutive pairings in the first set of pairings;\\
&length of the first helix in Motzkin paths and Pfold grammar\\
\STM & Number of nucleotides in the first stem of a secondary structure\\
\bottomrule
\end{tabular}
\caption{Parameters measuring the length of the first helix or stem in a secondary structure.}
\label{table:ExteriorHelixParameters}
\end{table}

As we will see in \Cref{sec:Data}, our analysis of databases of \rna\ secondary structures reveals that many \rna\ structures have low degree even compared to our models.  In particular, some well-structured families like 5S r\rna\ and t\RNA\ generally have degree one.  This led us to consider the length of this helix as a proxy for the pairing propensity of the ends of the sequence.

In general, we focus on the first helix in a secondary structure when reading from the $5'$ end to the $3'$ end of the strand.  Let \FirstHelix\ represent the length of the helix with a nucleotide closest to the $5'$ end of the strand.

In \Cref{sec:DyckHelices} below, we investigate the length distribution of exterior helices under the branching models defined by Dyck paths and Motzkin paths.

\subsection{Closing helix length in Dyck and Motzkin paths}
\label{sec:DyckHelices}
In terms of Dyck and Motzkin paths, \HEL\ corresponds to the height of its lowest valley or plateau.
We refer to this parameter as the \emph{depth} of a mountain.  For Dyck paths and Motzkin paths, we may rely on Proposition~\ref{prop:GF} again to discover the asymptotic distribution of depth.  In both cases, the resulting distributions are geometric, or equivalently, a negative binomial with only one coin flip.

\begin{figure}
\begin{center}
    \includegraphics[width=\textwidth]{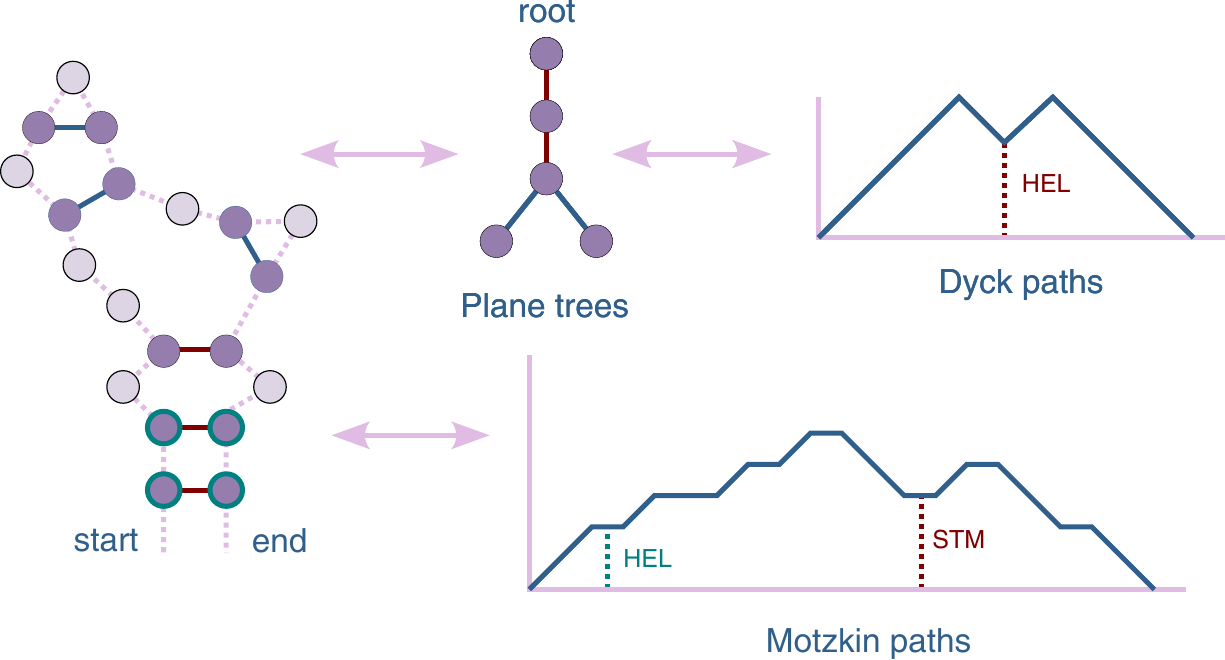}
\end{center}
    \caption{The length of the first helix in an \RNA\ secondary structure is equal to the depth of the corresponding Motzkin path or Dyck path, and the length of the first stem corresponds to the lowest valley or plateau excluding the starting and ending sides of the mountain.}
    \label{fig:MountainHeight}
\end{figure}

Two base pairs $(i,j)$ and $(k,l)$ with $i < k$ belong to the same stem if
$i < k < l < j$ and there is no base pair $(a,b)$ with $i < a < b < k$ 
or $l < a < b < j$.  The length of a stem is the total number of base pairs in it.  In terms of Motzkin paths, the first stem is the lowest occurrence of a down step followed by an up step.  Equivalently, it is the lowest valley or plateau excluding the two sides of the first mountain. 

Note that for the plane tree representation of secondary structures that corresponds to Dyck paths, each pair of parentheses corresponds to a helix (unlike in Motzkin paths where each pair corresponds to a single base pair).  Thus, for Dyck paths, \HEL\ counts the number of helices in the first stem of a secondary structure.

\begin{proposition} \label{prop:DyckMotzkinDepth}
Consider the uniform distribution over all Dyck or Motzkin paths of length $n$ as $n \to \infty$.
\begin{enumerate}
    \item For Dyck paths, \FirstHelix\ tends to $1 + \NB(1, 3/4)$ with mean $4/3$ and variance $4/9$.
    \item For Motzkin paths, \FirstHelix\ tends to $1 + \NB(1, 8/9)$ with mean $9/8$ and variance $9/64$.
    \item For Motzkin paths, the number of helices in the first stem tends to $1 + \NB(1, 27/32)$ with mean $32/27$ and variance $160/729$.
    \item For Motzkin paths, \STM\ tends to $1 + \NB(1, 3/4)$ with mean $4/3$ and variance $4/9$.
\end{enumerate}
\end{proposition}

\begin{proof}
    Our goal is to find a recurrence relation for each GF that tracks the corresponding depth of the paths.

    Let $D(z,y)$ be the GF which tracks Dyck paths by length ($z$) and height ($y$), where length is again the number of pairs of parentheses.  In order to find the recurrence for $D$, will need to define an auxiliary GF $H(x, y)$, described momentarily.  We begin with the standard Dyck path recurrence relation, modified slightly to include this helper GF:
    \[
    D(z, y) = 1 + [zy H(z,y)] D(z,1) 
    \]
    Overall, this recurrence is motivated by decomposing a Dyck path according to its first return to the $x$-axis, as in \Cref{fig:Dyck}.  A Dyck path is either empty (contributing a $1$), or it can be broken into a pair of parentheses enclosing a Dyck path, $(zy)$, the Dyck path in between the parentheses $(H(z, y))$, followed by another Dyck path $(D(z, 1))$.

    \begin{figure}
        \begin{center}
        \includegraphics[width=0.5\textwidth]{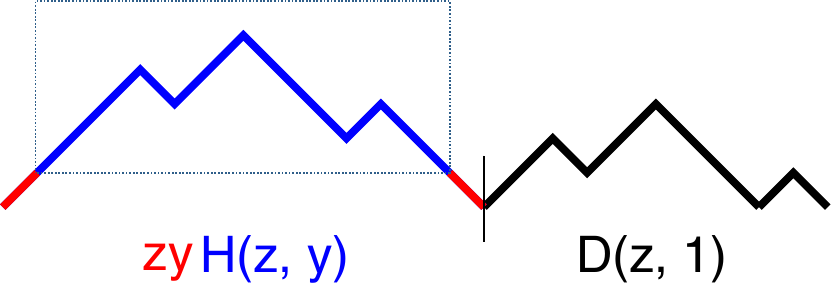}
        \caption{A decomposition of Dyck paths according to the first return to the line $y = 0$.  The Dyck path within the box adds to the length of the first helix only if its first and last position are paired.} \label{fig:Dyck}
        \end{center}
    \end{figure}
    Now, we explain why the terms are $H(z, y)$ and $D(z, 1)$: for $D(z, 1)$, this is the univariate generating function for all Dyck paths by length, which we include because the Dyck path at the end of the recurrence cannot contribute to the depth of the first mountain.  Then, we define $H(z,y)$ to count all Dyck paths by length $(z)$ and depth $(y)$, except Dyck paths with 2+ mountains (or with zero) are now counted by coefficients of $z^i y^0$.  This is necessary because the depth of a path with 2+ mountains will be $1$ after the path is enclosed by the set of parentheses in the recurrence, and should be counted in the overall generating function $D(z, y)$ by the coefficient of $z^{i + 1} y^{0 + 1}$.  In other words, $H(z, y)$ only assigns positive depth to Dyck paths whose first step is paired with its last step.  We obtain the recurrence,
    \[
    H(z,y) = (1-z) D(z,1) + zy H(z,y),
    \]
    where the $(1-z)D(z, 1)$ term counts Dyck paths whose first step is not paired with its last step, and the $zyH(z, y)$ term counts the rest.

    Solving this system reveals that the GF $D(z, y)$ is of the form required in Proposition~\ref{prop:GF} with $g(y) = y$, $a(z, y) = y(z^2 - 1)$, and $\rho = 1/4$.  Plugging into the proposition yields the limiting distribution encoded by $3y/(4-y)$, which is the PGF for the geometric random variable with parameter $p = 3/4$.
   
    Paralleling the changes for the Dyck path specification, let $M(z, y)$ be the bivariate GF enumerating Motzkin paths by length ($z$, now counting individual nucleotides) and by depth of the first left-to-right mountain $(y)$.  Also, let $H(z, y)$ be the equivalently defined helper function.  Then, the following specification holds:
    \begin{align*}
    M(z, y) &= 1 + zM + z^2yH(z, y)M(z, 1)\\
    H &= z^2yH(z, y) + (1-z^2)M(z, 1)
    \end{align*}
    Solving for $M(z, y)$ again reveals a GF fitting the requirements of Proposition~\ref{prop:GF}, this time with $g(y) = y$, $a(z, y) = z^2(yz^3 - yz^2 - z + 1)$, and $\rho = 1/3$.  The proposition gives the limiting GF $(8y/9)/(1-y/9)$,
    which again corresponds to a geometric random variable, this time with parameter $p = 8/9$.

    For the length and number of components of the stem in a Motzkin path, we develop two grammars that output all dot-bracket sequences exactly once but allow us to track additions to the stem.  First, for \STM, we have the start symbol $M_f$ corresponding to when the grammar has yet to start generating the first stem.  Then, $M_*$ is the symbol corresponding to the first stem being in the middle of production, and $M$ is the symbol for when the grammar has moved past producing the stem.  Additionally, we have a symbol $L$ that corresponds to adding any number of consecutive dots.  We obtain:
    \begin{align*}
        M_f &\to \emptyset\ \big|\ .M_f\ \big|\ (M_*)M\\
        M_* &\to L\ \big|\ L(M_*)L\ \big|\ L(M)L(M)M\\
        M &\to \emptyset \ \big|\ .M \ \big|\ (M)M\\
        L &\to \emptyset \ \big|\ .L
    \end{align*}
    Here, $\emptyset$ represents outputting nothing.  The symbols $M_f$ and $M$ both have production rules modeled after the standard recursion for Motzkin paths.  In contrast, $M_*$ has rules broken into whether the output has no mountains $(L)$, one mountain $(L(M_*)L)$, or more than one mountain, which breaks the stem.  This converts to the system of equations
    \begin{align*}
        M_f &= 1 + zM_f + z^2uM_*M,\\
        M_* &= L + z^2uL^2M_* + z^4L^2M^3\\
        M&=1 + zM + z^2M^2\\
        L &= 1+zL
    \end{align*}
    This is solved via Gr\"{o}bner bases in our {\tt SageMath} worksheet, leading to a generating function where Proposition~\ref{prop:GF} applies.  The limiting GF is $3u/(4-u)$, again corresponding to the PGF for a geometric random variable with $p = 3/4$.

    The grammar needs a slight modification to track the number of helices in the stem.  In the rules for $M_*$, we now need to know more information about the production $M_* \to L(M_*)L$: if both $L$s are empty, a pre-existing helix is growing longer, but if either $L$ is non-empty, a new helix is formed in the stem instead.  This change corresponds to replacing the equation for $M_*$ with
    \[
    M_* = L + z^2M_* + z^2u(2(L-1) + (L-1)^2)M_* + z^4L^2M^3.
    \]
    Every other part of the analysis is the same, and the limiting PGF is now $27u/(32 - 5u)$, corresponding to the geometric random variable with parameter $p = 27/32$.
\end{proof}

In the result above, we note that the distribution of \HEL\ in Dyck paths is the same as the distribution of \STM\ for Motzkin paths.  We could derive the Motzkin result directly from the Dyck path result by noticing that a Motzkin path can be compressed to a Dyck path skeleton by removing all of the unpaired bases.  The number of Motzkin paths of length $n$ with $k$ base pairs that correspond to the same Dyck skeleton depends only on $n$ and $k$, and not the arrangement of the skeleton.  Ultimately, this implies that the distributions should be the same.

Additionally, we note that the results imply an asymptotic independence: the mean number of components in the first stem times the mean length of the first helix is $32/27 \cdot 8/9 = 4/3$, the mean length of the stem.

\subsection{Closing helix length in the Pfold grammar}

The statement for the Pfold grammar is again analogous to Dyck and Motzkin paths.  The primary challenge is finding an equivalent form of the Pfold grammar that allows one to track the length of the first helix.

\begin{proposition} \label{prop:PfoldHelix}
    Define $\rho$ as the unique smallest positive solution to $R(z) = (1 - p_1q_2z)^2(1-p_3z^2) - 4p_2q_1q_2q_3z^3$.  In the Pfold grammar, \HEL\ has a limiting distribution of $1 + \NB(1, 1-p_3\rho^2)$.
\end{proposition}

\begin{proof}
    As in Proposition~\ref{prop:Pfold}, we begin with the three production rules $[1], [2],$ and $[3]$, and we add two new symbols $S_f$ and $F_f$ that track when the grammar is still outputting the first helix versus when it has moved past the helix.  So, symbols without the tag `$f$' correspond to productions that cannot impact the first helix.  Then, we add the two rules:
    \[
    [1'']\ S_f \to (F_f)S \ (p_1p_2) \ \big| \ .S_f \ (p_1q_2) \ \big| \ (F_f) \ (q_1p_2) \ \big| \ . \ (q_1q_2)
    \]
    \[
    [2'']\ F_f \to (F_f) \ (p_3) \ \big| \ LS \ (q_3)
    \]
    Rule $[1'']$ was derived by combining rules $[1]$ and $[2]$ in the original grammar, which is necessary in order to track whether the initial symbol $L$ in the rule $S \to LS$ becomes the first helix or whether it becomes a dot, in which case the remaining $S$ may still become the first helix.

    From this point onwards, the analysis pipeline matches that of Proposition~\ref{prop:Pfold}, and details are provided in the {\tt SageMath} worksheet.
\end{proof}

\section{Comparison to known RNA families} \label{sec:Data}

To gain insight into whether our models accurately reflect the end-to-end distances and helix lengths of \rna\ secondary structures, we compare to well-studied families of \rna\ from the ArchiveII database \cite{Sloma:2016}, the Comparative RNA\ Web-2 \cite{Cannone:2002, Chan:2023}, and the lnc\RNA\ and m\RNA\ from \cite{Lai:2018}.  The families we studied are listed in \Cref{table:FamilyInfo}.  Our results by family are summarized in \Cref{table:ArchiveIIStatistics}, \Cref{table:CRWCombinedStatistics}, and \Cref{table:MathewsStatistics}.

For the families of \RNA\ from the ArchiveII database or the Comparative RNA Web-2, we took multiple sequence representatives from each family, while for the m\RNA\ and lnc\RNA\ from \cite{Lai:2018}, we generated samples of 1000 structures for each sequence by using Rsample from the RNAStructure software suite \cite{Reuter:2010}.  

\begin{table}[t]
\centering
\begin{tabular}{c}
  \begin{tabular}{c|ccccccc}
  \toprule
  Family & tRNA & SRP & 5S & RNaseP & tmRNA & 16S & 23S\\
  Median length & 76 & 118 & 119 & 330 & 363 & 1509 & 3021\\
  \#Sequences & 557 & 924 & 1283 & 454 & 462 & 54 & 38\\
  \bottomrule
  \end{tabular}
  \\
  \addlinespace
  \begin{tabular}{c|ccccc}
  \toprule
  Sequence & RPL41A & ATP & HSBP1 & MIF & Bglobin\\
  Length   & 327    & 424 & 535   & 562 & 590\\
  \cmidrule(l){1-6}
  Sequence & MRP51 & GAPDH & FLuciferase & Hotair & Neat1\\
  Length   & 674   & 1327 & 1658        & 2148  & 3734\\
  \bottomrule
  \end{tabular}
\end{tabular}
\caption{The families and sequences studied in our analysis.  In the top block, each family had multiple representatives.  In the m\RNA\ and lnc\RNA\ in the bottom block, we analyzed samples of size $1000$ from each sequence.}
\label{table:FamilyInfo}
\end{table}

For each family, we computed each measurement of end-to-end distance and the length of the first helix and first stem.  Overall, we find that the mean \ETE\ value is often somewhat smaller than our theoretically predicted averages, but the variance is substantially smaller.

To determine whether small \ETE\ values are a generic property of free energy minimization or a special property of \RNA, we then shuffled the sequences using the uShuffle software \cite{Jiang:2008} while maintaining $2$-let counts.  Then, we used RNAfold from the ViennaRNA software suite \cite{Lorenz:2011}.  RNAfold is a minimum free energy (MFE) secondary structure prediction algorithm.  In general, we find that regardless of sequence length or nucleotide composition, the resulting shuffles move the distributions towards our theoretical averages, increasing the variances and means.  We describe specifics by dataset below. 

\begin{table}[h]
\centering
\begin{tabular}{cccccccccccccc}
\toprule
\multicolumn{1}{c}{Seq.} 
  & \multicolumn{2}{c}{DEG}
  & \multicolumn{2}{c}{UNP}
  & \multicolumn{2}{c}{LEN}
  & \multicolumn{2}{c}{HEL}
  & \multicolumn{2}{c}{STM}
  & \multicolumn{2}{c}{ETE}\\
  & $\mu$ & Var & $\mu$ & Var & $\mu$ & Var & $\mu$ & Var & $\mu$ & Var & $\mu$ & Var \\
\cmidrule(l){1-13}

\multicolumn{13}{c}{\textbf{Original distributions}} \\
\cmidrule(l){1-13}

t\RNA      & 1.0 & 0.0 & 4.0 & 0.3 & 6.0 & 0.3 & 6.8 & 0.6 & 6.9 & 0.4 & 2.1 & 0.0 \\
SRP \RNA   & 1.5 & 0.5 & 6.2 & 25.1 & 9.3 & 31.5 & 5.7 & 12.5 & 17.7 & 192.7 & 2.7 & 0.7 \\
5S r\RNA   & 1.0 & 0.0 & 2.0 & 4.2 & 4.0 & 4.5 & 8.2 & 5.5 & 9.1 & 1.8 & 1.8 & 0.1 \\
RNaseP    & 1.2 & 0.4 & 6.4 & 17.2 & 8.8 & 24.8 & 5.4 & 13.5 & 8.5 & 8.4 & 2.5 & 0.6 \\
tm\RNA     & 1.0 & 0.0 & 3.9 & 1.0 & 5.9 & 1.4 & 7.0 & 0.1 & 7.2 & 4.8 & 2.1 & 0.0 \\

\cmidrule(l){1-13}
\multicolumn{13}{c}{\textbf{Pfold theoretical distribution}} \\
\cmidrule(l){1-13}

Pfold & 2.6 & 2.8 & 12.4 & 18.2 & 17.5 & 138.8 & 4.7 & 17.5 & -- & -- & 3.8 & 2.2 \\

\cmidrule(l){1-13}
\multicolumn{13}{c}{\textbf{Shuffled MFE distributions}} \\
\cmidrule(l){1-13}

t\RNA      & 1.7 & 0.5 & 9.9 & 46.7 & 13.4 & 49.3 & 4.3 & 3.3 & 12.2 & 48.4 & 3.3 & 0.8 \\
SRP \RNA   & 1.9 & 0.8 & 9.2 & 35.5 & 13.1 & 42.7 & 4.4 & 3.6 & 12.7 & 82.1 & 3.3 & 0.8 \\
5S r\RNA   & 2.0 & 0.8 & 10.2 & 40.2 & 14.1 & 47.8 & 4.3 & 3.1 & 12.5 & 78.4 & 3.4 & 0.9 \\
RNaseP    & 2.1 & 1.0 & 10.4 & 51.2 & 14.5 & 63.1 & 4.5 & 3.7 & 11.1 & 75.3 & 3.5 & 1.1 \\
tm\RNA     & 2.1 & 1.0 & 11.1 & 57.3 & 15.4 & 69.9 & 4.6 & 4.0 & 11.8 & 84.4 & 3.6 & 1.1 \\

\bottomrule
\end{tabular}
\caption{Distribution statistics for original, Pfold theoretical, and shuffled MFE structures across \RNA\ families from the ArchiveII database.  Each sequence was shuffled 5 times.}
\label{table:ArchiveIIStatistics}
\end{table}
\begin{table}[h]
\centering
\begin{tabular}{cccccccccccccc}
\toprule
\multicolumn{1}{c}{Seq.} 
  & \multicolumn{2}{c}{DEG}
  & \multicolumn{2}{c}{UNP}
  & \multicolumn{2}{c}{LEN}
  & \multicolumn{2}{c}{HEL}
  & \multicolumn{2}{c}{STM}
  & \multicolumn{2}{c}{ETE}\\
  & $\mu$ & Var & $\mu$ & Var &  $\mu$ & Var &  $\mu$ & Var &  $\mu$ & Var &  $\mu$ & Var \\
\cmidrule(l){1-13}

\multicolumn{13}{c}{\textbf{Original distributions}} \\
\cmidrule(l){1-13}

16S & 4.0 & 0.0 & 36.4 & 127.0 & 44.4 & 127.0 & 5.0 & 0.0 & - & - & 6.6 & 0.6 \\
23S & 4.4 & 20.1 & 24.3 & 1235.8 & 33.0 & 1800.8 & 7.3 & 9.0 & - & - & 4.9 & 15.8 \\

\cmidrule(l){1-13}
\multicolumn{13}{c}{\textbf{Pfold theoretical distribution}} \\
\cmidrule(l){1-13}

Pfold & 2.6 & 2.8 & 12.4 & 18.2 & 17.5 & 138.8 & 4.7 & 17.5 & -- & -- & 3.8 & 2.2 \\

\cmidrule(l){1-13}
\multicolumn{13}{c}{\textbf{Shuffled MFE distributions}} \\
\cmidrule(l){1-13}

16S & 2.2 & 1.0 & 11.6 & 54.2 & 16.0 & 65.8 & 4.8 & 4.4 & 12.8 & 77.4 & 3.7 & 1.1 \\
23S & 2.2 & 1.0 & 11.7 & 55.4 & 16.1 & 67.5 & 5.1 & 5.2 & 12.5 & 92.6 & 3.7 & 1.1 \\

\bottomrule
\end{tabular}
\caption{Distribution statistics for original, Pfold theoretical, and shuffled MFE structures for ribosomal \RNA s.  Each sequence was shuffled 5 times.}
\label{table:CRWCombinedStatistics}
\end{table}
\begin{table}[h]
\centering
\begin{tabular}{cccccccccccccc}
\toprule
\multicolumn{1}{c}{Seq.} 
  & \multicolumn{2}{c}{DEG}
  & \multicolumn{2}{c}{UNP}
  & \multicolumn{2}{c}{LEN}
  & \multicolumn{2}{c}{HEL}
  & \multicolumn{2}{c}{STM}
  & \multicolumn{2}{c}{ETE}\\
  & $\mu$ & Var & $\mu$ & Var &  $\mu$ & Var &  $\mu$ & Var &  $\mu$ & Var &  $\mu$ & Var \\
\cmidrule(l){1-13}

\multicolumn{13}{c}{\textbf{Original Rsample distributions}} \\
\cmidrule(l){1-13}

RPL41A & 1.9 & 0.3 & 16.6 & 46.9 & 20.5 & 58.1 & 3.7 & 0.7 & 5.1 & 3.7 & 4.0 & 1.0 \\
ATP    & 2.0 & 0.0 & 4.3 & 1.1 & 8.3 & 1.1 & 4.4 & 0.3 & 10.8 & 1.4 & 2.8 & 0.0 \\
HSBP1  & 1.1 & 0.1 & 7.8 & 6.4 & 10.0 & 8.7 & 6.4 & 1.8 & 6.9 & 3.0 & 2.7 & 0.2 \\
MIF    & 2.0 & 0.0 & 7.5 & 2.3 & 11.5 & 3.0 & 6.6 & 1.3 & 7.3 & 1.6 & 3.2 & 0.1 \\
Bglobin& 3.2 & 0.4 & 18.0 & 17.9 & 24.3 & 20.3 & 3.7 & 0.4 & 8.4 & 2.8 & 4.8 & 0.3 \\
MRP51  & 1.0 & 0.0 & 7.6 & 2.6 & 9.7 & 3.1 & 3.9 & 0.8 & 13.2 & 1.5 & 2.6 & 0.1 \\
GAPDH  & 2.2 & 2.2 & 7.4 & 45.3 & 11.8 & 83.6 & 3.1 & 2.9 & 5.8 & 3.3 & 3.1 & 2.1 \\
FLucif. & 1.9 & 0.7 & 12.3 & 65.9 & 16.2 & 86.0 & 4.0 & 3.2 & 5.8 & 9.9 & 3.6 & 1.3 \\
Hotair & 3.9 & 0.4 & 13.5 & 11.8 & 21.2 & 18.0 & 2.8 & 0.8 & 19.9 & 16.3 & 4.7 & 0.4 \\
Neat1  & 2.3 & 0.5 & 15.7 & 52.9 & 20.4 & 51.0 & 6.3 & 1.9 & 9.4 & 9.2 & 4.2 & 0.6 \\

\cmidrule(l){1-13}
\multicolumn{13}{c}{\textbf{Pfold theoretical distribution}} \\
\cmidrule(l){1-13}

Pfold & 2.6 & 2.8 & 12.4 & 18.2 & 17.5 & 138.8 & 4.7 & 17.5 & -- & -- & 3.8 & 2.2 \\

\cmidrule(l){1-13}
\multicolumn{13}{c}{\textbf{Shuffled MFE distributions}} \\
\cmidrule(l){1-13}

RPL41A & 2.1 & 1.0 & 13.4 & 73.4 & 17.6 & 87.7 & 5.0 & 3.9 & 12.6 & 89.1 & 3.8 & 1.3 \\
ATP    & 2.2 & 1.1 & 11.7 & 70.5 & 16.0 & 83.9 & 4.4 & 3.1 & 10.3 & 54.4 & 3.6 & 1.3 \\
HSBP1  & 2.1 & 0.9 & 12.6 & 55.1 & 16.9 & 67.5 & 4.9 & 4.2 & 12.3 & 76.4 & 3.7 & 1.1 \\
MIF    & 2.1 & 1.0 & 9.9 & 38.1 & 14.1 & 48.7 & 4.6 & 3.9 & 11.9 & 72.2 & 3.4 & 0.9 \\
Bglobin& 2.2 & 1.1 & 10.4 & 38.2 & 14.7 & 48.9 & 4.9 & 3.8 & 12.5 & 86.8 & 3.5 & 0.9 \\
MRP51  & 2.2 & 1.1 & 9.8 & 38.6 & 14.3 & 48.4 & 4.8 & 4.5 & 12.4 & 83.6 & 3.5 & 0.9 \\
GAPDH  & 2.1 & 1.0 & 11.0 & 52.9 & 15.2 & 65.3 & 4.7 & 3.7 & 11.6 & 79.0 & 3.5 & 1.1 \\
FLucif. & 2.1 & 1.0 & 12.2 & 62.9 & 16.4 & 74.3 & 5.0 & 4.5 & 11.7 & 72.0 & 3.7 & 1.1 \\
Hotair & 2.2 & 1.1 & 11.3 & 62.8 & 15.6 & 76.7 & 4.6 & 4.0 & 10.7 & 64.9 & 3.6 & 1.2 \\
Neat1  & 2.2 & 1.1 & 10.1 & 40.5 & 14.5 & 52.6 & 5.1 & 5.5 & 13.4 & 102.1 & 3.5 & 1.0 \\

\bottomrule
\end{tabular}
\caption{Distribution statistics for m\RNA\ and lnc\RNA\ sequences and their shuffles.  Each sequence was shuffled 1000 times.}
\label{table:MathewsStatistics}
\end{table}

\subsection{ArchiveII}

In the ArchiveII database, many structures exhibited high percentages of unpaired bases, which suggests an absence of pairing information for these positions.  To approximate the complete structures, we removed any pseudoknots and non-canonical pairs from each structure.  Then, we used the remaining partial structure as a constraint and filled out the rest of the structure using RNAfold, with the standard command-line setting of disallowing lonely pairs.  As reported in \Cref{table:ArchiveIIStatistics}, all five of these \RNA\ families have \ETE\ values concentrated at low values (typically $< 2.5$nm) even though the families cover a substantial range of sequence lengths.

\begin{figure}
    \begin{center}
        \includegraphics[scale=0.68]{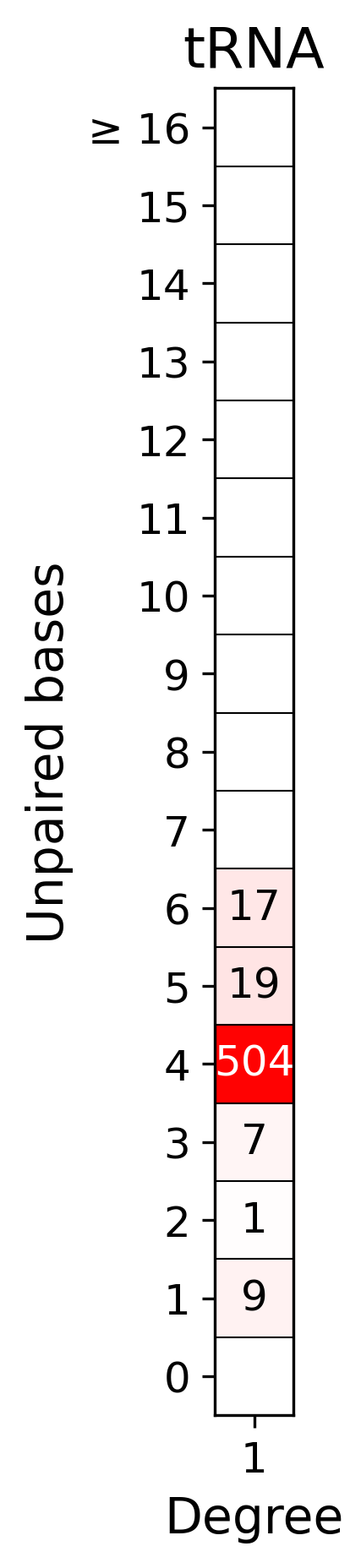}
        \includegraphics[scale=0.68]{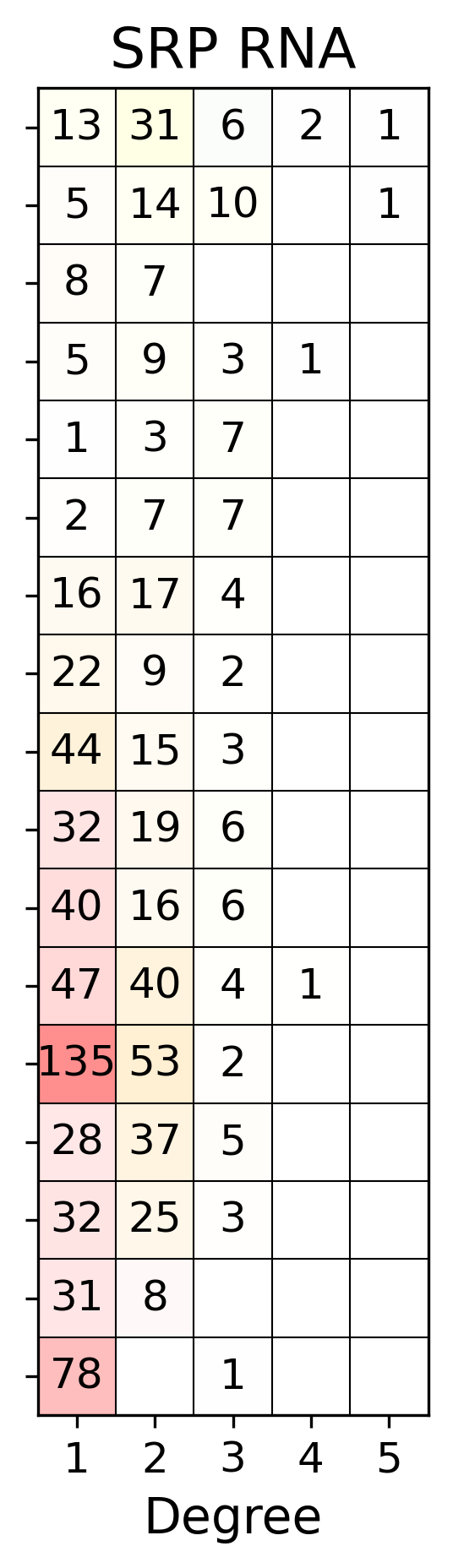}
        \includegraphics[scale=0.68]{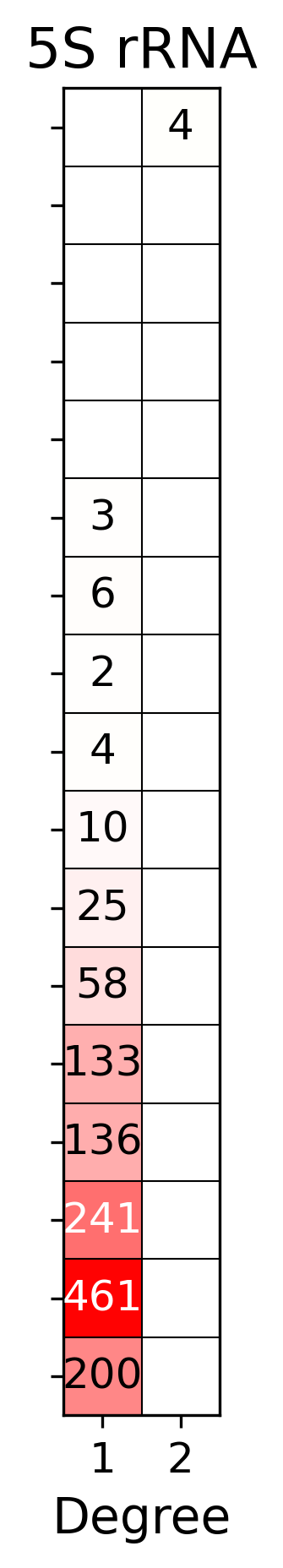}
        \includegraphics[scale=0.68]{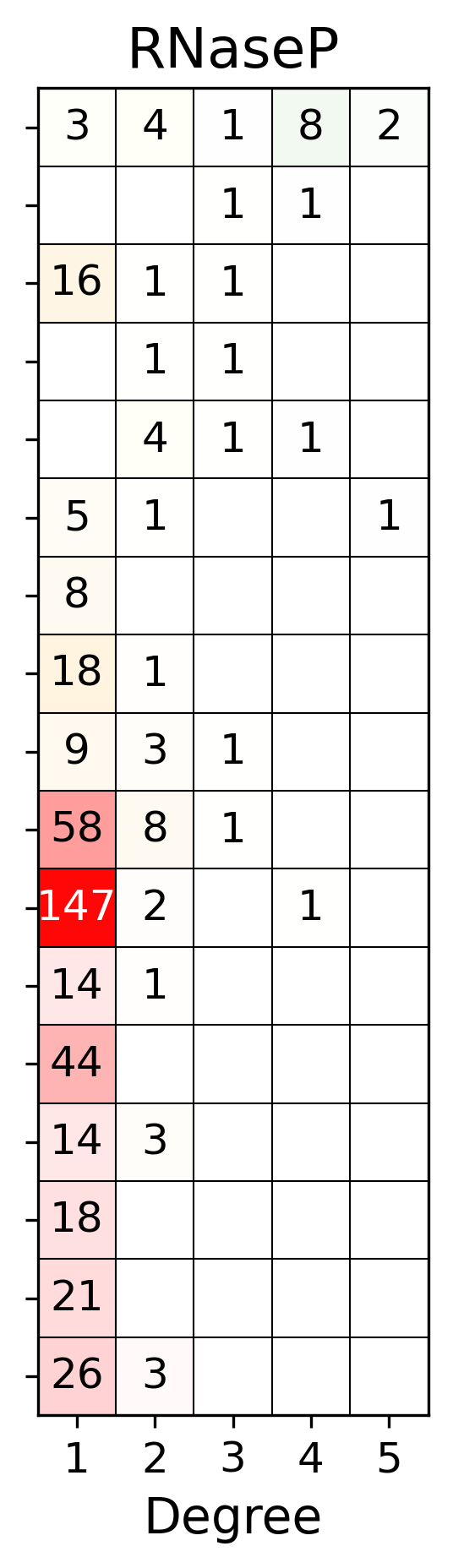}
        \includegraphics[scale=0.68]{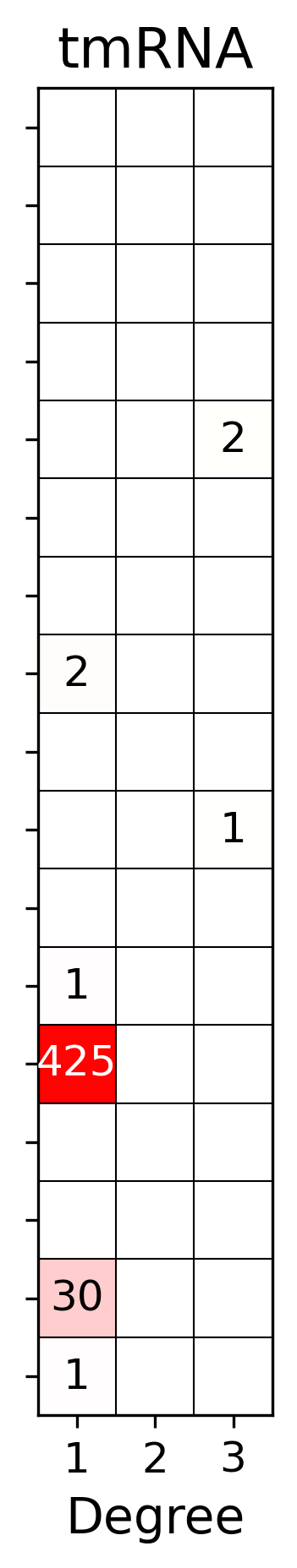}
    \end{center}
    \caption{The distribution of \DEG\ versus \UNP\ across the ArchiveII dataset.  The intensity of each color represents the percent of structures with the corresponding \DEG\ and \UNP\ values.  Colors represent approximate \ETE\ distances, color-coded as red: 1.5 - 2.5, orange: 2.5-3.5, yellow; 3.5-4.5, green: 4.5-5.5, blue: 5.5-6.5, purple: 6.5-7.5.}
    \label{fig:ShortHeat}
\end{figure}

In \Cref{fig:ShortHeat}, we examine the variability of \UNP\ and \DEG\ within each family.  The effect of homology is clear within the t\RNA, 5S r\RNA, and tm\RNA\ families, as \DEG\ is consistently $1$ with a negligible variance, and even \UNP\ remains concentrated.  In contrast, \SRP\ and \RNaseP\ exhibit a wider breadth of possible secondary structures.  This leads to a bit more variability for \DEG, but much more variability in \UNP\ and therefore also \LEN.  Even here, the structures are concentrated around relatively few modes.  Despite this increased variability, the formula defining \ETE\ weights \DEG\ more heavily than \UNP.  Thus, even \SRP\ and \RNaseP\ have consistently small \ETE\ values with small variances.  In all cases, the \ETE\ values and variances are smaller than predicted by our Pfold theoretical distribution.

In contrast to \ETE, \HEL\ and \STM\ tend to be much longer than our predicted models for every family.  Although \SRP\ and \RNaseP\ illustrate the smallest values of \HEL, it turns out that these helices often form the initial segment of longer stems.  For \RNaseP, many helices of length 3 are part of a stem of length 8, while for \SRP, some stems remain short while others are much longer, explaining the large variance seen here.

For each sequence, we then created five shuffles using uShuffle, and computed the MFE structure using RNAfold.  For each family, the mean and variance of \ETE\ are now much closer to the Pfold theoretical distribution both in terms of mean and variance as well as in terms of overall shape.  For example, in \Cref{fig:ArchiveIISideBySide},  we find that the distribution of \LEN\ is concentrated for many of the families, although the values were still within the typical bounds of the theoretical distribution.  This reflects how sequences within any one family do not form uniformly random samples of \rna.  In contrast, once the sequences were shuffled, we find a distribution closely resembling a negative binomial.  Graphs for the other parameters can be found in \Cref{sec:DataGraphs}.

\begin{figure}
    \begin{center}
        \includegraphics[width=0.48\textwidth]{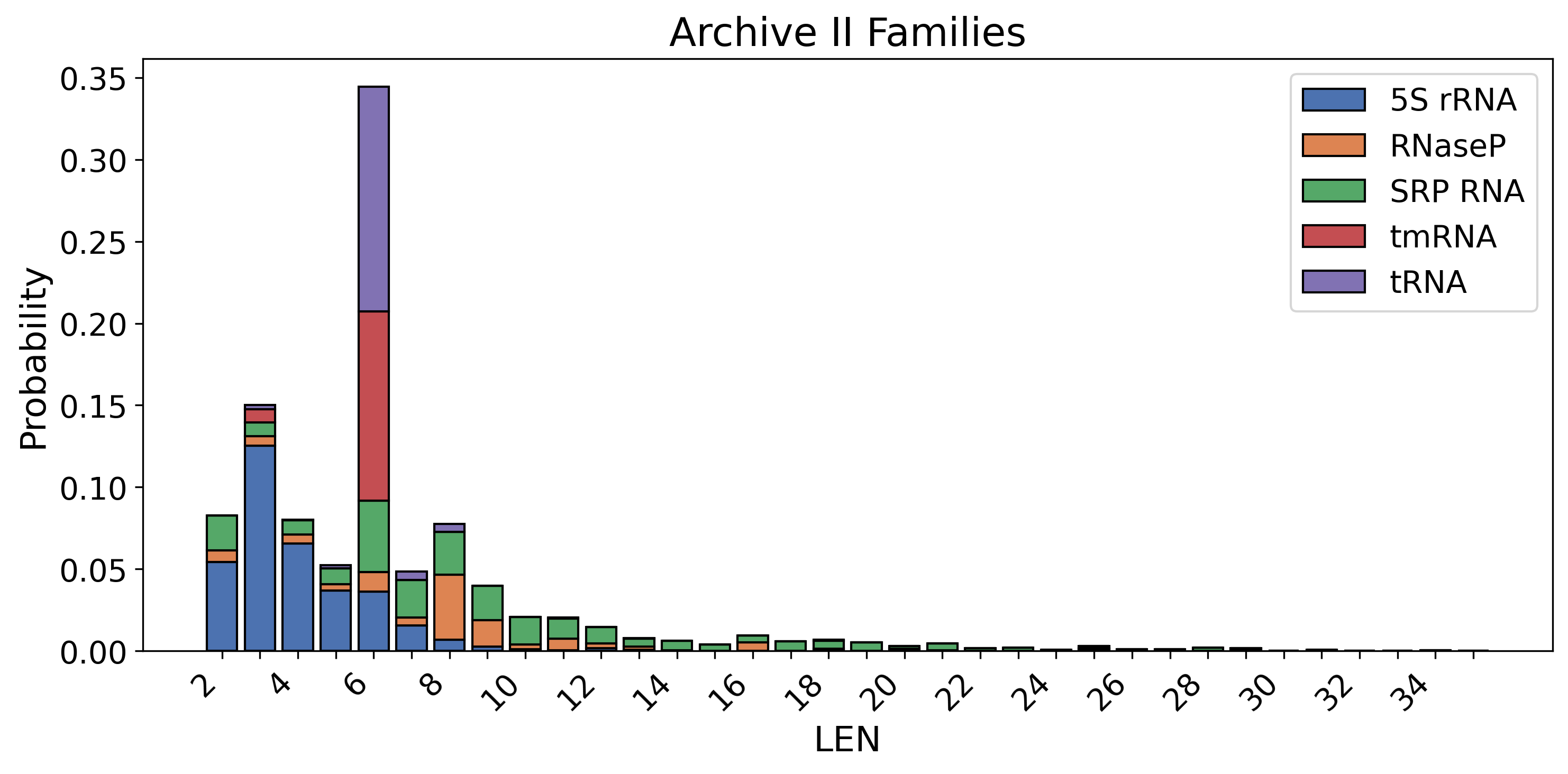}
        \includegraphics[width=0.48\textwidth]{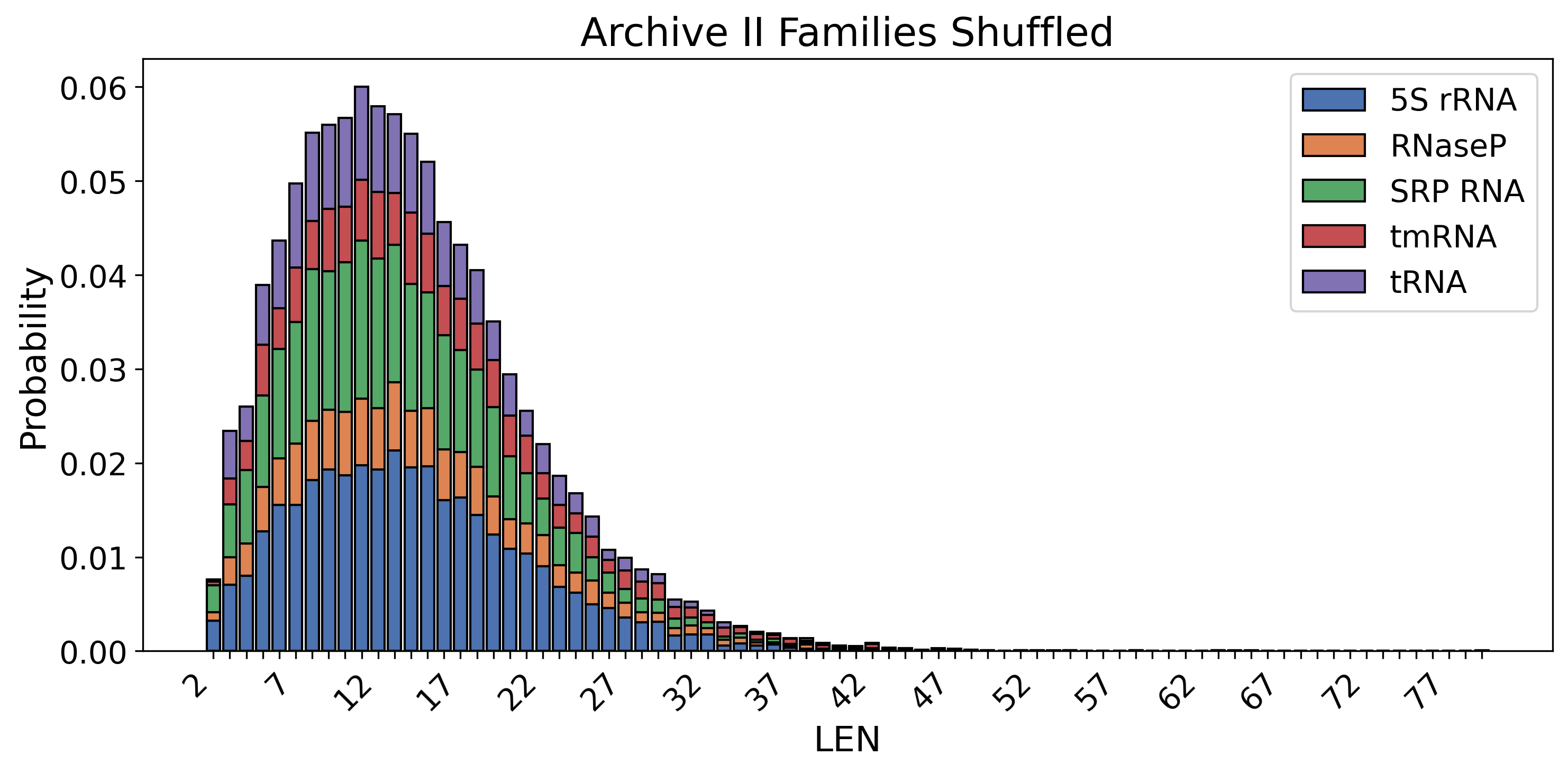}
        \includegraphics[width=0.48\textwidth]{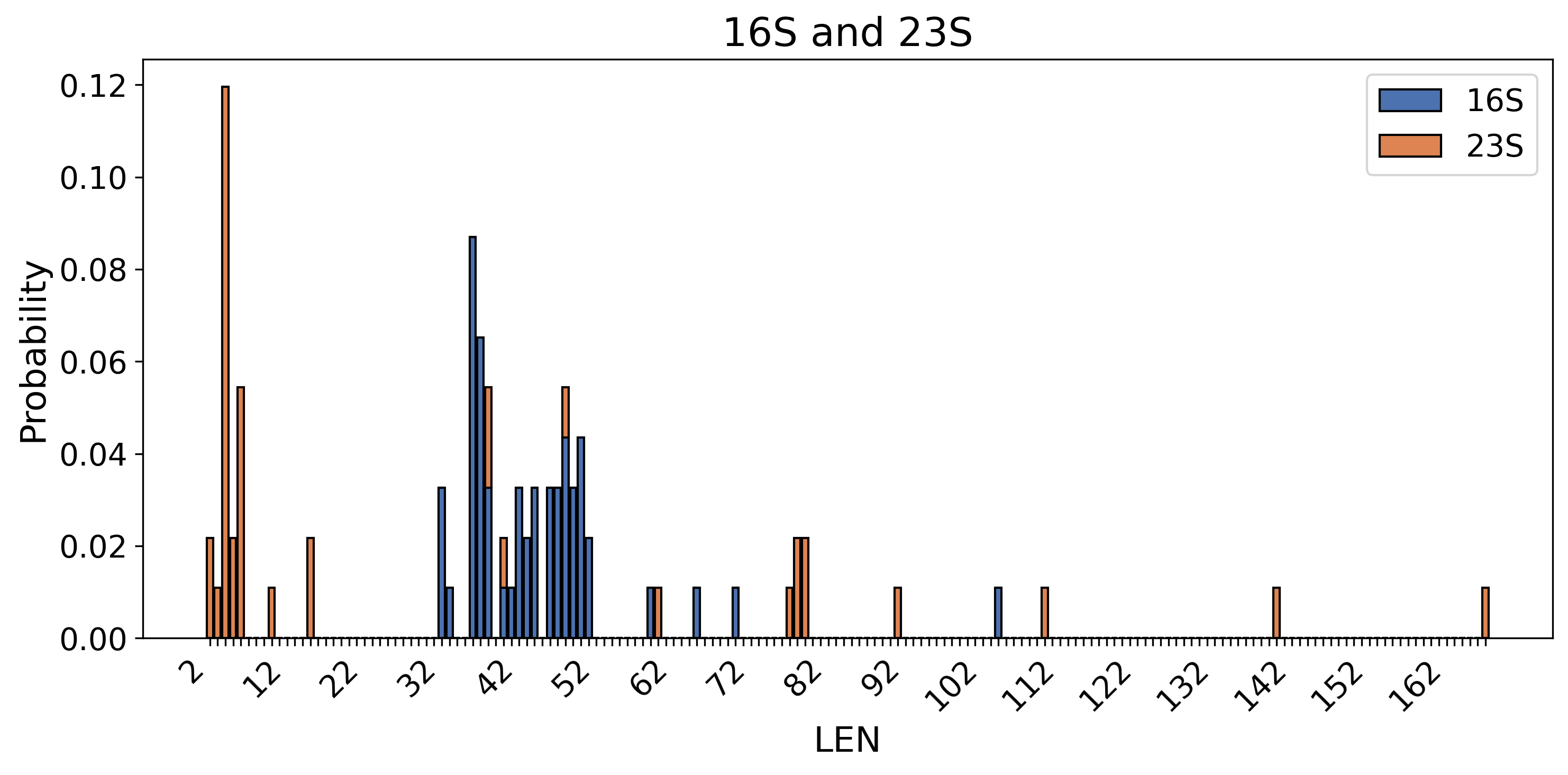}
        \includegraphics[width=0.48\textwidth]{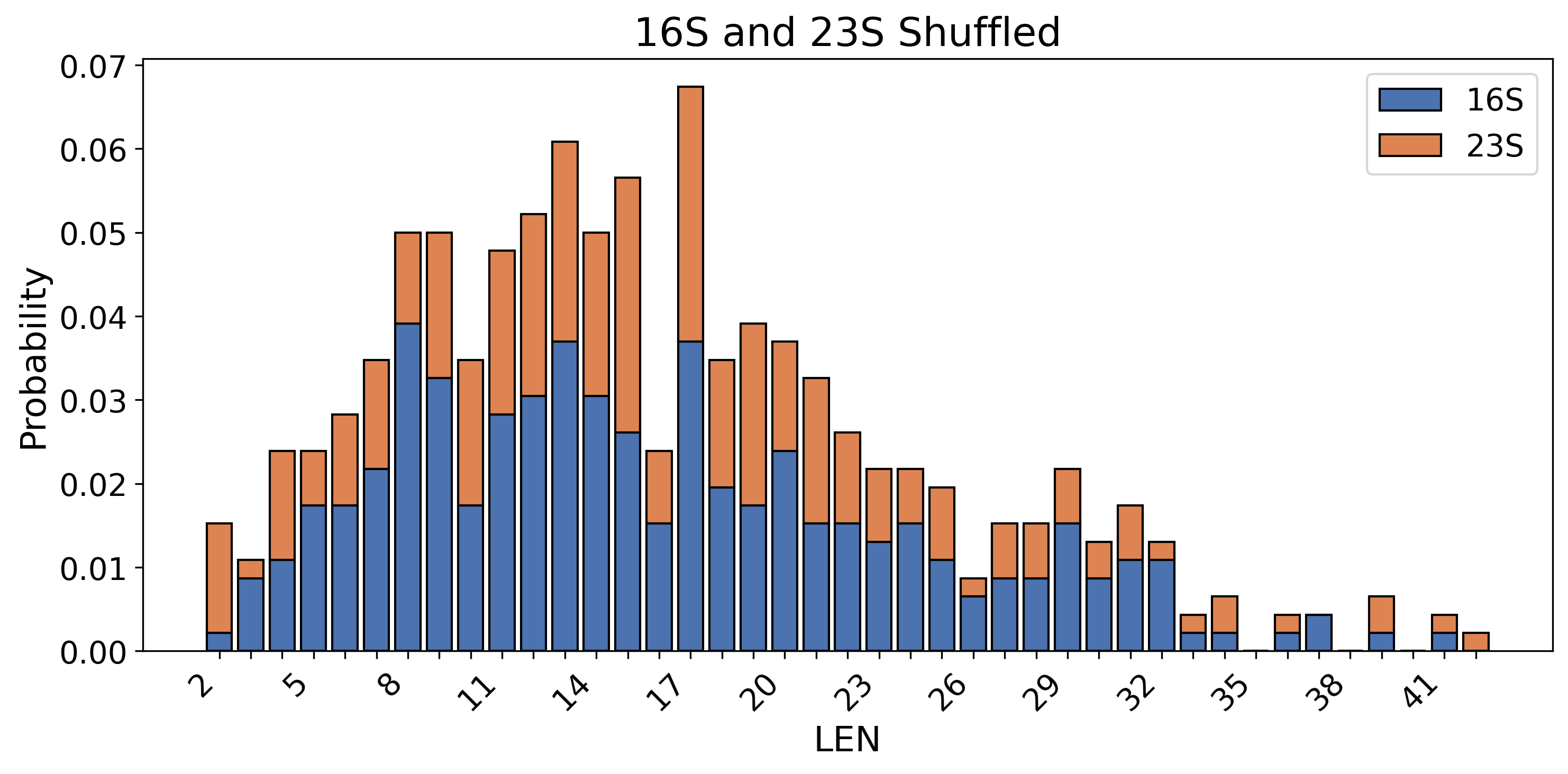}
        \includegraphics[width=0.48\textwidth]{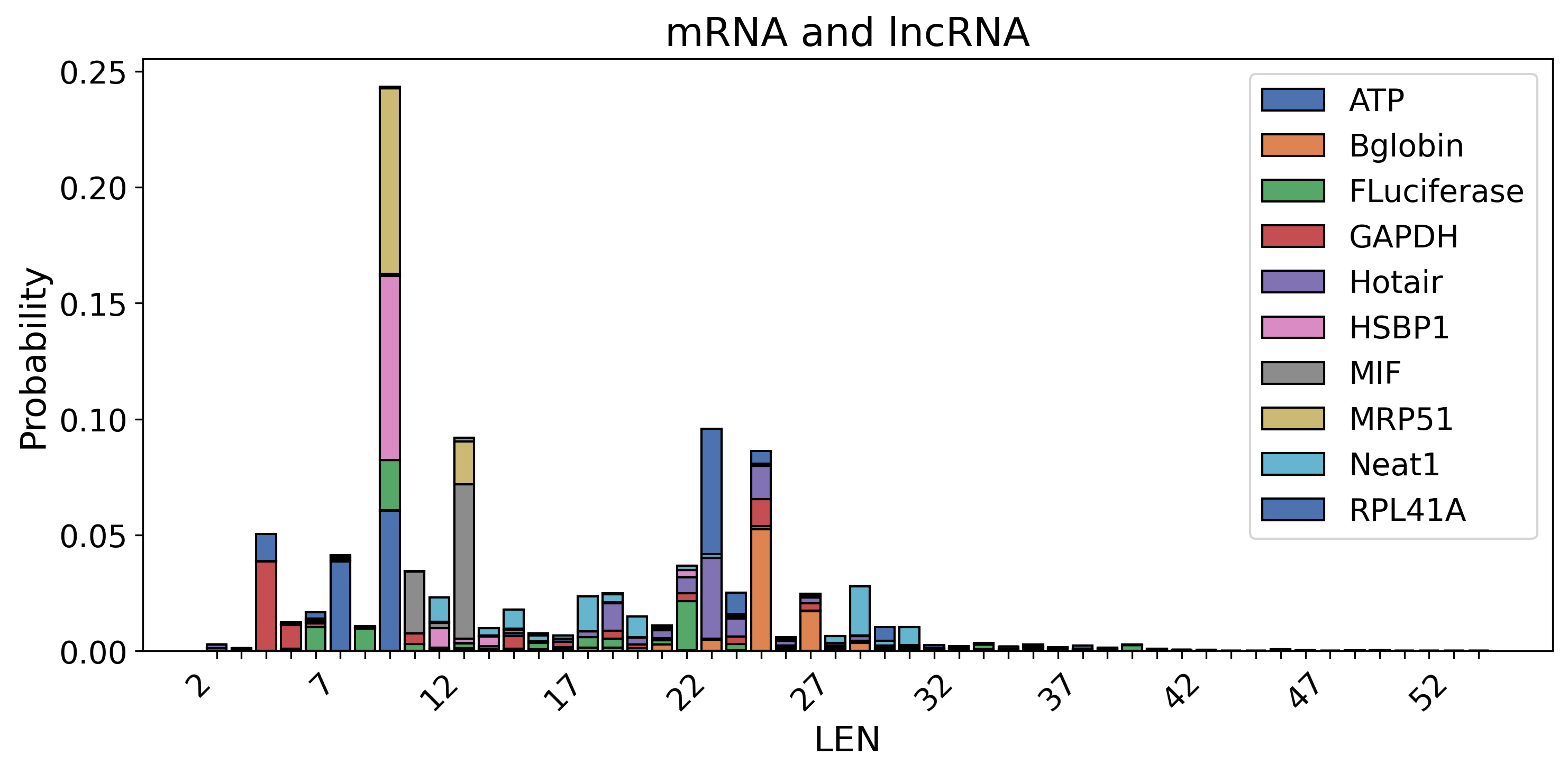}
        \includegraphics[width=0.48\textwidth]{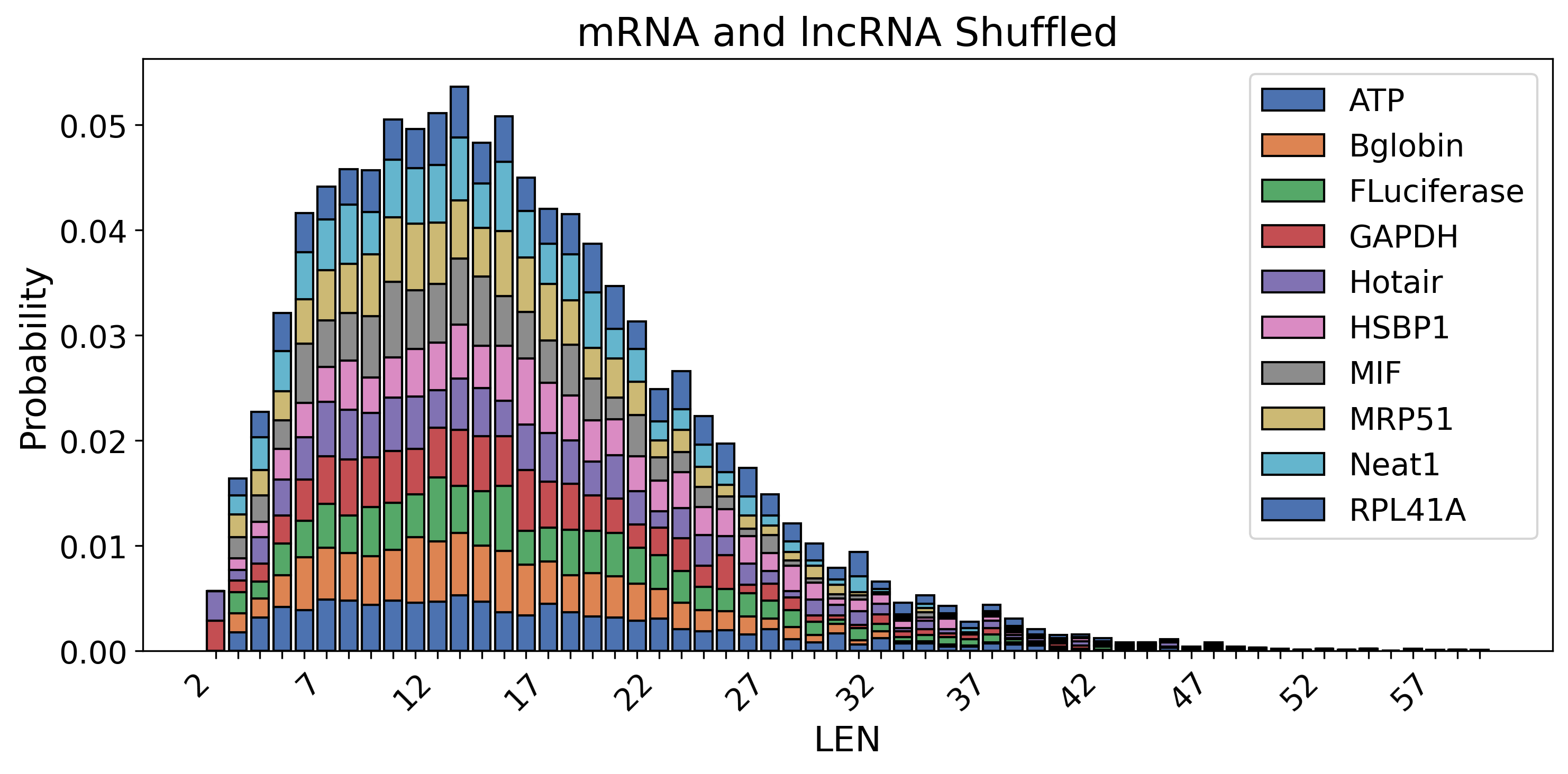}
    \end{center}
    \caption{The distribution of \LEN\ among sequences in each database.  On the left, the original sequences.  On the right, shuffled sequences.  For the ArchiveII data and the ribosomal \rna, each sequence was shuffled $5$ times and the MFE structure was calculated, while for the m\rna\ and lnc\rna, each sequence was shuffled 1000 times.  Graphs for the other parameters are in the appendix, \Cref{sec:DataGraphs}.}
    \label{fig:ArchiveIISideBySide}
\end{figure}

\subsection{Comparative RNA Web-2}

Although the ArchiveII database does include some 16S and 23S ribosomal \rna\ structural information, it is derived from and 20 distinct 16S and 5 distinct 23S sequences, each split into multiple domains.  So, to explore ribosomal \rna\ wholistically, we began with the curated collections of 16S sequences from \cite{Konings:1995} and 23S sequences from \cite{Fields:1996}.  From the sequences studied in these papers, we took all the sequences that had complete structures available in the Comparative RNA Web-2 website, which left us with 54 16S sequences and 38 23S sequences.

The 16S structures often contain a pseudoknot in the exterior loop.  Pseudoknots may provide an evolutionary strategy for increasing end-to-end distances \cite{Ermolenko:2020}.  There are several computational methods for pseudoknot removal, with no one being convention \cite{Smit:2008}.  The impact of choosing different pseudoknot removal methods on our measurements of end-proximity is unclear.  So, to assess \DEG\ and \UNP\ in this context, we computed the shortest path in each structure from the $5'$ end to the $3'$ end, and then counted unpaired bases and base pairs in this path.  We omit computing stem length \STM\ for these pseudoknotted structures.

Both \DEG\ and \UNP\ are substantially higher for this set than the other \RNA\ families we consider.  While \DEG\ means around $4$ are within the parameters of our theoretical model, the mean and variance of \UNP\ are both unusually high.  One partial explanation is that many (but not all) of the structures recorded here begin and end with long strings of unpaired nucleotides.  In contrast, the consensus structures for 16S and 23S in Rfam \cite{Griffiths:2003} tend to have many fewer unpaired nucleotides at the beginning and end of the sequences.  Nonetheless, the unpaired nucleotides play a smaller role in determining \ETE\ values, so even with the additional nucleotides, the \ETE\ means are less than $7$nm.

We again shuffled each sequence $5$ times using uShuffle and predicted MFE structures for each shuffle.  After this process, the resulting distributions for each parameter again became much closer to the theoretical Pfold distribution.  While the means and variances are generally still slightly elevated when compared to the ArchiveII shuffled data, they are now close.

\subsection{mRNA and lncRNA}

We used the ten sequences curated in \cite{Lai:2018} as representatives for m\RNA\ and lnc\RNA.  For each of these sequences, we used Rsample to create 1000 sample structures.

In contrast to the ArchiveII database, the structures had more variation in the value of \DEG, with some sequences having an average degree of nearly $4$.  While there was also more variability in \UNP, the sampled structures generally fell into a handful of structural modes.  Almost every sequence had multiple modes.  For some sequences (ATP, 443nt), the structural modes were quite similar in terms of \DEG\ and \UNP, while for others (RPL41A, 321nt), the modes varied much more.  While the differences in modes therefore cannot be explained by length alone, in general, the longer sequences had mode distinct structural modes.

Despite the increased variations in this set of \RNA, the approximation formula for \ETE\ proved to be a smoothing function, leaving all \ETE\ means between $2.6$nm and $4.8$nm, with relatively low variances.  All \ETE\ mean values were within reasonable ranges of the theoretical Pfold distribution mean.  The primary distinguishing factor is that the variances were much lower.

We then shuffled each sequence 1000 times and computed the MFE structures with RNAfold.  As with the ArchiveII dataset, the \ETE\ means moved towards the mean in the theoretical Pfold distribution and the variances increased.  This suggests that for \ETE, the variance could be signal distinguishing between actual \RNA\ sequences (lower variance) and randomized ones (higher variance).

\section{Conclusion} \label{sec:Conclusion}

The multivariate analytic combinatorics pipeline used here has enabled us to fully characterize the distributions of several parameters of \rna\ secondary structures. While the means of other characteristics of \rna\ have been investigated previously \cite{Hofacker:1998}, the pipeline here may bring richer information about these parameters.

As seen in \Cref{sec:Data}, some \rna\ structures have shorter first helices but longer first stems.  Refining the combinatorial specifications for the Pfold grammar could yield more information on stems and whether their length accounts for the large number of \rna\ sequences with degree one exterior loops.  Initial testing shows that the additional symbols needed to track stem length greatly increase the computational complexity in the corresponding system of equations defining the GF.

As is often the case with research at the interface of mathematics and biology, the results here tie back to other well-studied combinatorial objects.  In particular, Dyck paths can be encoded as pattern-avoiding permutations~\cite{Krattenthaler:2001, Adin:2018}, and the results here give distributional information on the blocks contained within a pattern.  The results on the degree of the Motzkin paths connects to the number of blocks in a \emph{Motzkin permutation}~\cite{Elizalde:2005}.  Hence, the results here may have additional implications on properties of pattern-avoiding permutations.

Finally, the recurrence of the negative binomial distribution brings an obvious question: what does a coin flip represent combinatorially?  If each tails represents increasing a measure of end-proximity by one, then is there a corresponding interpretation for a heads?

\section*{Acknowledgments} The authors thank David Mathews for providing the m\rna\ and lnc\rna\ sequences, and Yuta Hozumi for providing the forced structures for the ArchiveII dataset. C.H. was partially supported by NIH R01GM126554, and by the NSF-Simons Southeast Center for Mathematics and Biology (SCMB) through
NSF DMS \#1764406 and Simons Foundation/SFARI 594594.

\section*{Data Availability} The Archive II dataset \cite{Sloma:2016} and the 16S and 23S sequences \cite{Cannone:2002, Chan:2023} are available online.  The lnc\RNA\ and m\RNA\ sequences from \cite{Lai:2018} were obtained via personal communication with the authors.  Verifications of our theoretical computations are available in a companion {\tt SageMath} worksheet: \url{https://github.com/gtDMMB/EndProximity}

\newpage

\begin{appendices}

\section{\ETE\ for Motzkin paths} \label{sec:MotzkinETE}
Here, we complete the proof of \Cref{cor:Motzkin} by computing the mean and variance of \ETE.
\begin{proof}
    Given the limiting bivariate PGF $M_p(u, v)$ from \cref{lem:MotzkinPaths}, we first find an explicit formula for the probability mass function for the joint distribution.  To do so, we expand $M_p(u, v)$ as a series using that $1/(1-x)^2 = \sum_{n = 0}^\infty (n + 1)x^n$:
    \begin{align*}
        \frac{v}{(3-u-v)^2} &= \frac{v}{9} \left(\frac{1}{1 - \frac{u + v}{3}}\right)^2\\
        &= \frac{v}{9} \sum_{n = 0}^\infty (n + 1) \left(\frac{u + v}{3}\right)^n\\
        &= \sum_{n = 0}^\infty \sum_{k = 0}^n \binom{n}{k} \frac{n + 1}{3^{n + 2}} u^k v^{n - k + 1}\\
        &= \sum_{i = 0}^\infty \sum_{j = 0}^\infty \binom{i + j - 1}{i} \frac{i + j}{3^{i + j + 1}} u^i v^j
    \end{align*}
    Collectively, this tells us:
    \[
    p_{i, j} := \mathbb{P}(\UNP = i, \DEG = j) = \binom{i + j - 1}{i} \frac{i + j}{3^{i + j + 1}}.
    \]
    With the approximate distance formula $\ETE = \sqrt{1.5^2 \cdot \DEG^{6/5} + 0.62^2 \cdot \COV^{6/5}}$, we can now numerically compute $\mathbb{E} (\ETE)$ and $\mathbb{E} (\ETE)^2$.  First,
    \begin{equation} \label{eq:ETE-E-Motzkin}
    \mathbb{E}(\ETE) = \sum_{i=0}^\infty \sum_{j=0}^\infty \sqrt{1.5^2 \cdot j^{6/5} + 0.62^2 \cdot (i+j-1)^{6/5}} \cdot p_{i, j}
    \end{equation}
    We claim that summing the terms through $i + j = 50$ yields the expected value to within $0.0025$.  To see this, we give a rough bound the tail of this sum: 
    \[
    T_k := \sum_{i + j \geq k} p_{i, j}\sqrt{1.5^2 \cdot j^{6/5} + 0.62^2 \cdot (i+j-1)^{6/5}}
    \]
    Our goal is to compare $T_k$ to the tail of series that are derivatives of geometric series, because these tails can be computed explicitly.
    
    First, note that $\binom{i + j - 1}{i} \leq 2^{i + j - 1}$ for all $i$ and $j$.  Thus, $p_{i, j} \leq (2/3)^{i + j + 1}(i + j).$  Additionally, from the triangle inequality,
    \[
    \ETE \leq 1.5 (\DEG)^{3/5} + 0.62(\COV)^{3/5} \leq 2.12(\DEG + \UNP)^{3/5}.
    \]
    Combining these two inequalities shows that
    \begin{align*}
    T_k &\leq \sum_{i + j \geq k} (2.12)(i + j)^{8/5} \left(\frac{2}{3}\right)^{i + j - 1} = 2.12\sum_{n = k}^\infty (n + 1)n^{8/5}\left(\frac{2}{3}\right)^{n-1} \\
    &\leq 2.12 \cdot 3/2 \cdot\sum_{n = k}^\infty n^3 \left(\frac{2}{3}\right)^{n},
    \end{align*}
    where the last inequality is true for $k \geq 3$ because $n + 1 \leq n^{7/5}$ whenever $n \geq 3$. Temporarily replace $2/3 = x$.  Then, using common geometric series formulas,
    \begin{align*}
    2.12 \cdot 3/2 \cdot\sum_{n = k}^\infty n^3 x^{n} &= 2.12 \cdot 3/2 \cdot \sum_{n = 0}^\infty (n + k)^3 x^{n + k}\\
    &= 2.12 \cdot 3/2 \cdot x^k \cdot \left[\sum_{n = 0}^\infty n^3 x^n + 3k \sum_{n = 0}^\infty n^2 x^n + 3k^2 \sum_{n = 0}^\infty nx^n + k^3 \sum_{n = 0}^\infty x^n \right]\\
    &= 2.12 \cdot 3/2 \cdot x^k \cdot \left[ \frac{k^3}{1 - x} + \frac{3k^2x}{(1-x)^2} + \frac{3kx(1 + x)}{(1-x)^3} + \frac{x(1 + 4x + x^2)}{(1-x)^4}\right]
    \end{align*}
    Replacing $x = 2/3$ once more yields
    \[
    T_k \leq 2.12 \cdot 3/2 \cdot\sum_{n = k}^\infty n^3 x^{n} \leq \left(\frac{2}{3}\right)^k \left(3k^3 + 18k^2 + 90k + \frac{86}{3}\right) \leq 2.12 \cdot 3/2 \cdot \left(\frac{2}{3}\right)^k \cdot 4k^3
    \]
    where the last equality is true for $k \geq 23$.  It is easy to verify that this bound is decreasing for $k > 8$, and for $k = 50$, is less than $0.0025$.  Thus, it suffices to sum \Cref{eq:ETE-E-Motzkin} over all $i, j \leq 50$ to achieve the mean to within $0.0025$.

    For the second moment,
    \[
    \mathbb{E}(\ETE)^2 = \sum_{i=0}^\infty \sum_{j=0}^\infty \left(1.5^2 \cdot j^{6/5} + 0.62^2 \cdot (i+j-1)^{6/5}\right) \cdot p_{i, j},
    \]
    and a similar argument shows that summing among terms where $i + j \leq 65$ gives an error less than $0.002$.  Hence, we can calculate the expressed variance using the shortcut formula $\mathbb{E}(\ETE)^2 - [\mathbb{E}(\ETE)]^2$ with an appropriate increase in precision for the estimate of $\mathbb{E}(\ETE)$ to ensure the error of $(\mathbb{E}(\ETE))^2$ remains small.
\end{proof}

\newpage

\section{Graphs of distributions of parameters in datasets}\label{sec:DataGraphs}

Here, we provide stacked bar graphs in the style of \Cref{fig:ArchiveIISideBySide} illustrating the distributions of each parameter in each dataset before and after shuffling the sequences.

 \subsection{\DEG}

 \begin{center}
     \includegraphics[width=0.48\textwidth]{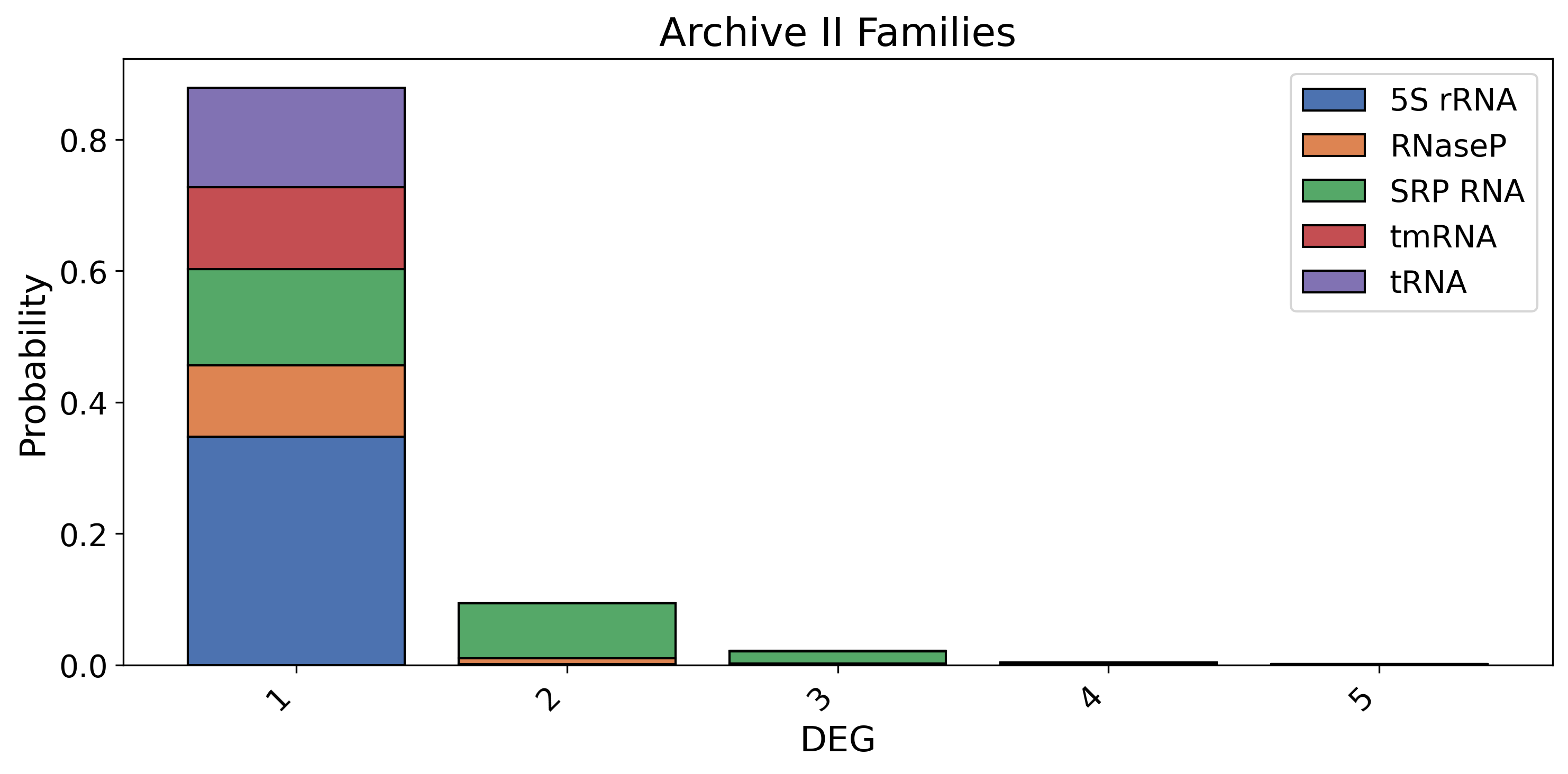}
     \includegraphics[width=0.48\textwidth]{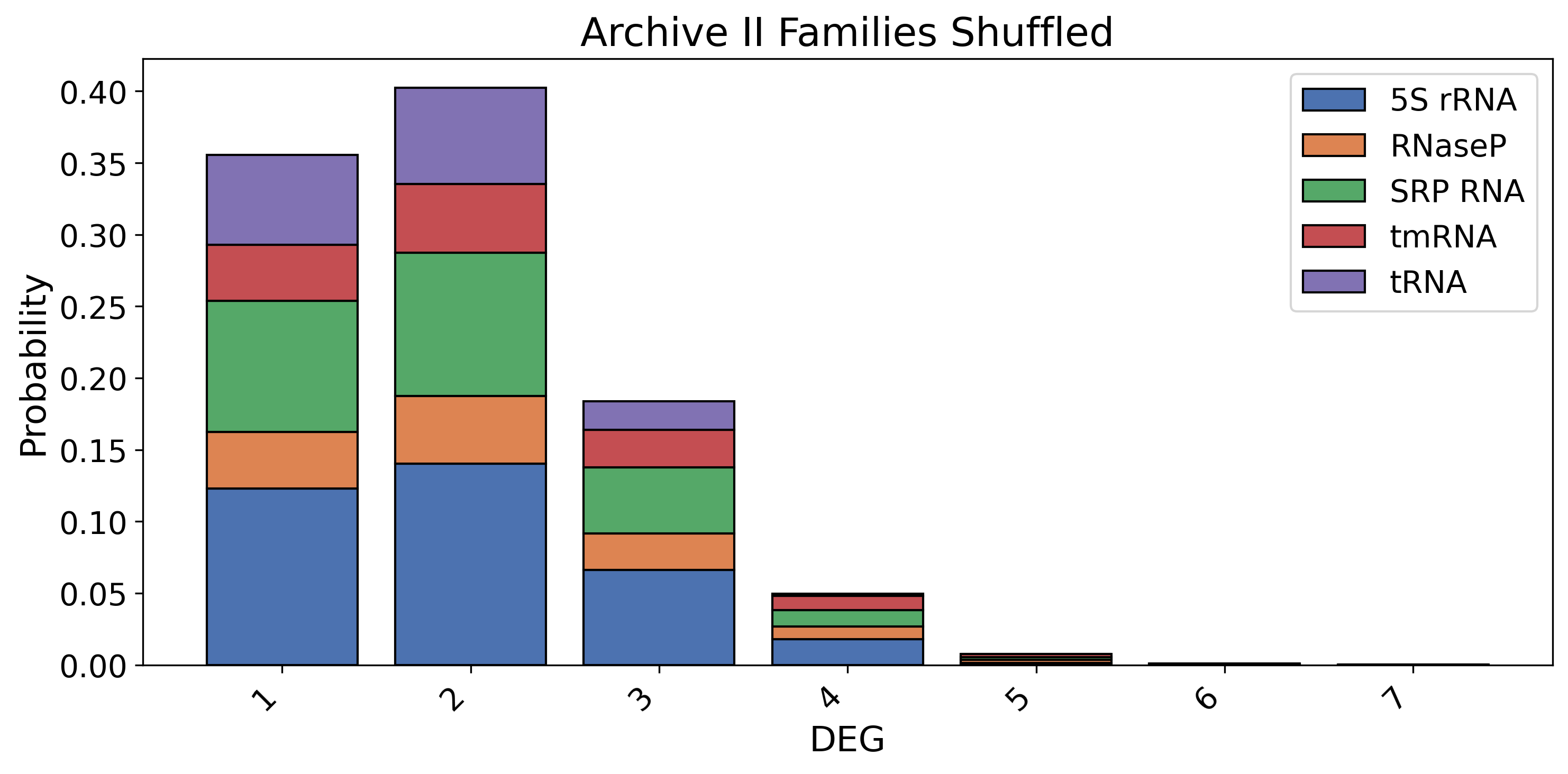}
     \includegraphics[width=0.48\textwidth]{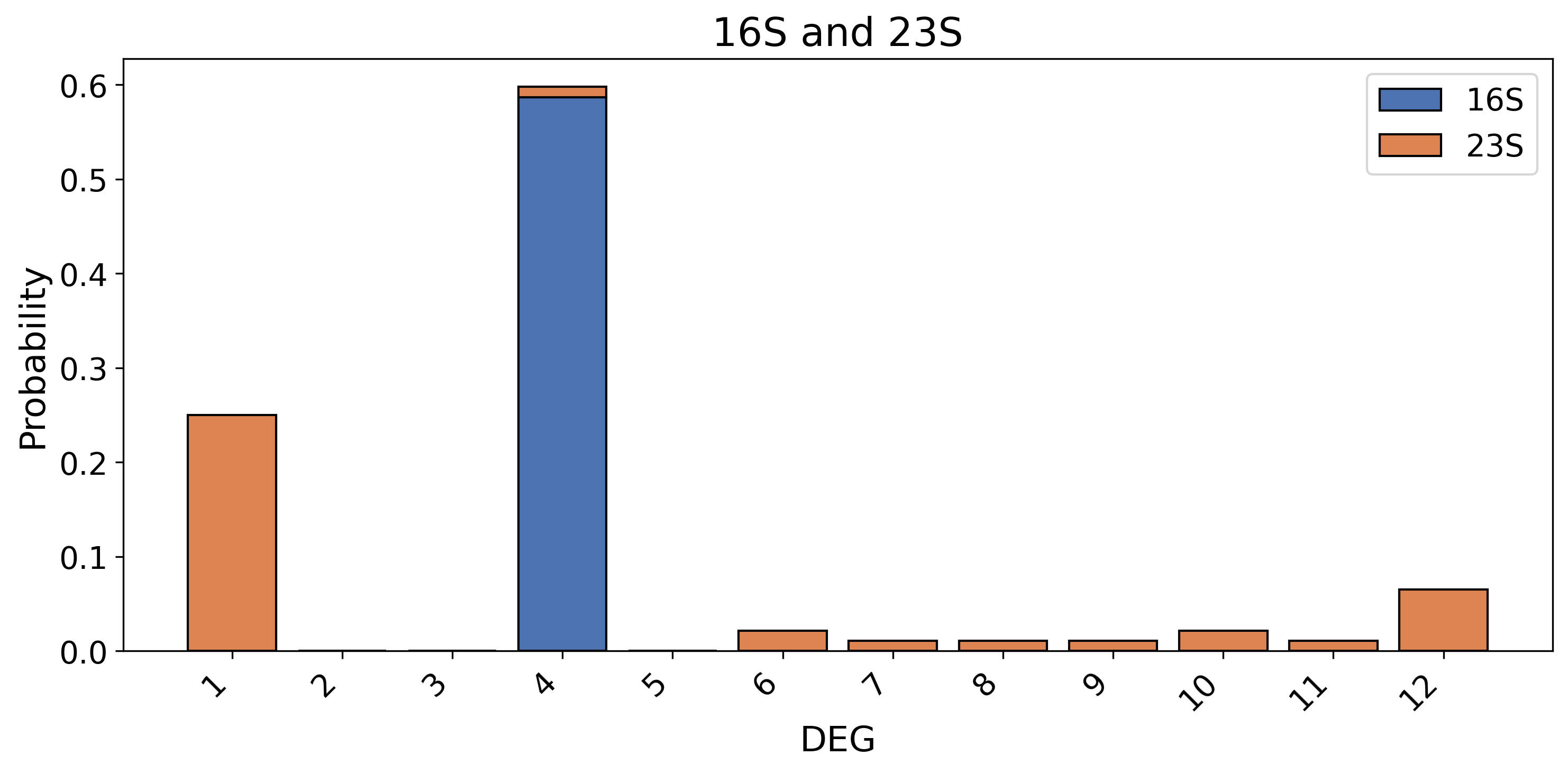}
     \includegraphics[width=0.48\textwidth]{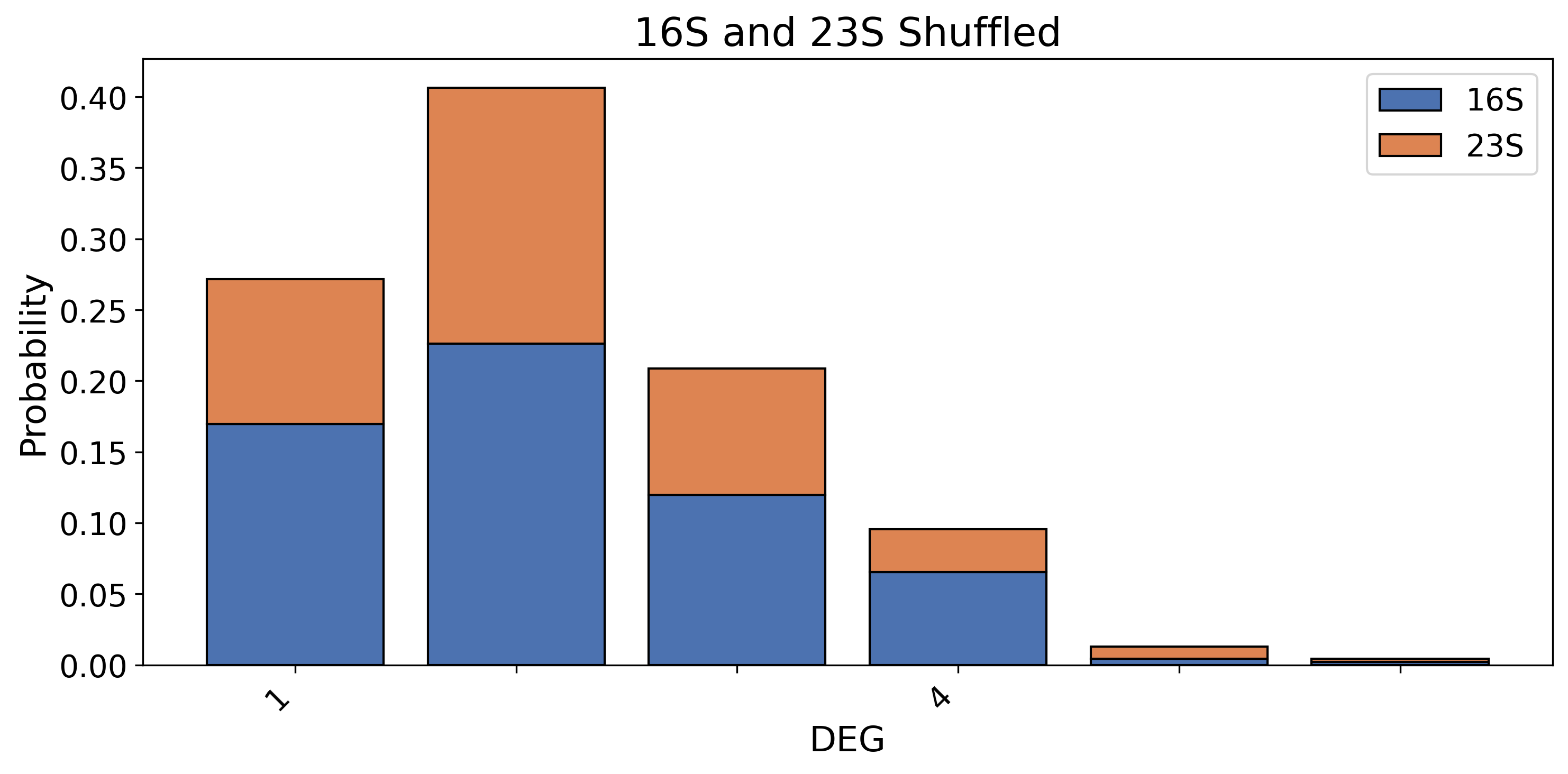}
     \includegraphics[width=0.48\textwidth]{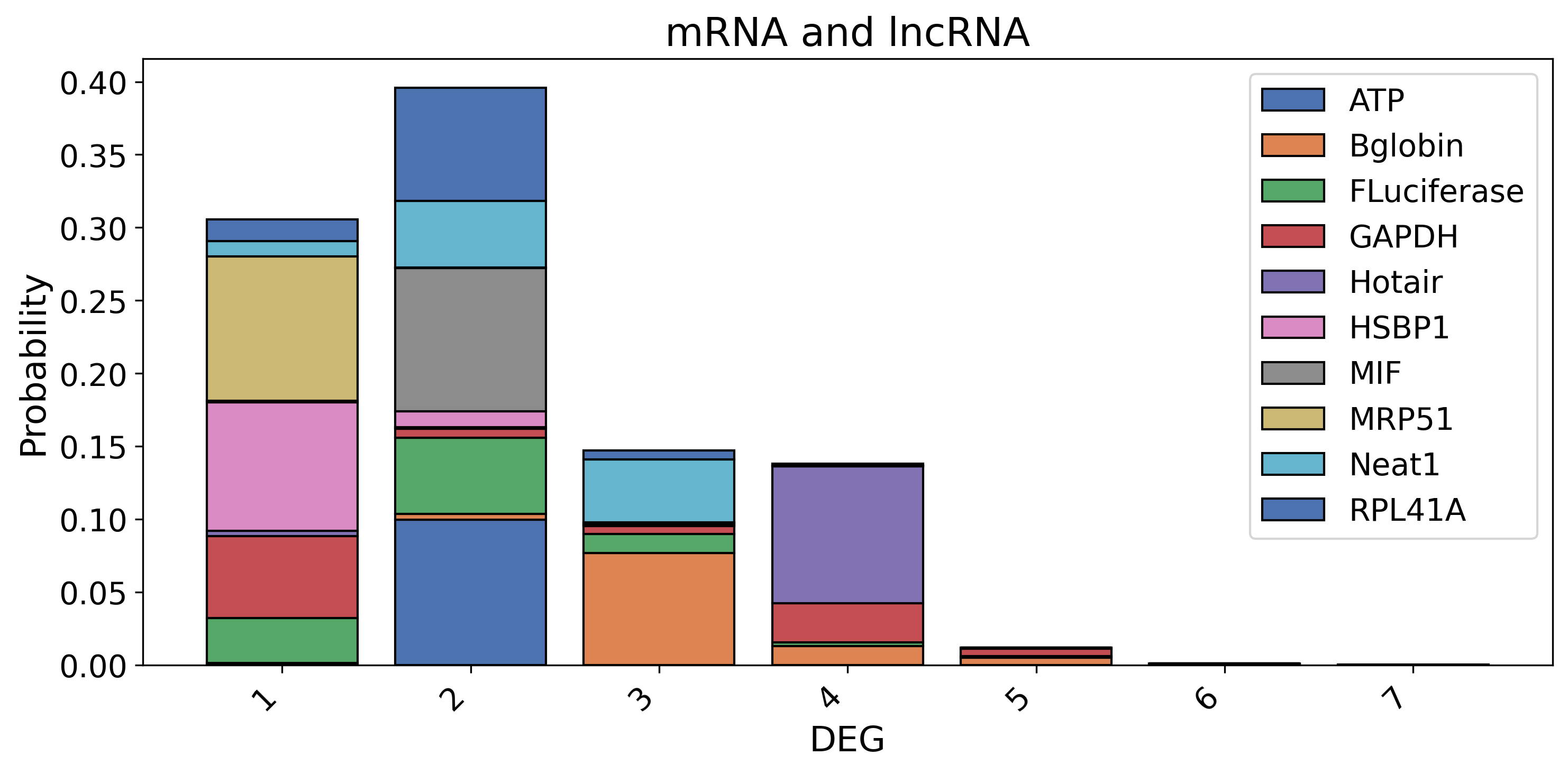}
     \includegraphics[width=0.48\textwidth]{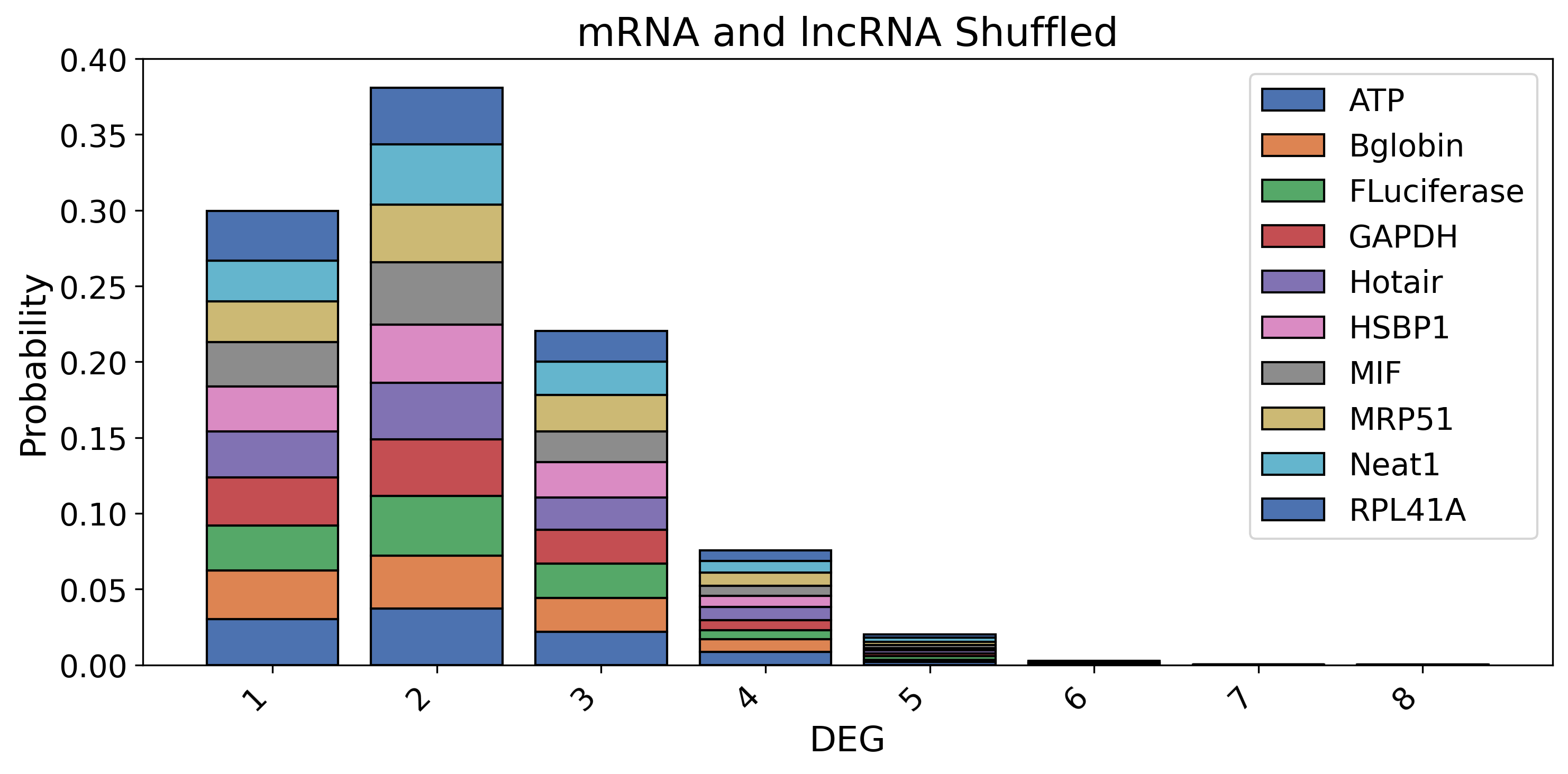}
 \end{center}

 \subsection{\UNP}

 \begin{center}
     \includegraphics[width=0.48\textwidth]{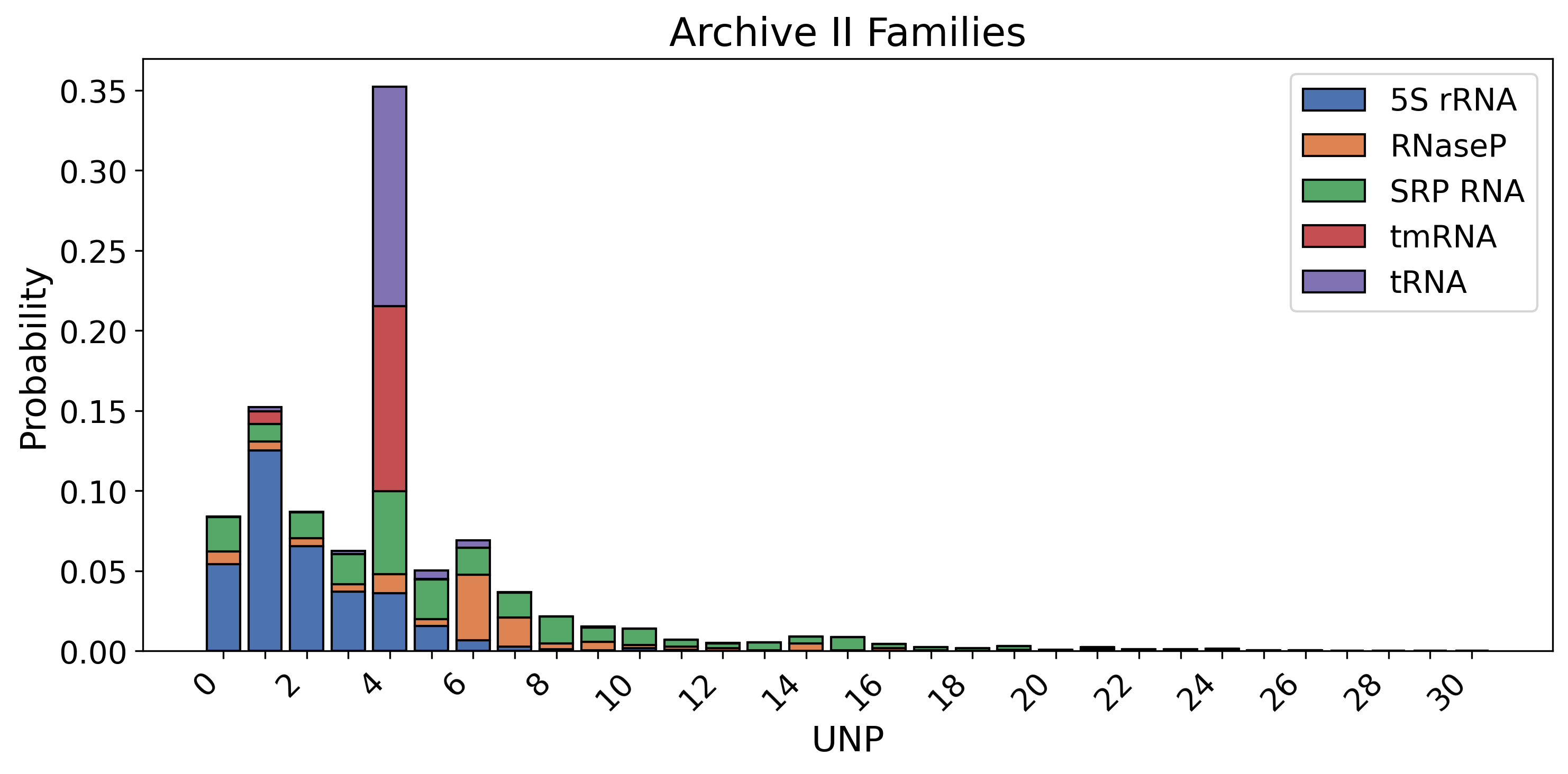}
     \includegraphics[width=0.48\textwidth]{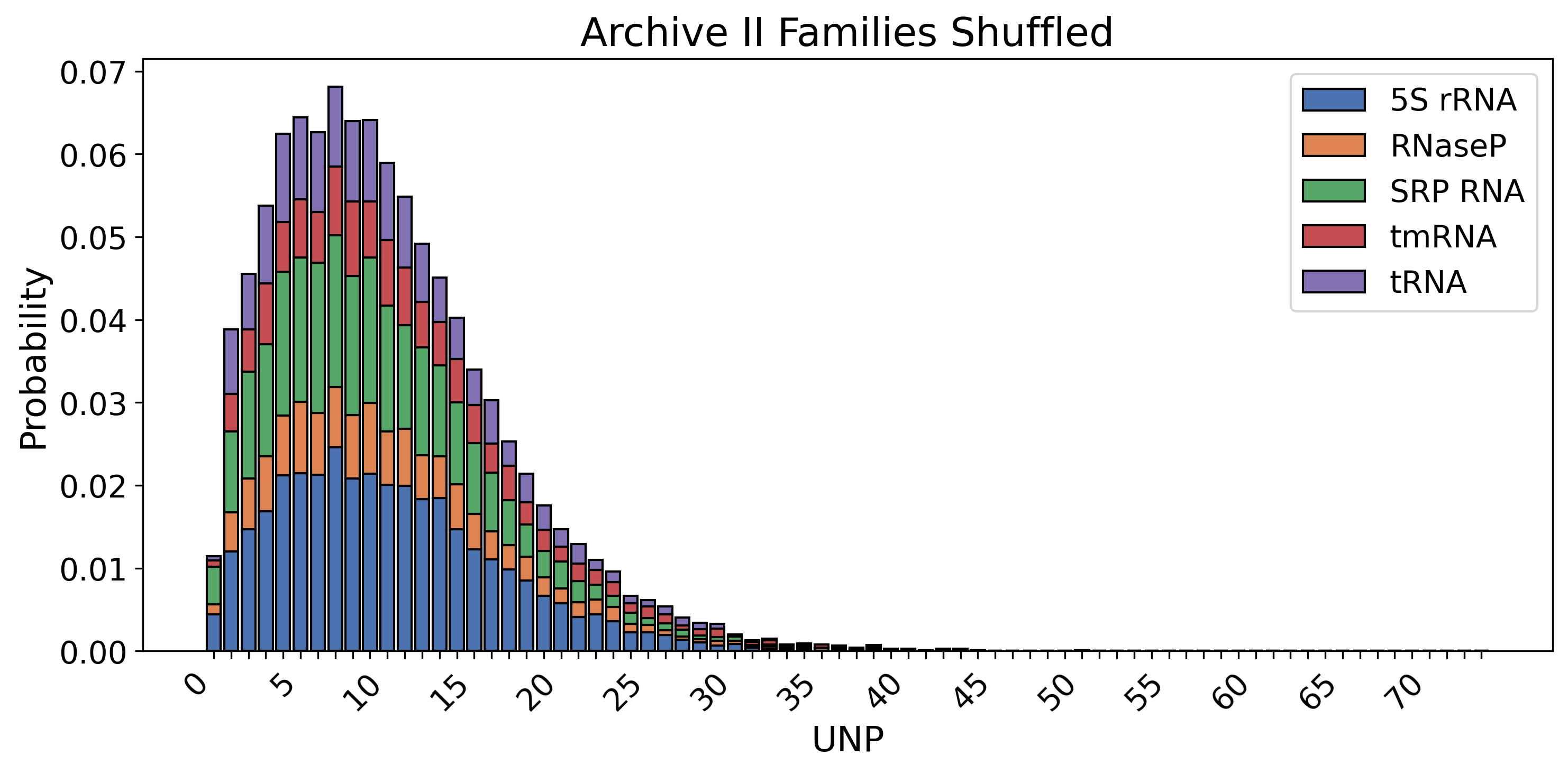}
     \includegraphics[width=0.48\textwidth]{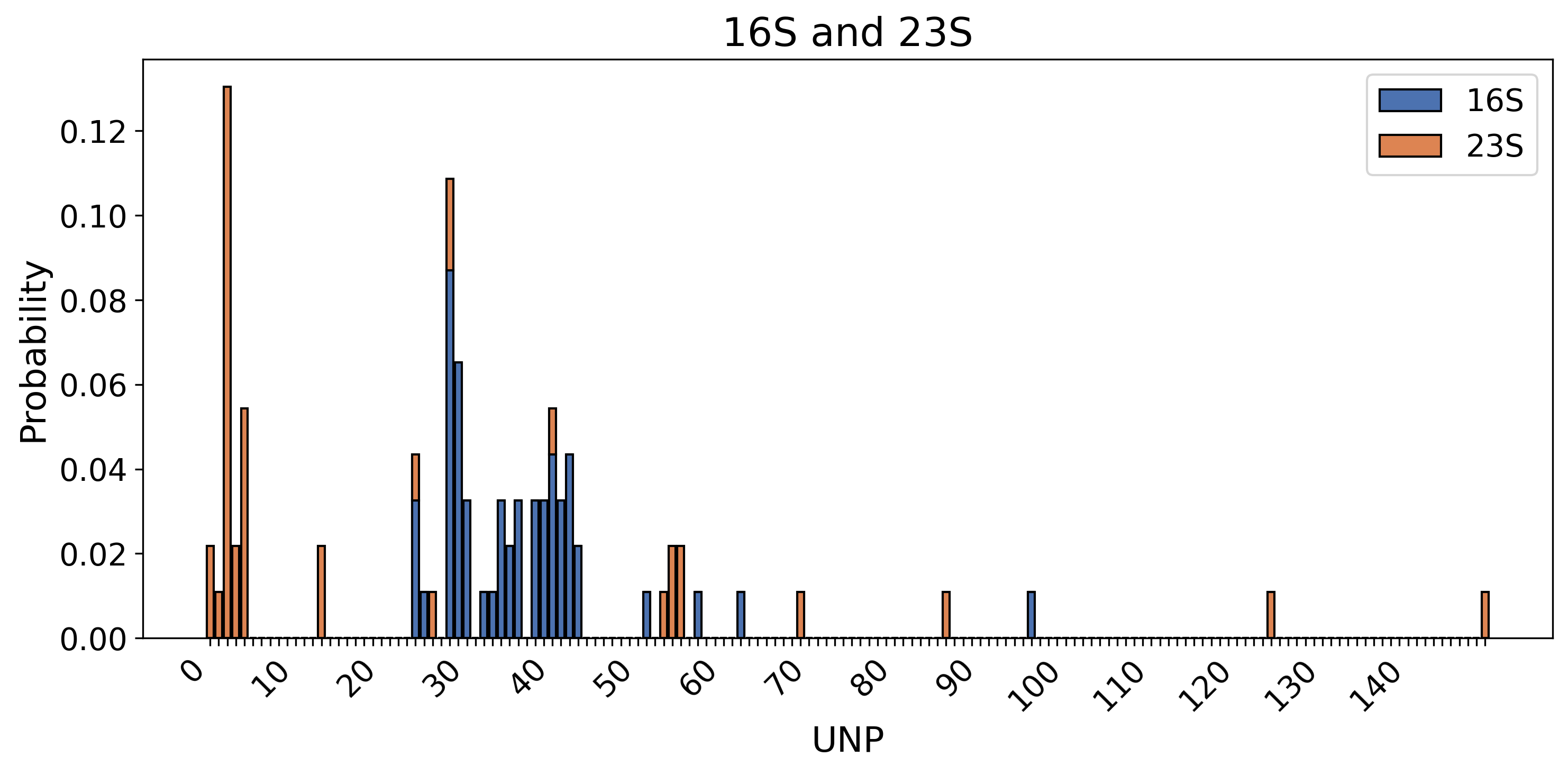}
     \includegraphics[width=0.48\textwidth]{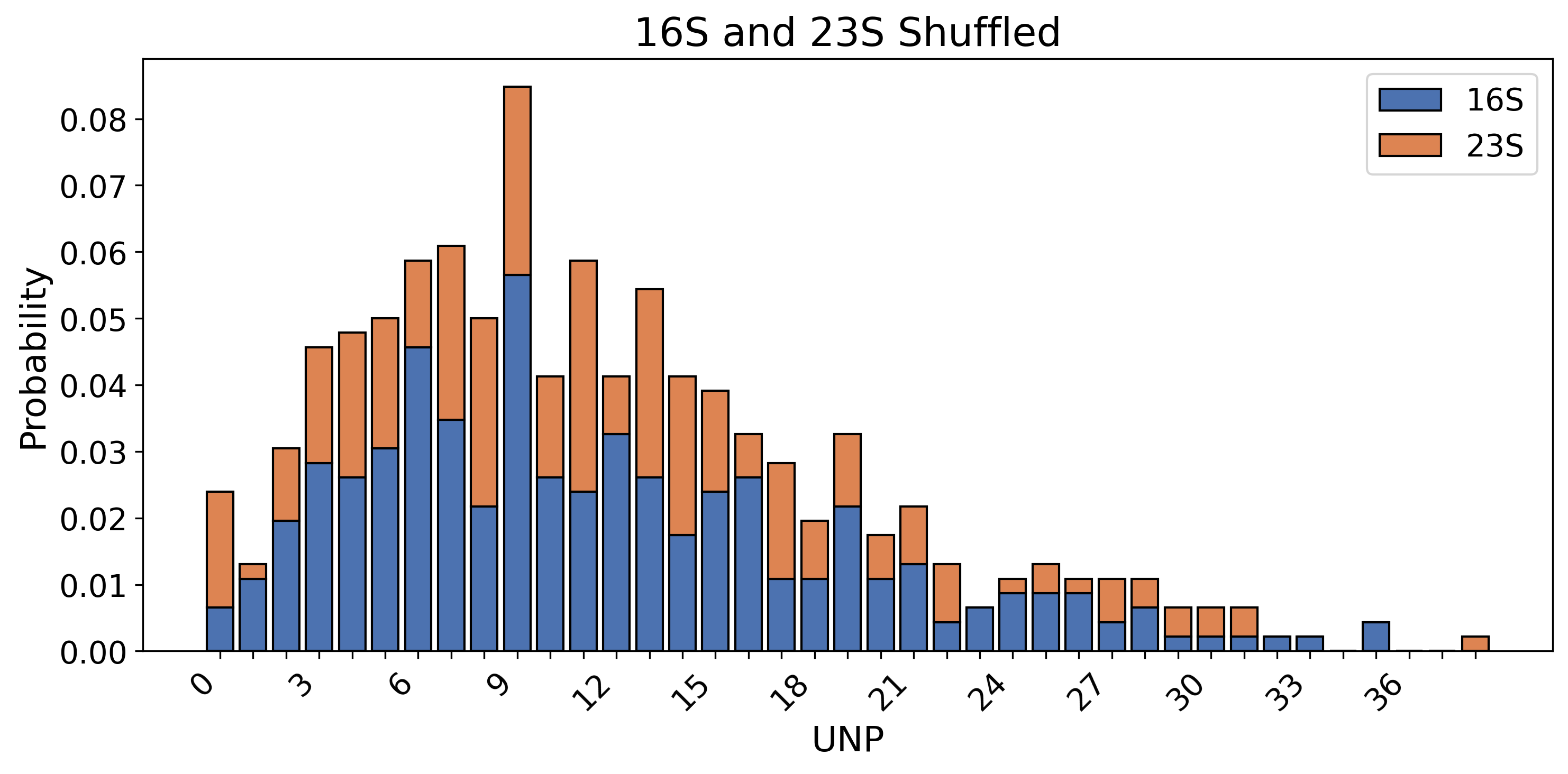}
     \includegraphics[width=0.48\textwidth]{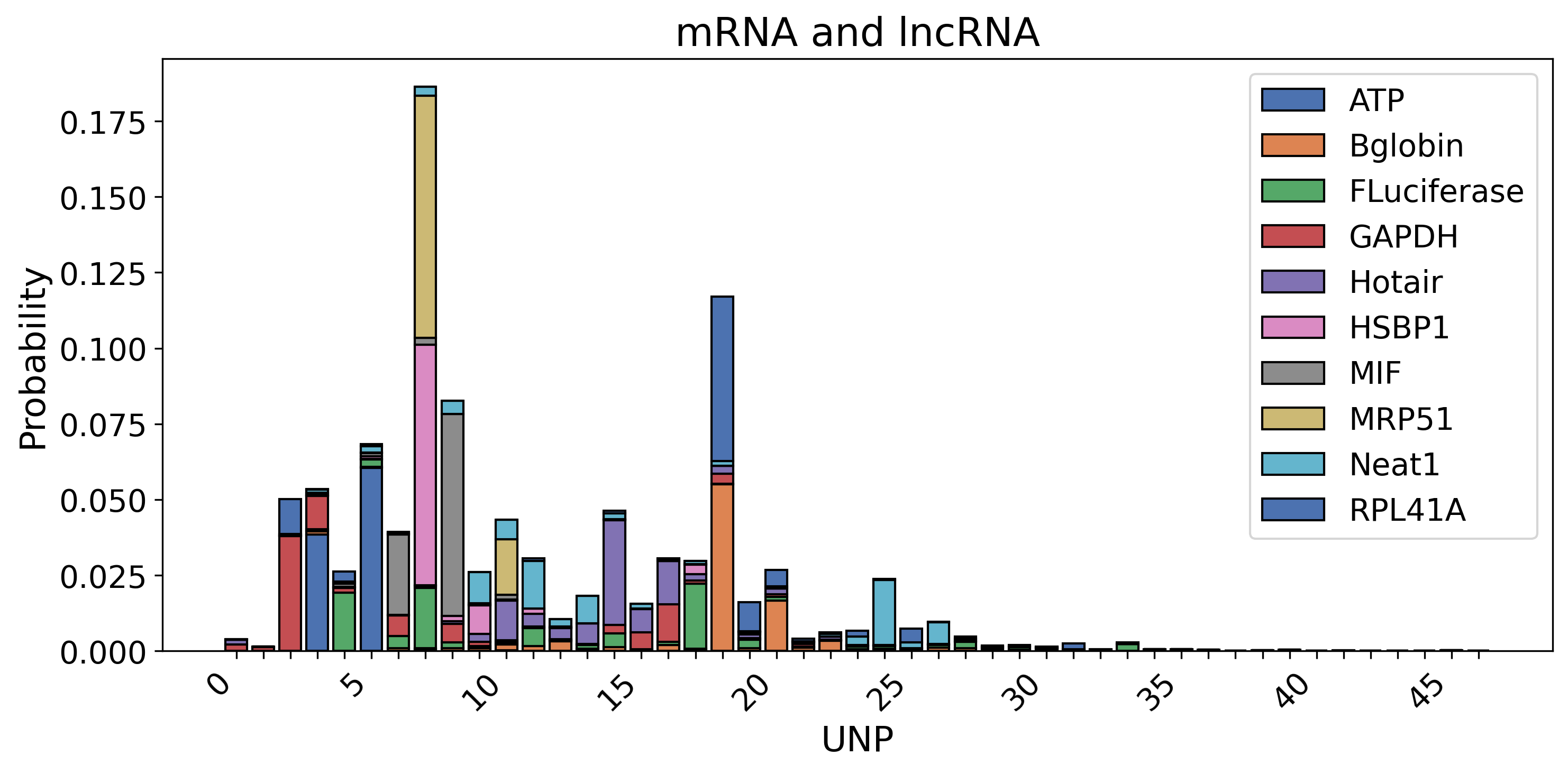}
     \includegraphics[width=0.48\textwidth]{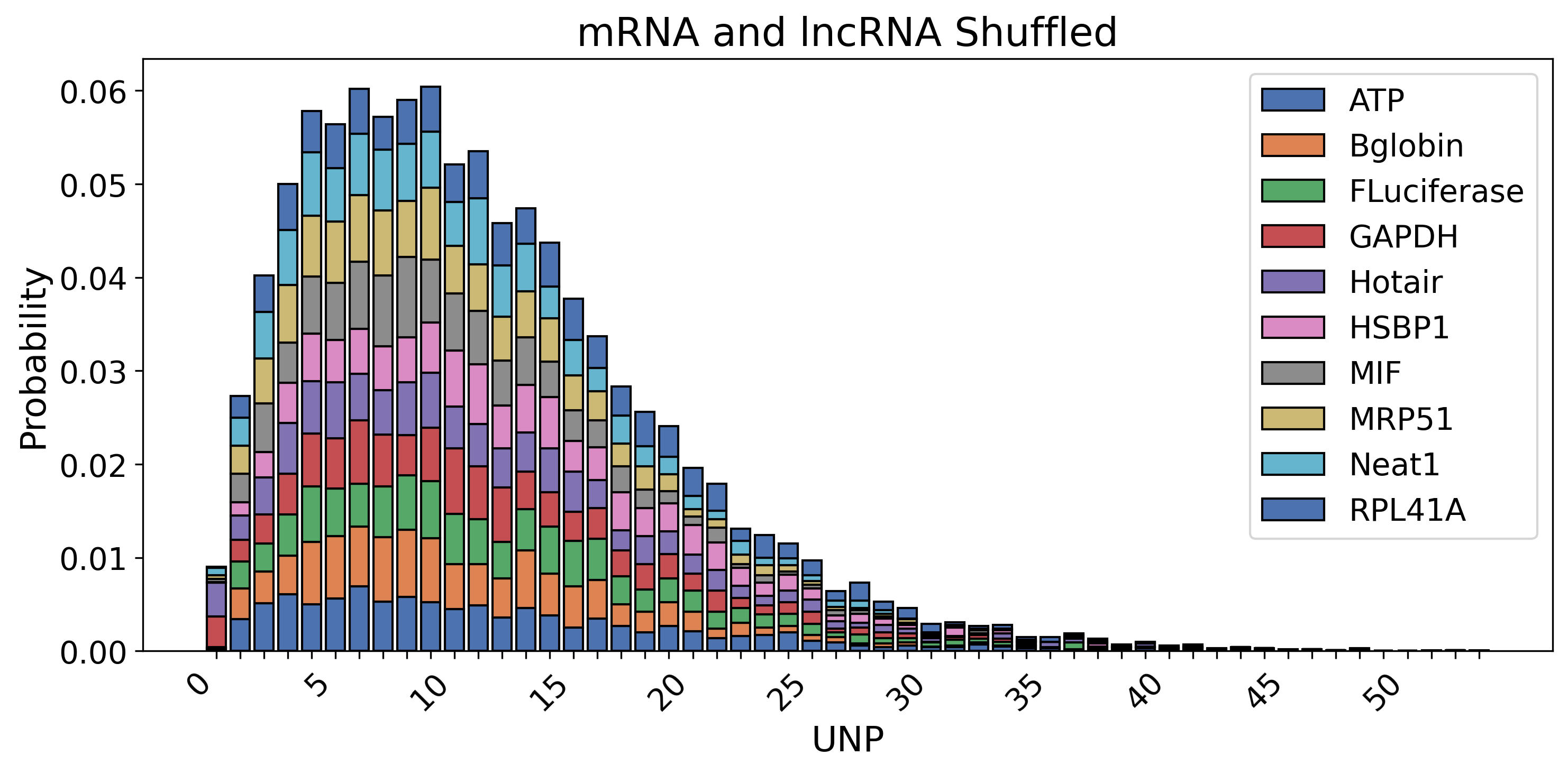}
 \end{center}

 \subsection{\COV}
 \begin{center}
     \includegraphics[width=0.48\textwidth]{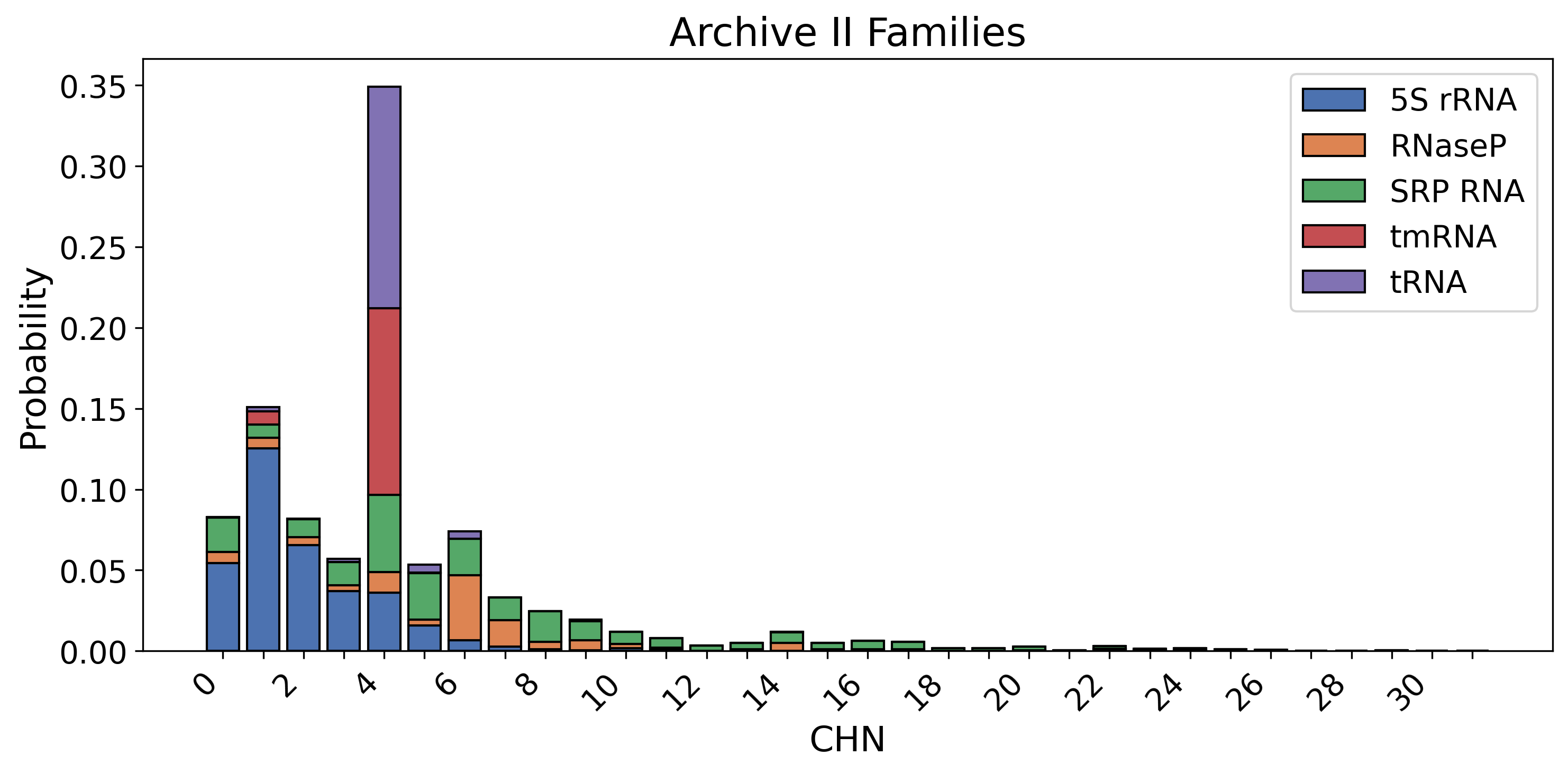}
     \includegraphics[width=0.48\textwidth]{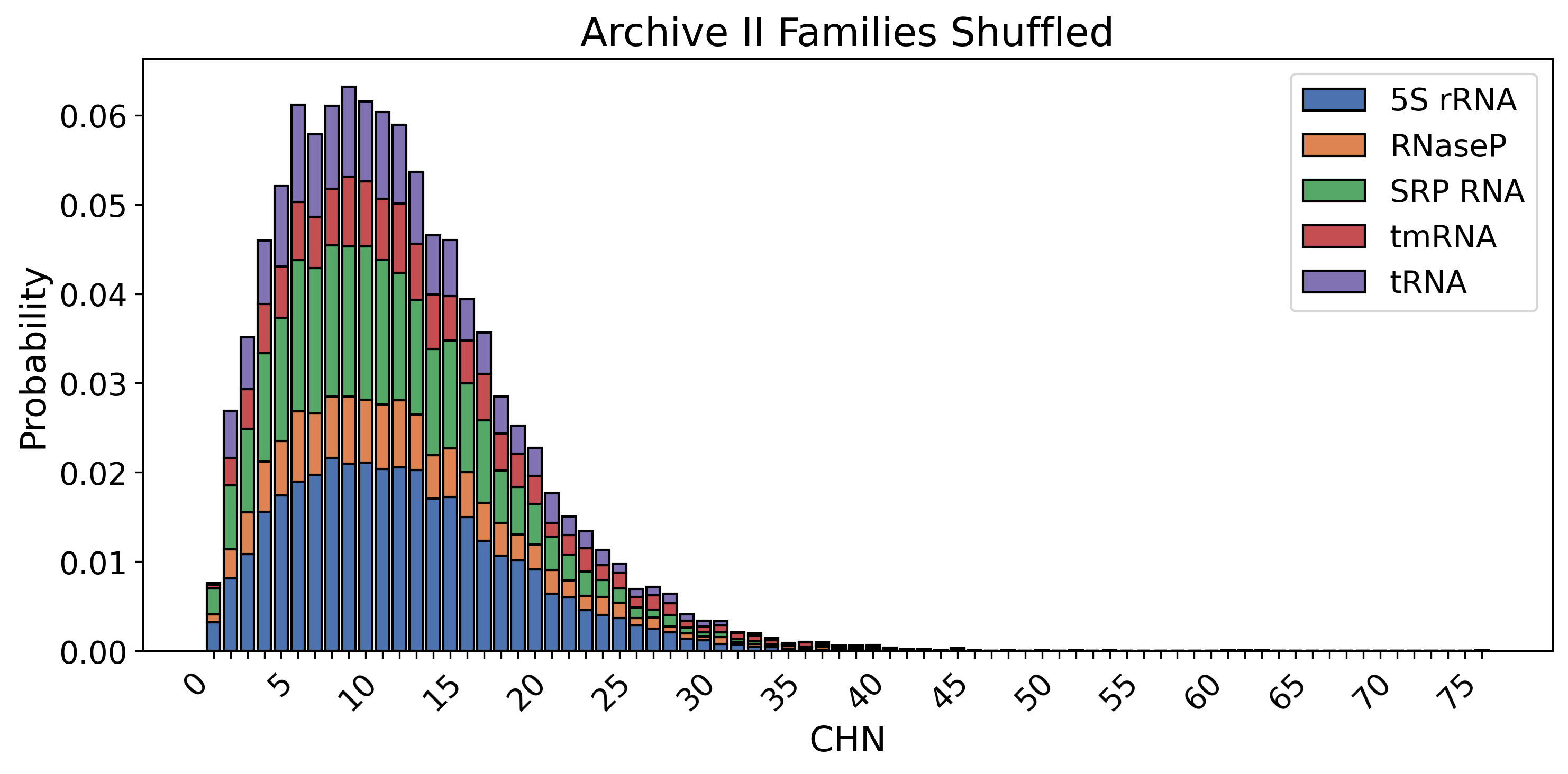}
     \includegraphics[width=0.48\textwidth]{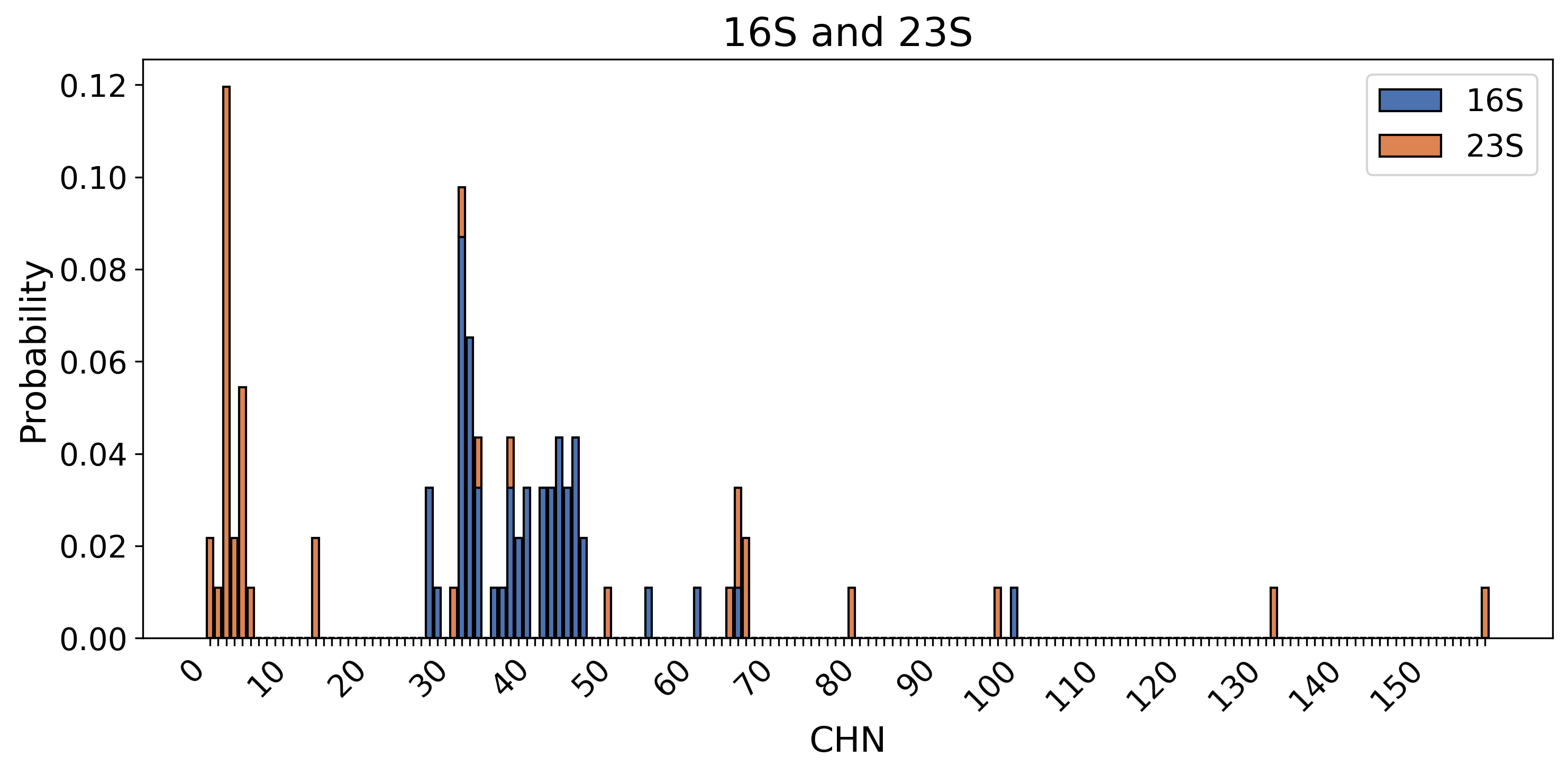}
     \includegraphics[width=0.48\textwidth]{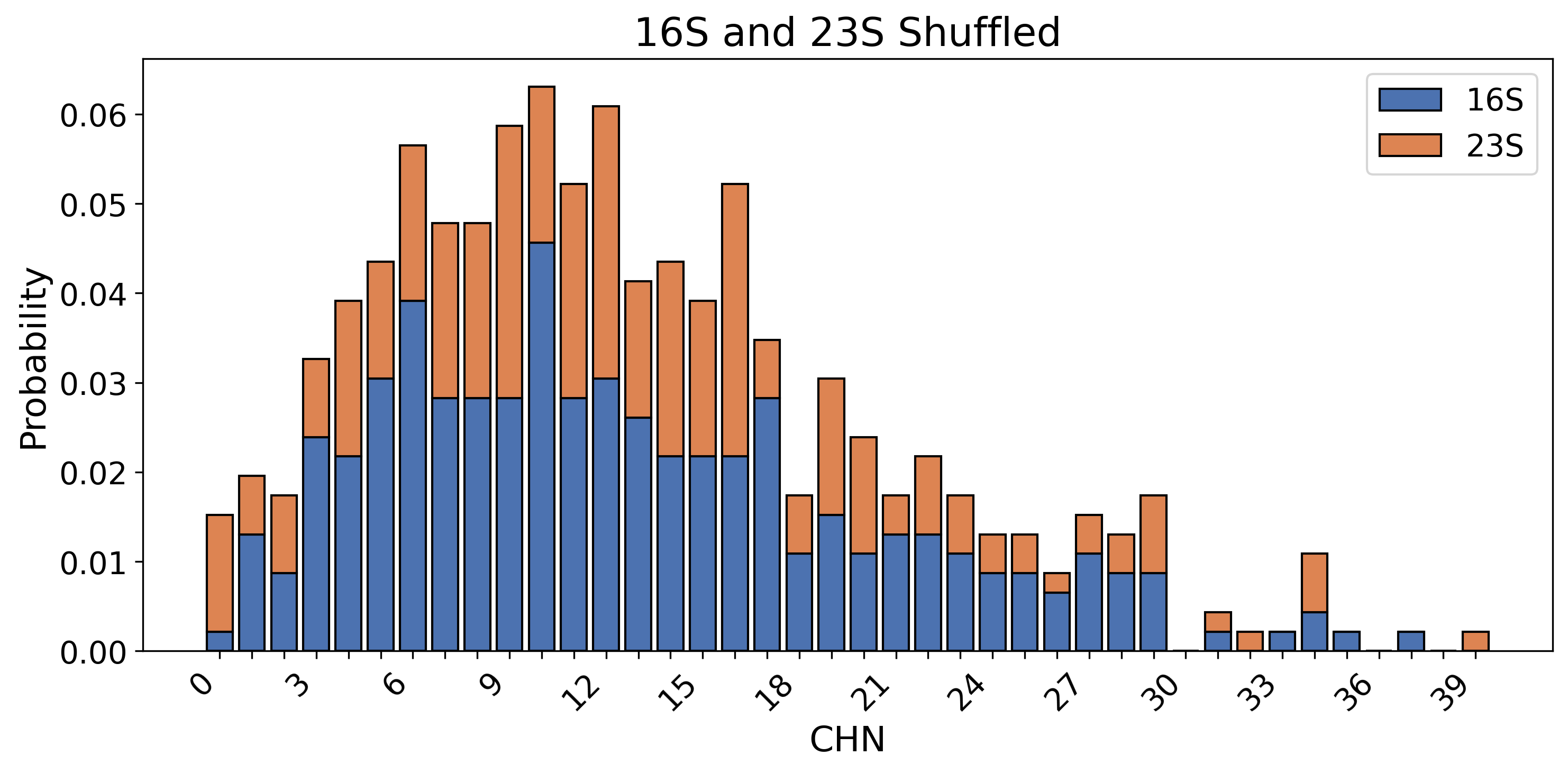}
     \includegraphics[width=0.48\textwidth]{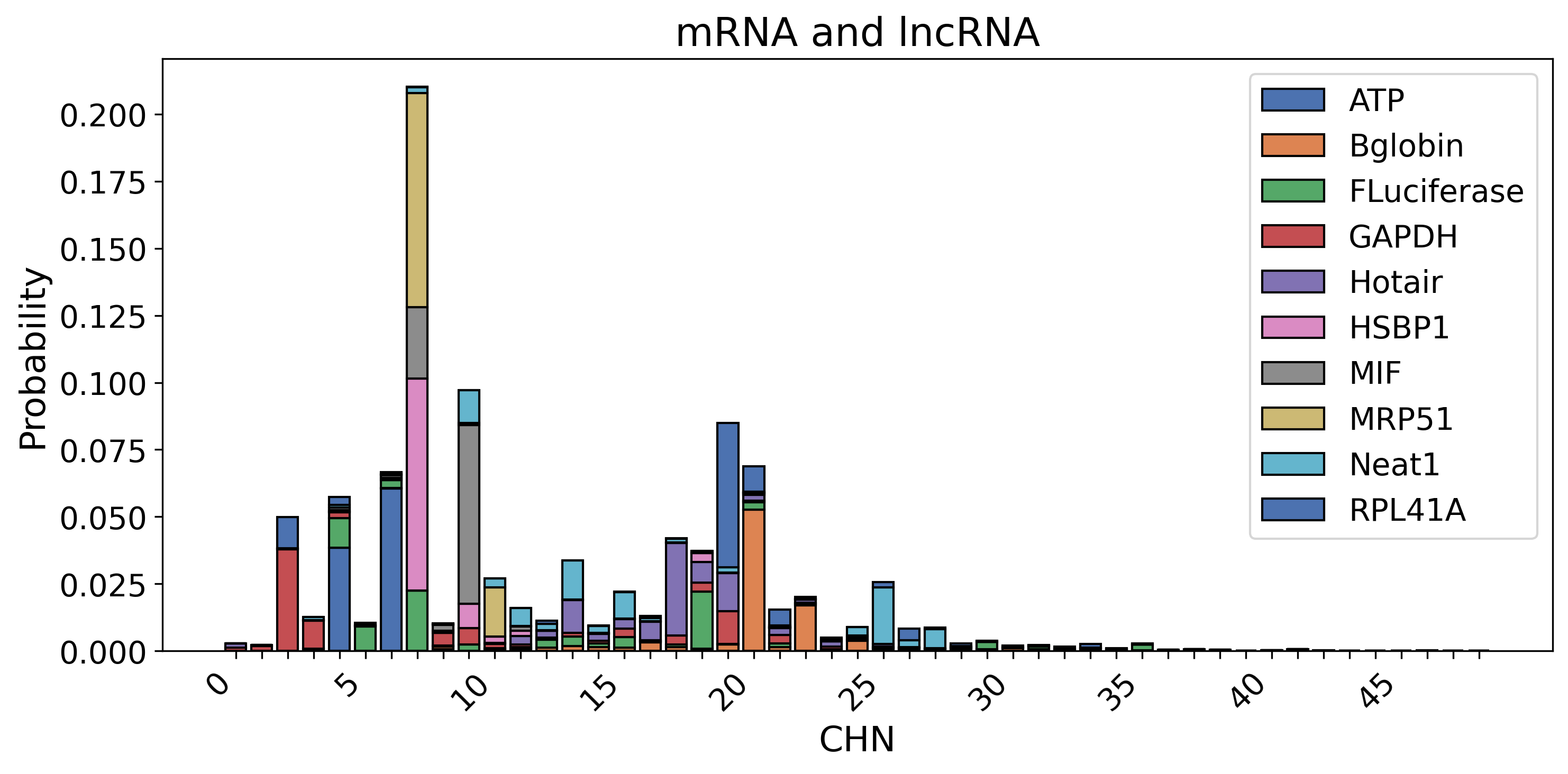}
     \includegraphics[width=0.48\textwidth]{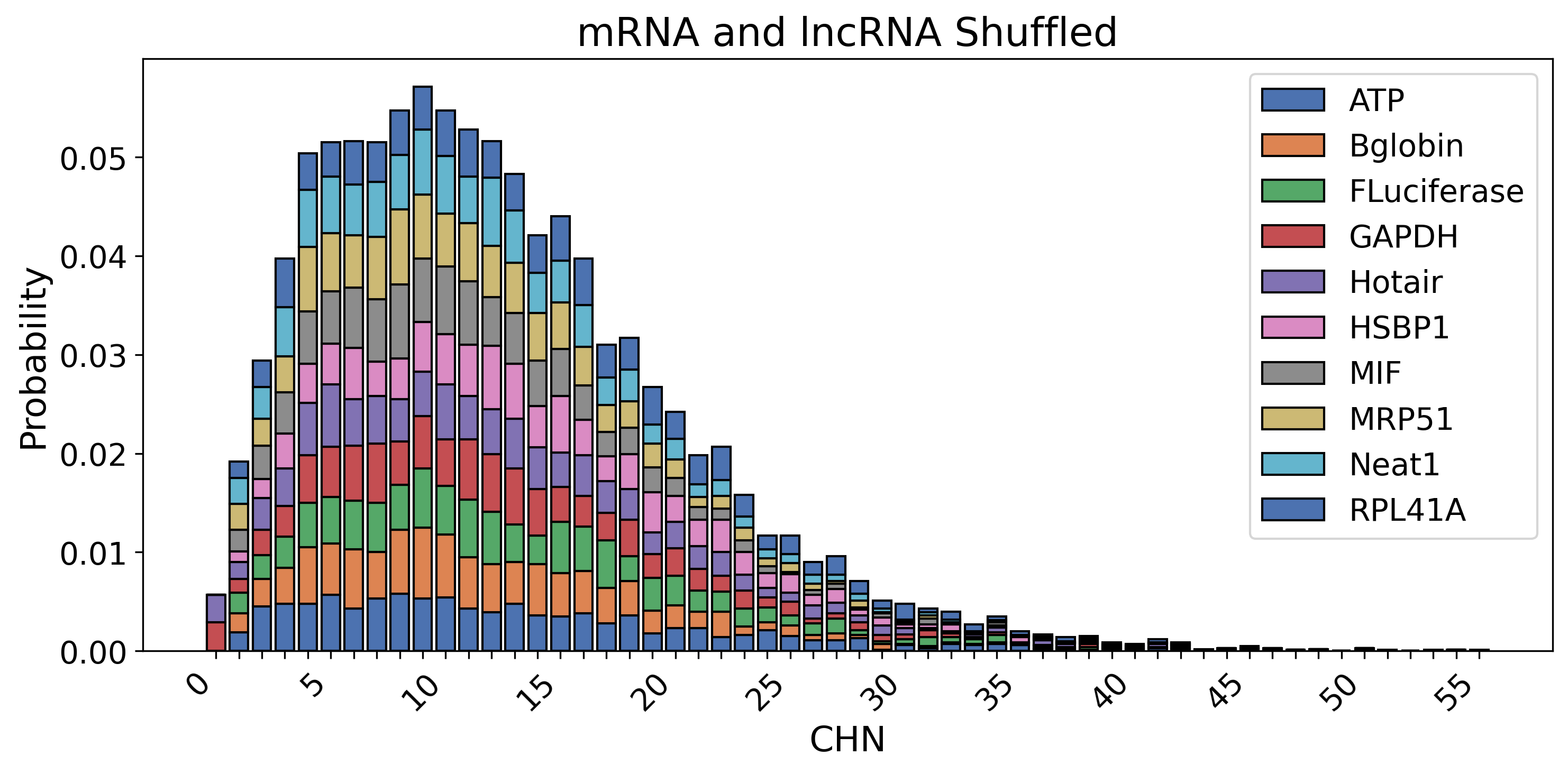}
 \end{center}

 \subsection{\ETE}
 \begin{center}
     \includegraphics[width=0.48\textwidth]{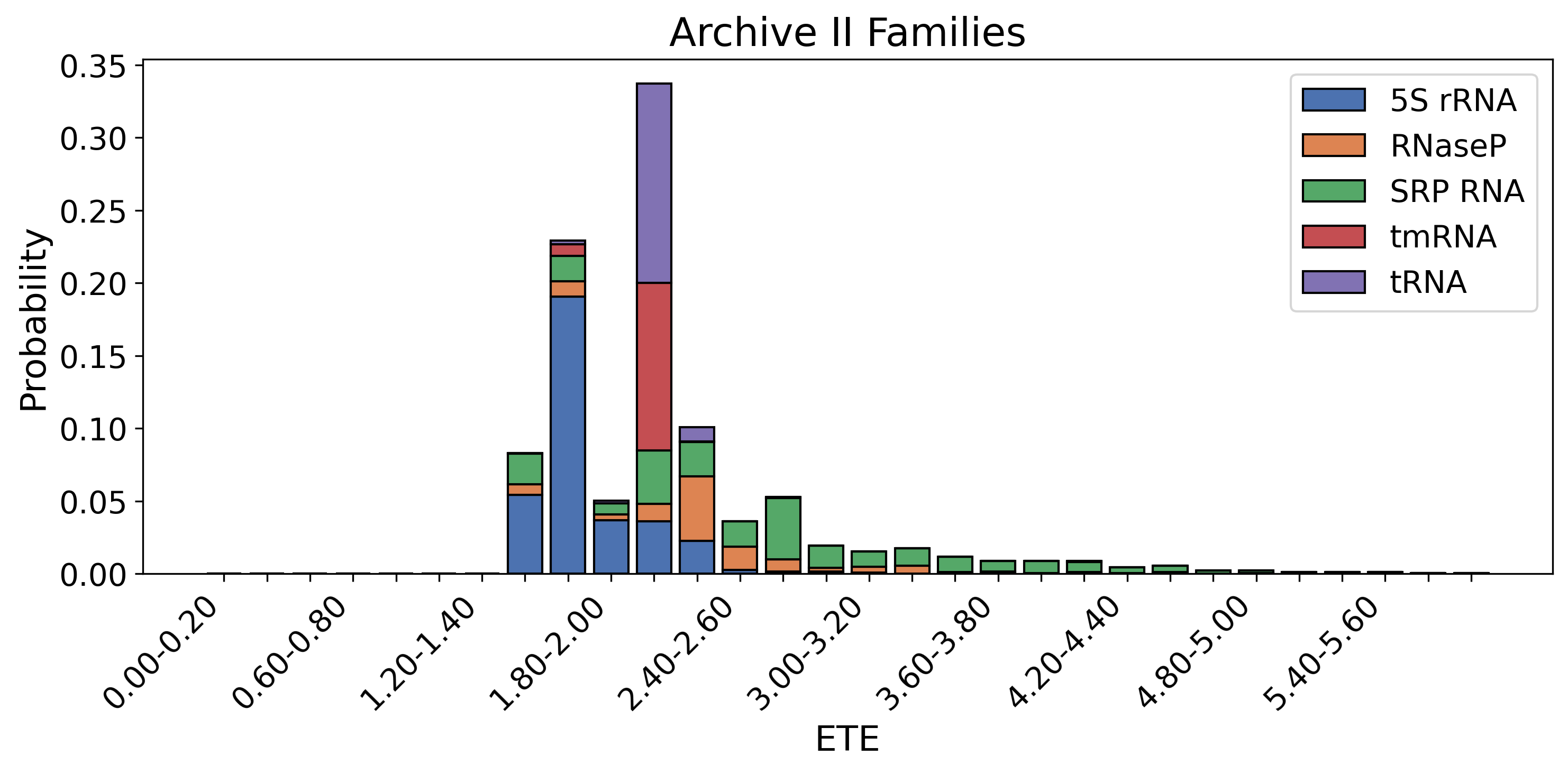}
     \includegraphics[width=0.48\textwidth]{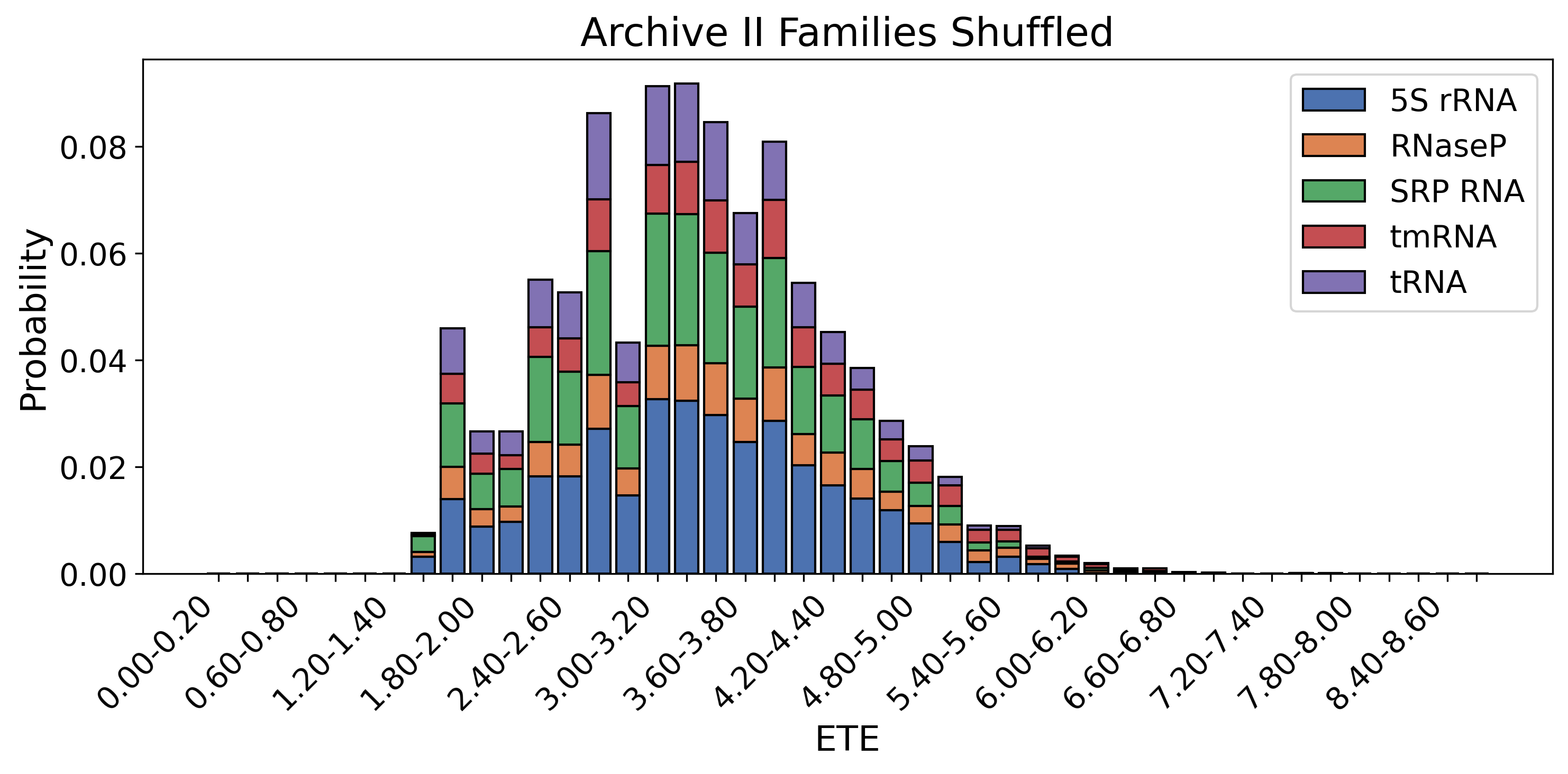}
     \includegraphics[width=0.48\textwidth]{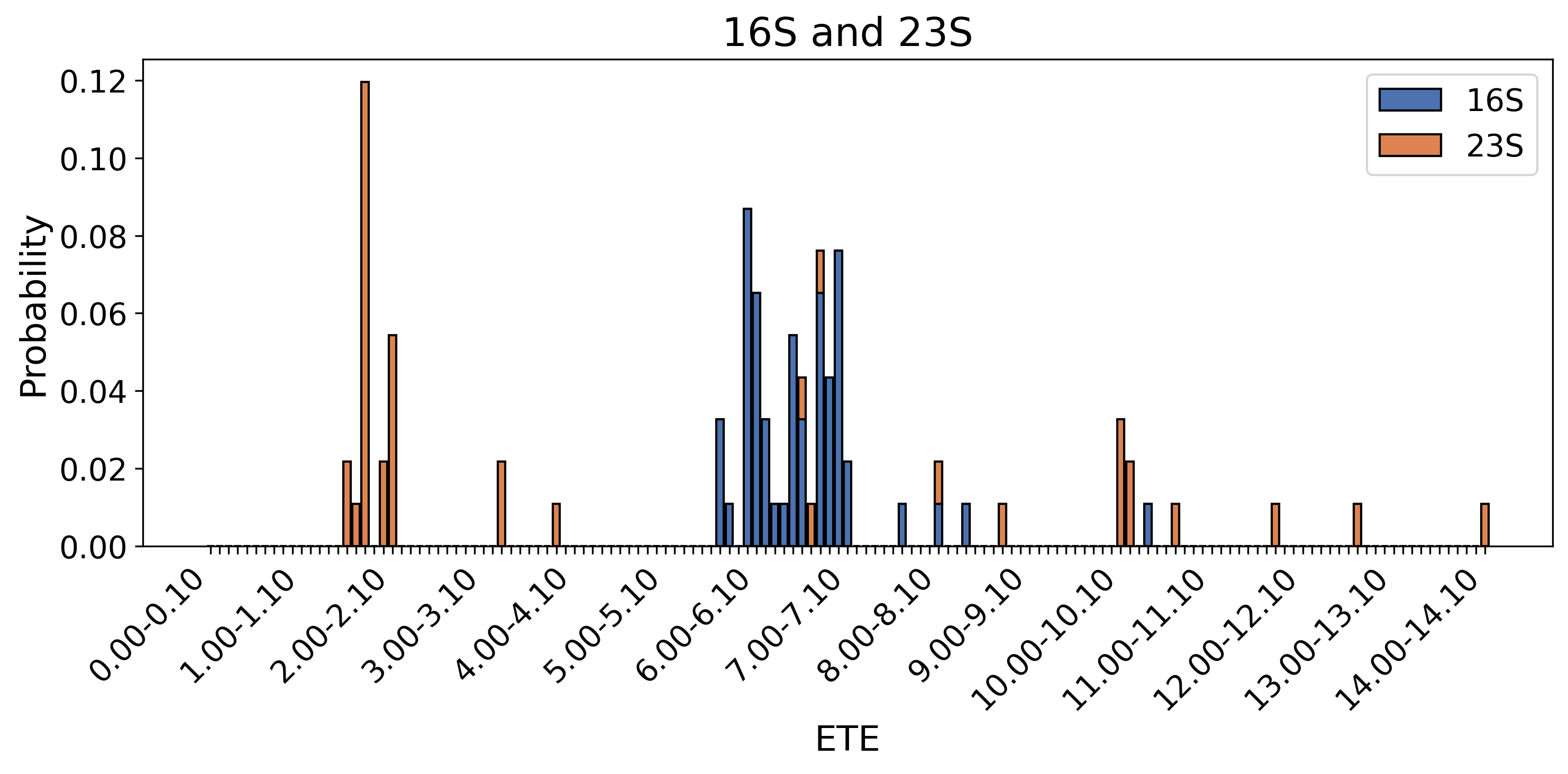}
     \includegraphics[width=0.48\textwidth]{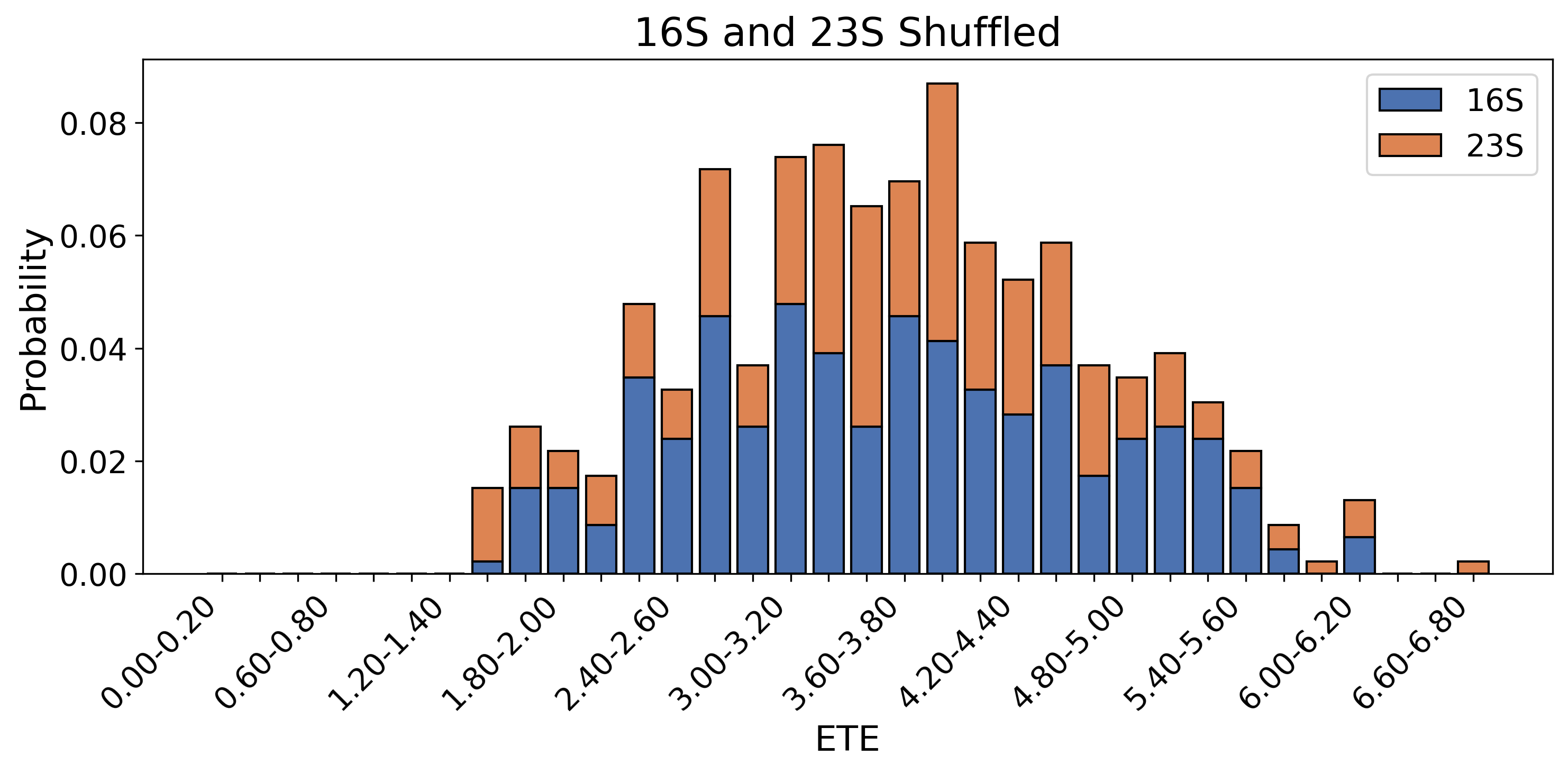}
     \includegraphics[width=0.48\textwidth]{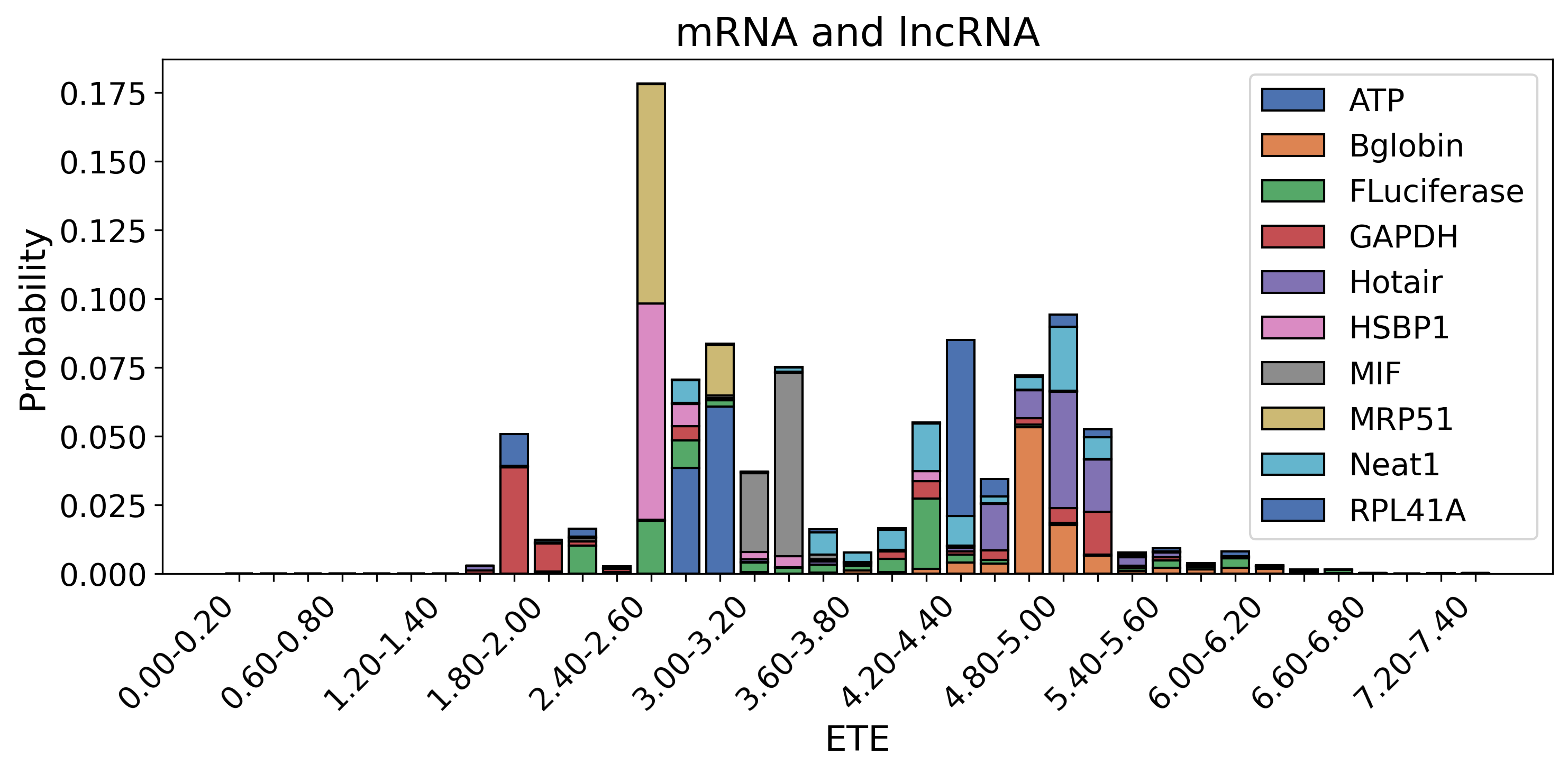}
     \includegraphics[width=0.48\textwidth]{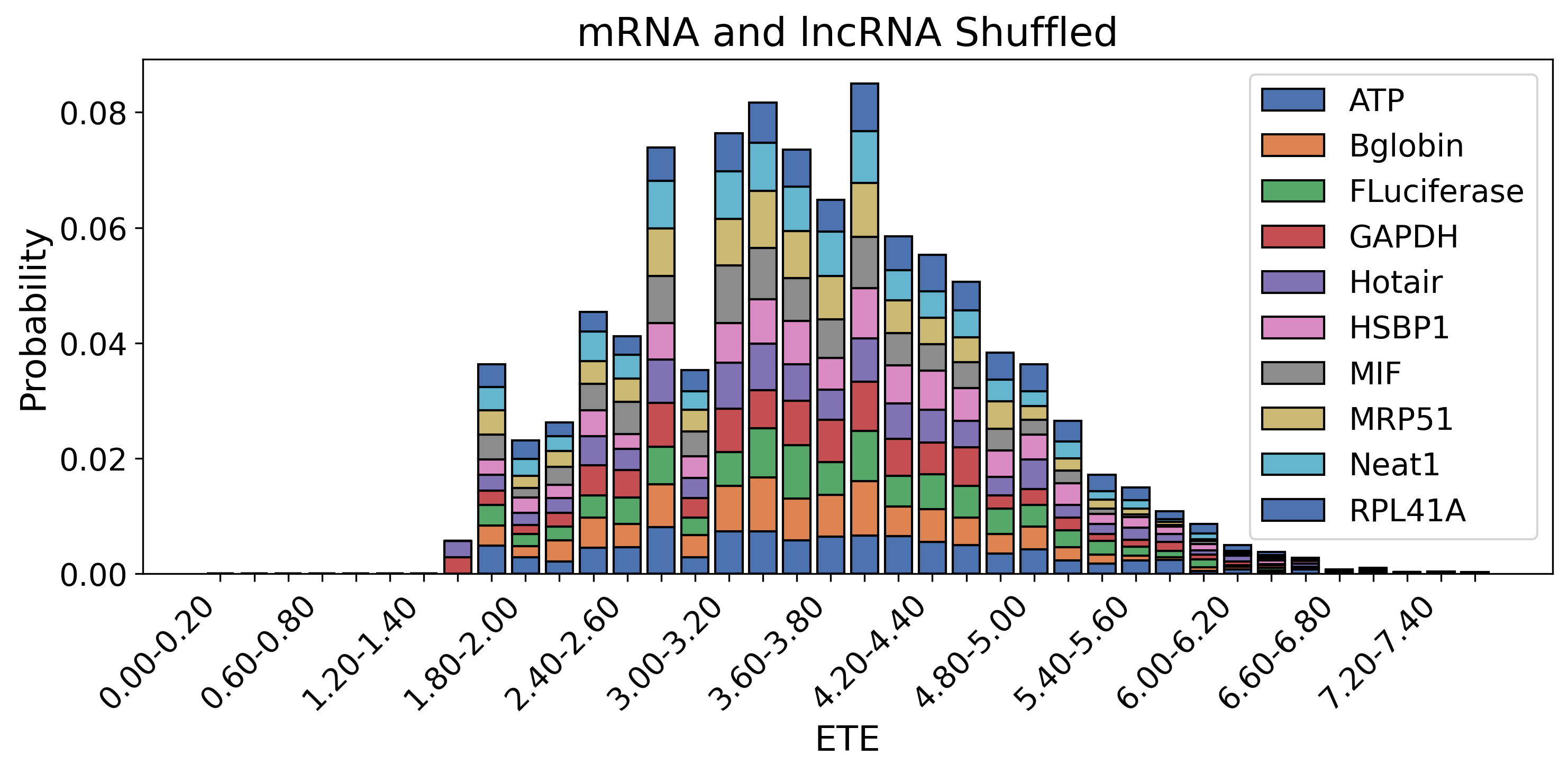}
 \end{center}

 \subsection{\HEL}
 \begin{center}
     \includegraphics[width=0.48\textwidth]{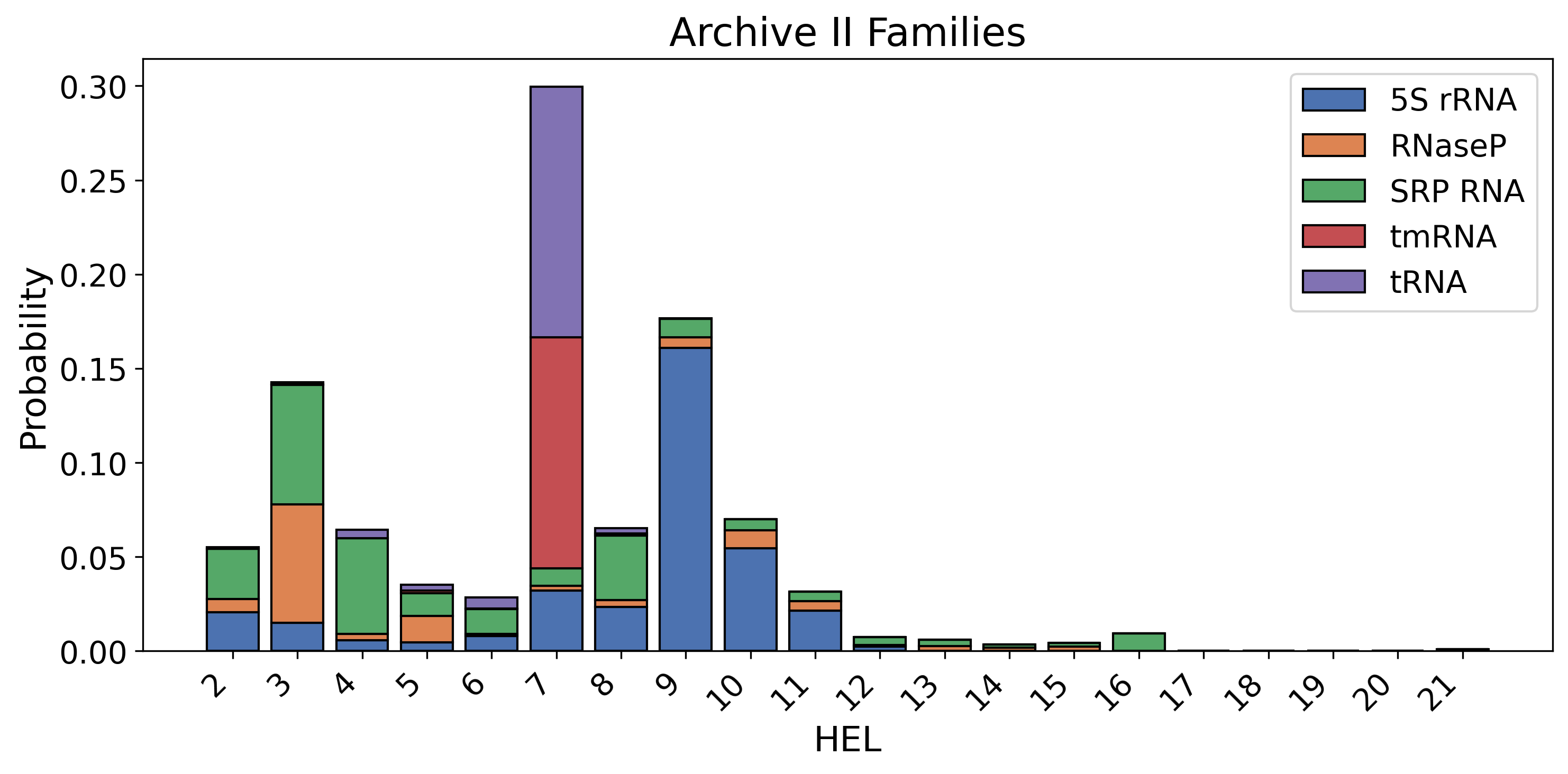}
     \includegraphics[width=0.48\textwidth]{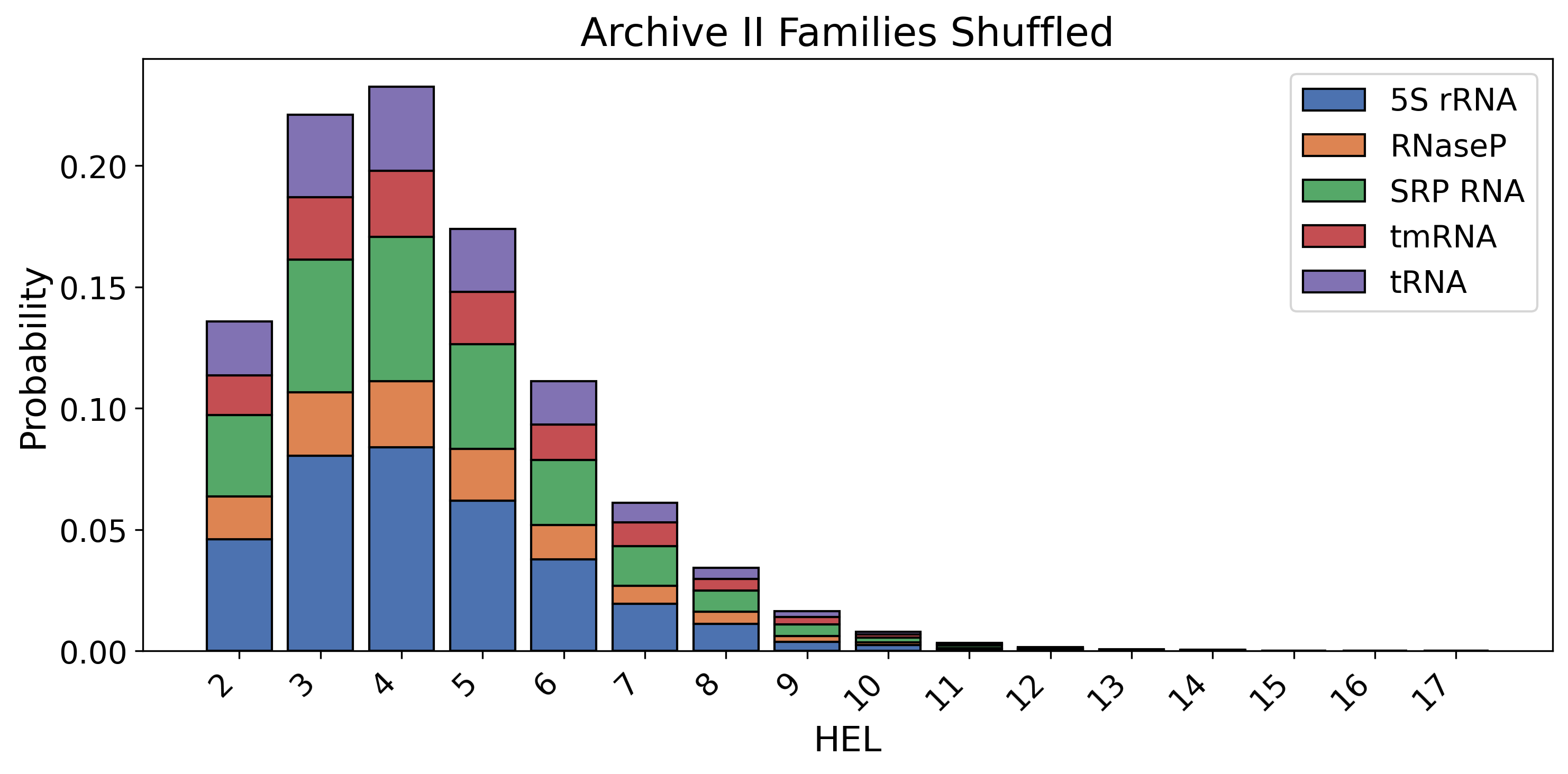}
     \includegraphics[width=0.48\textwidth]{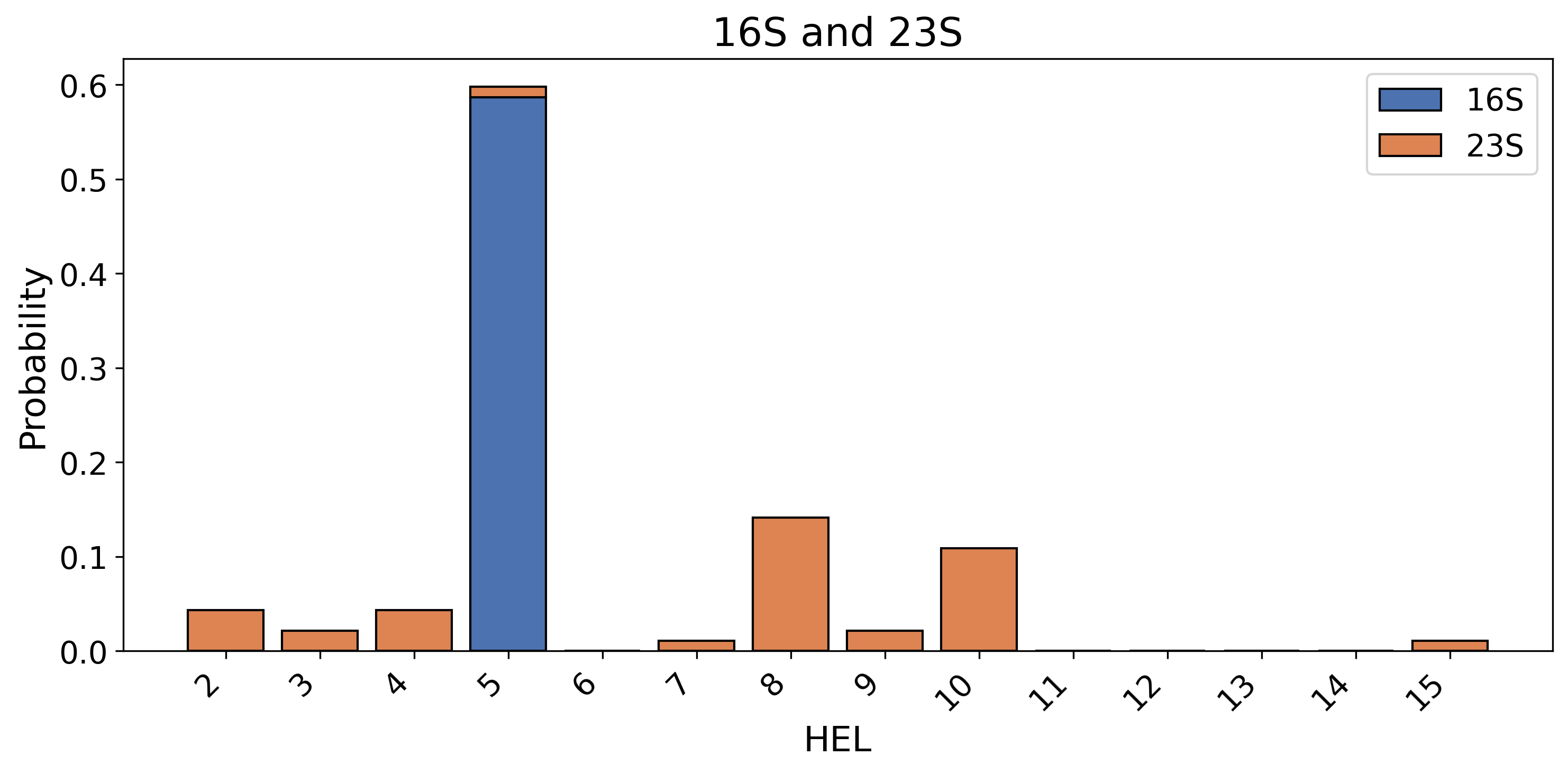}
     \includegraphics[width=0.48\textwidth]{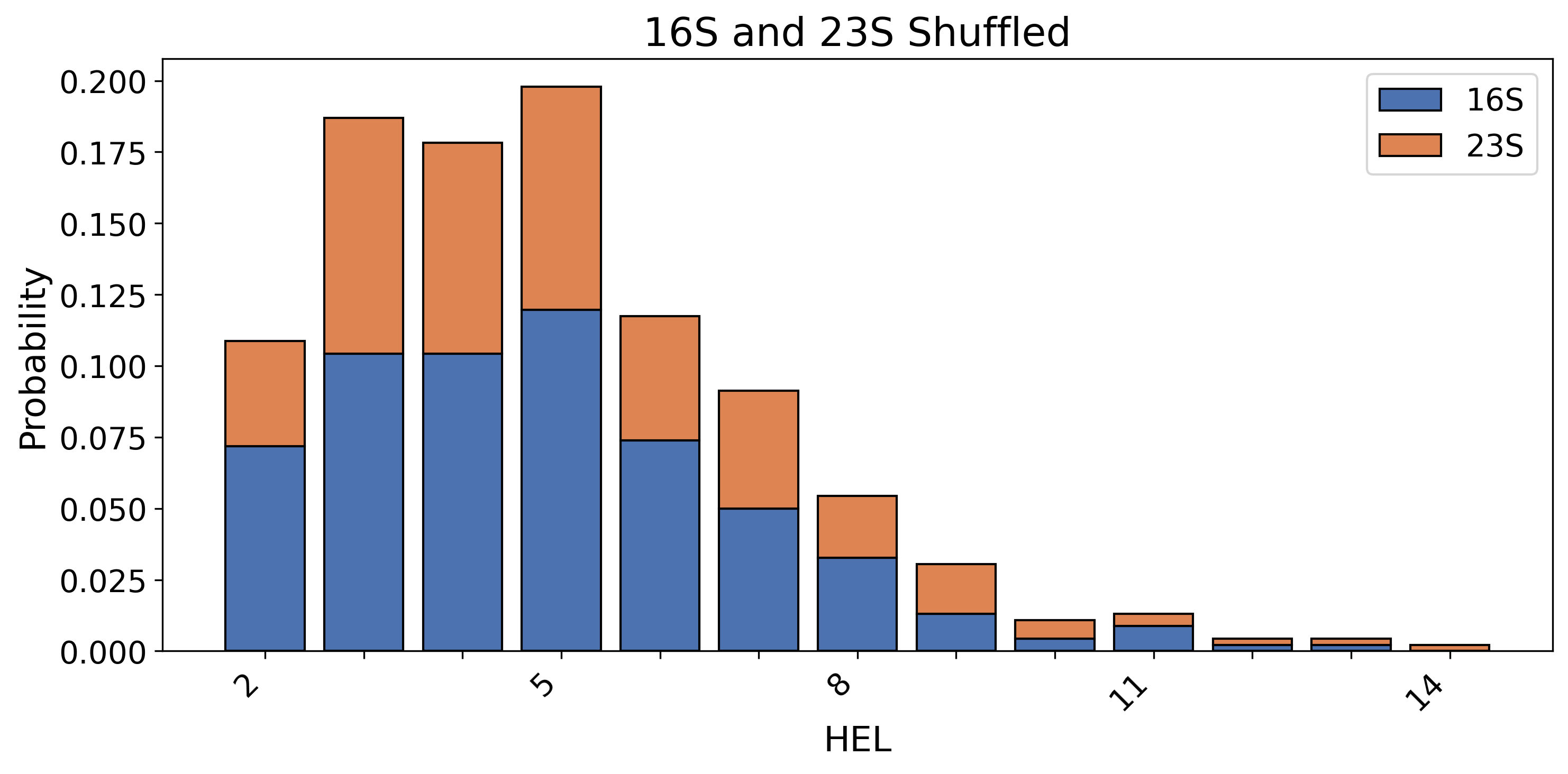}
     \includegraphics[width=0.48\textwidth]{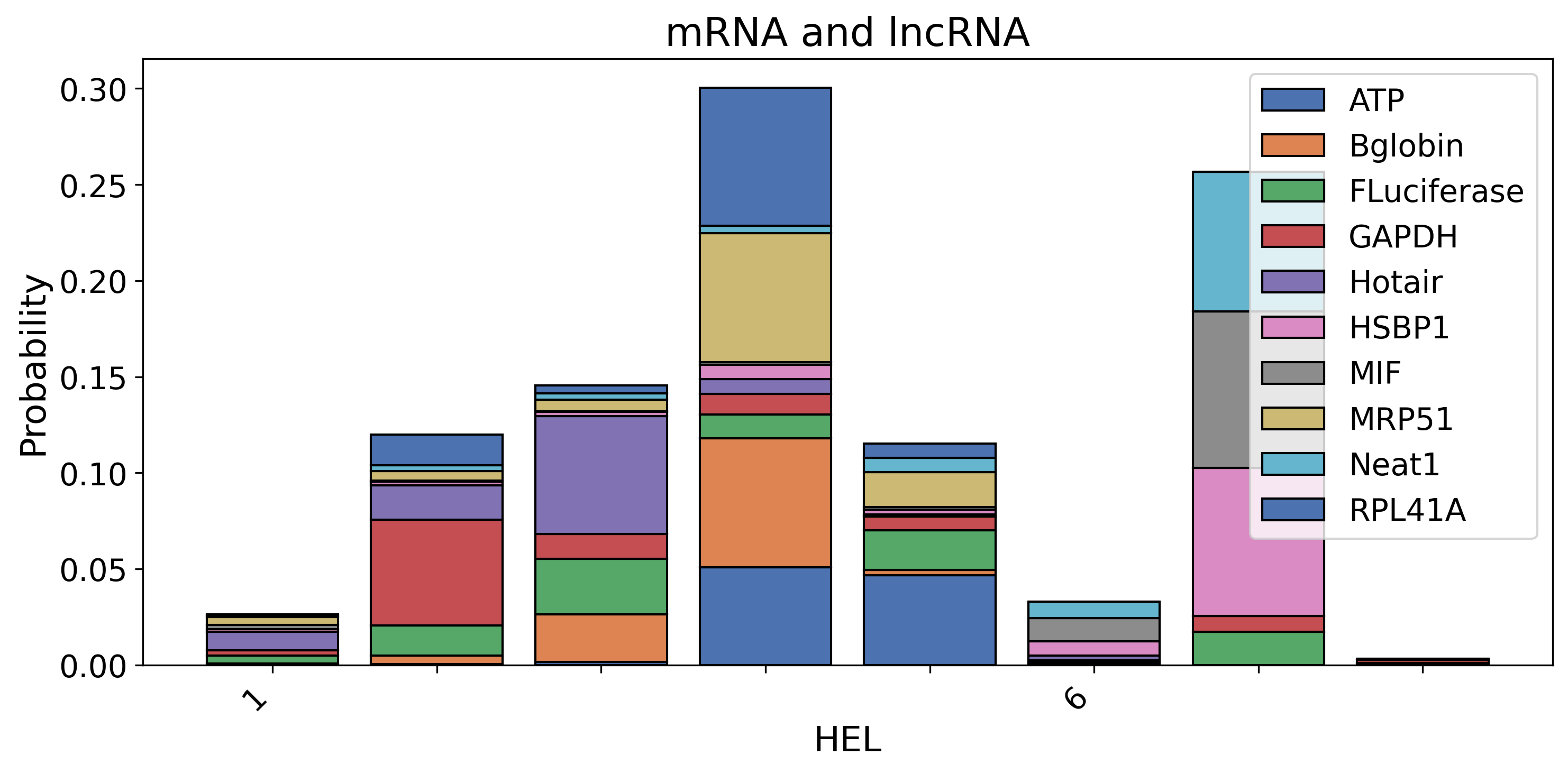}
     \includegraphics[width=0.48\textwidth]{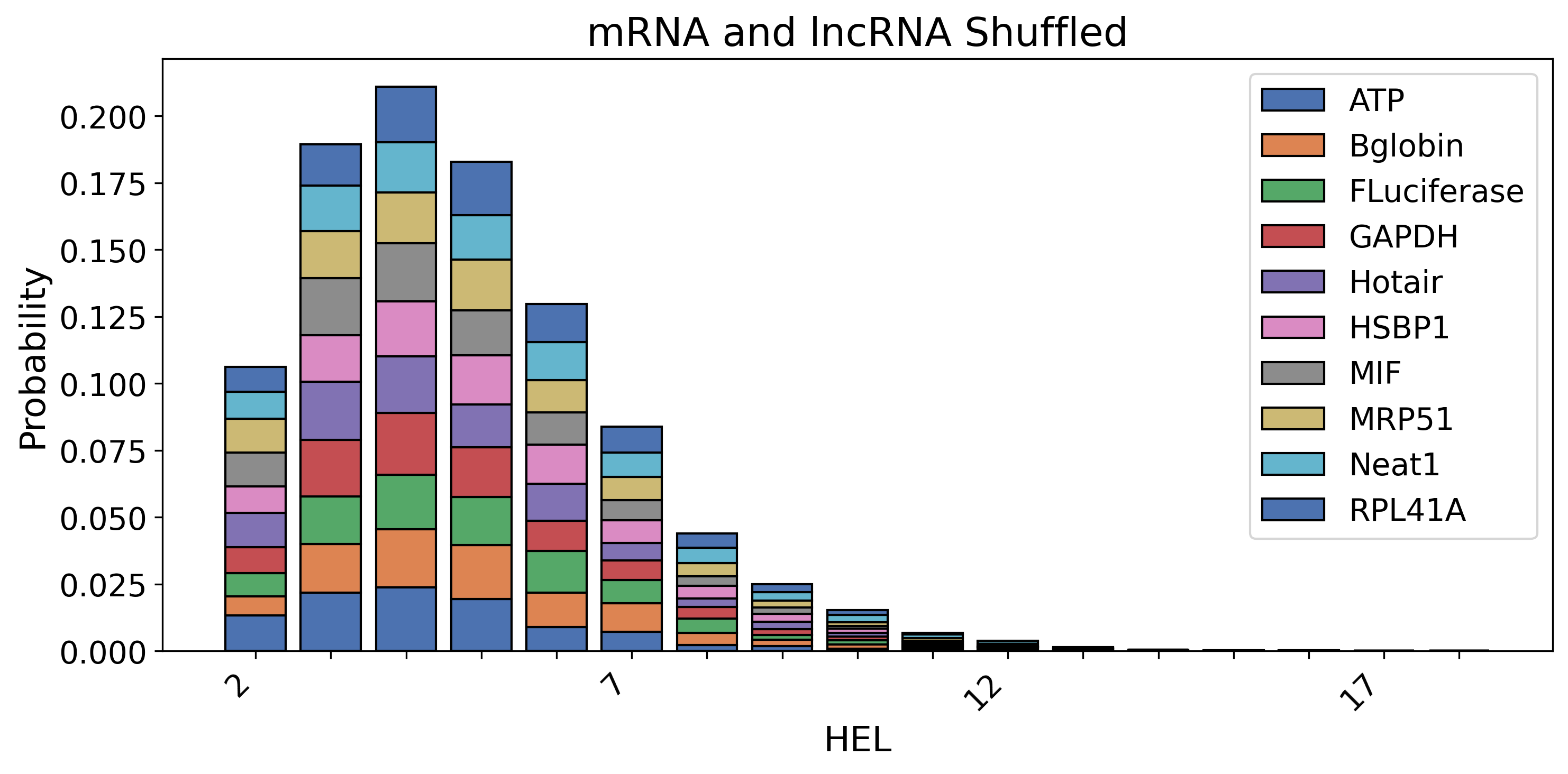}
 \end{center}

 \newpage

 \subsection{\STM}
 We did not assess the first stem length in the 16S and 23S sequences due to pseudoknots in the exterior loop.
 \begin{center}
     \includegraphics[width=0.48\textwidth]{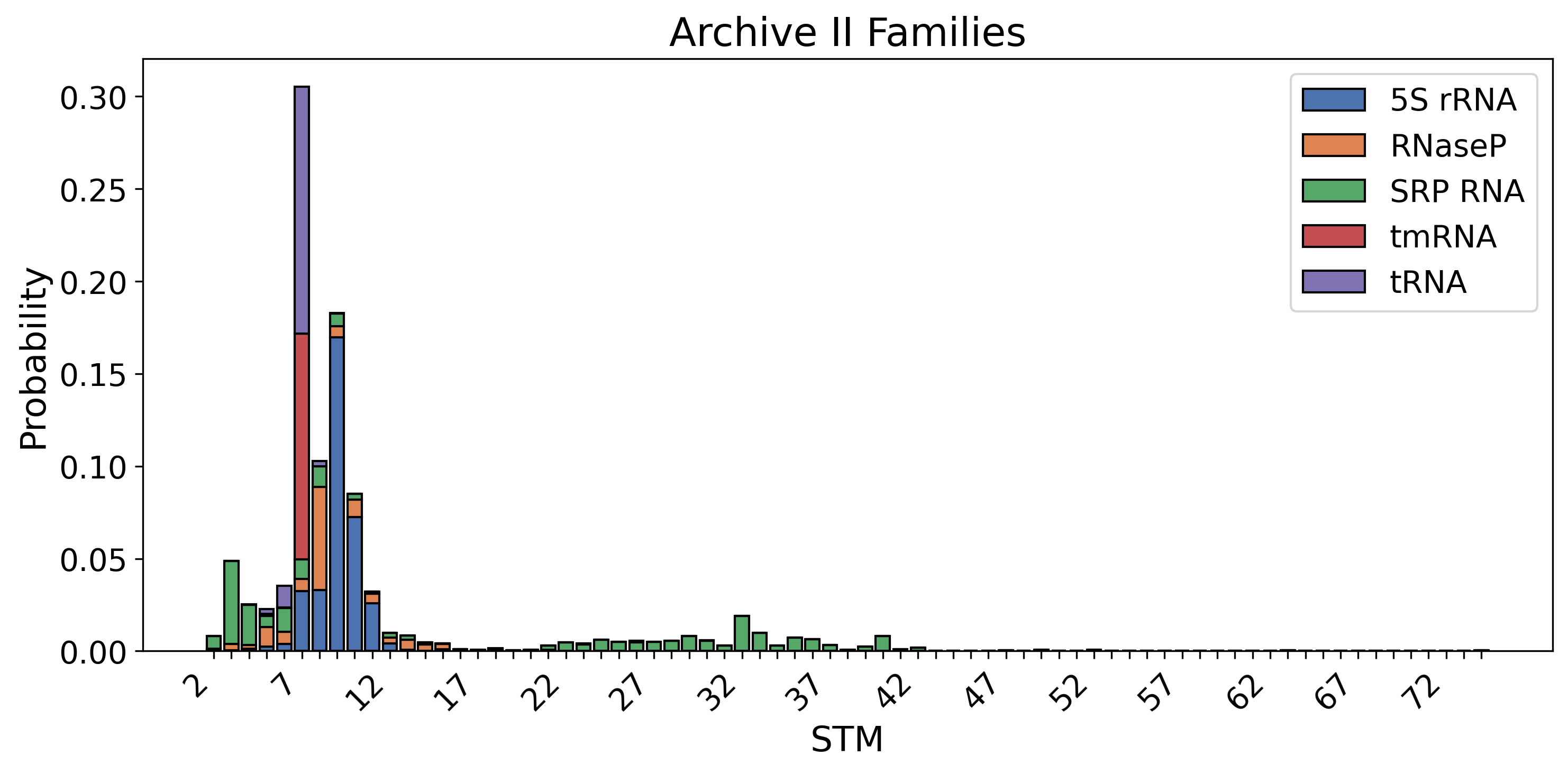}
     \includegraphics[width=0.48\textwidth]{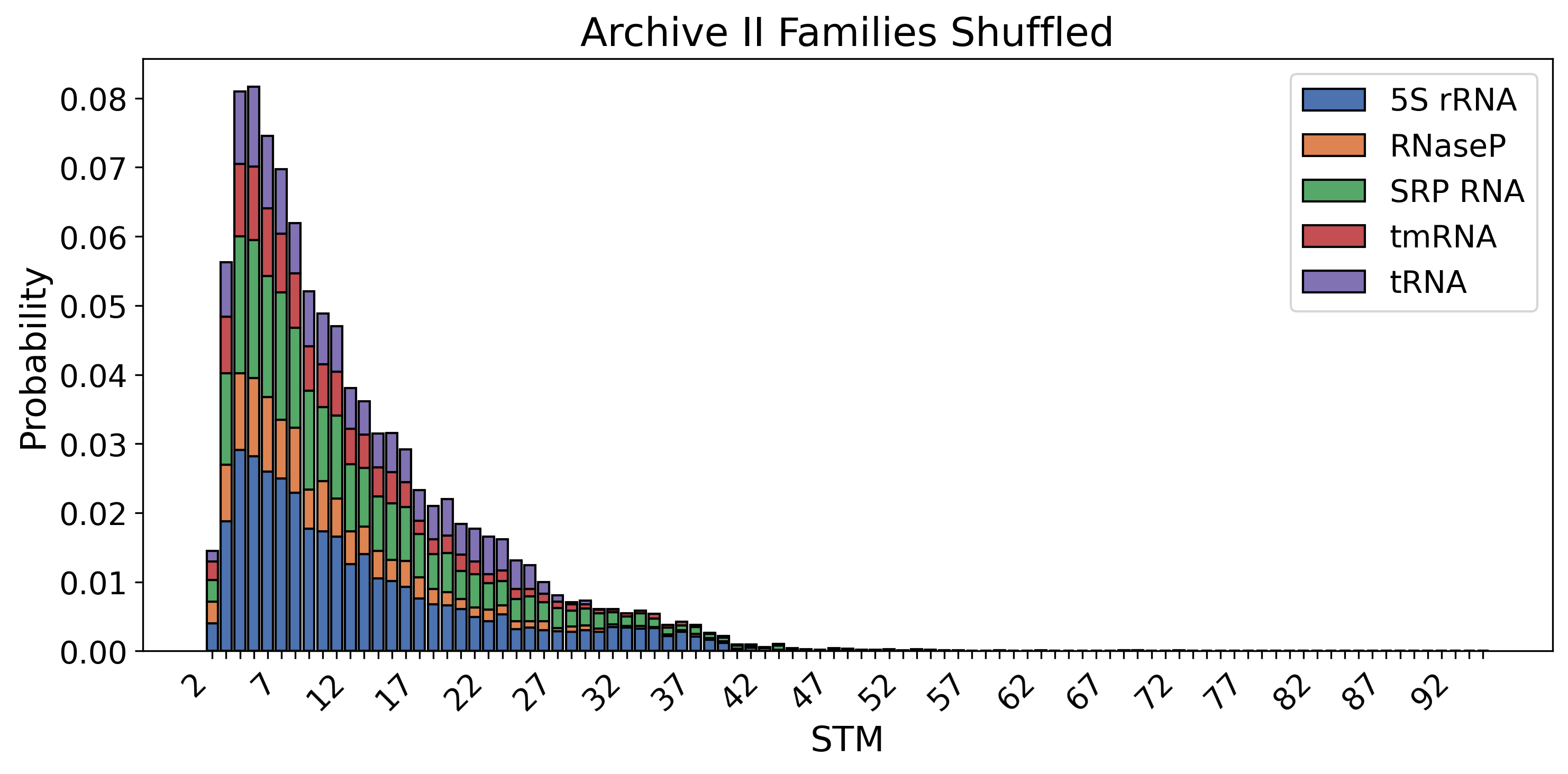}\\
     \hspace{.48\textwidth}
     \includegraphics[width=0.48\textwidth]{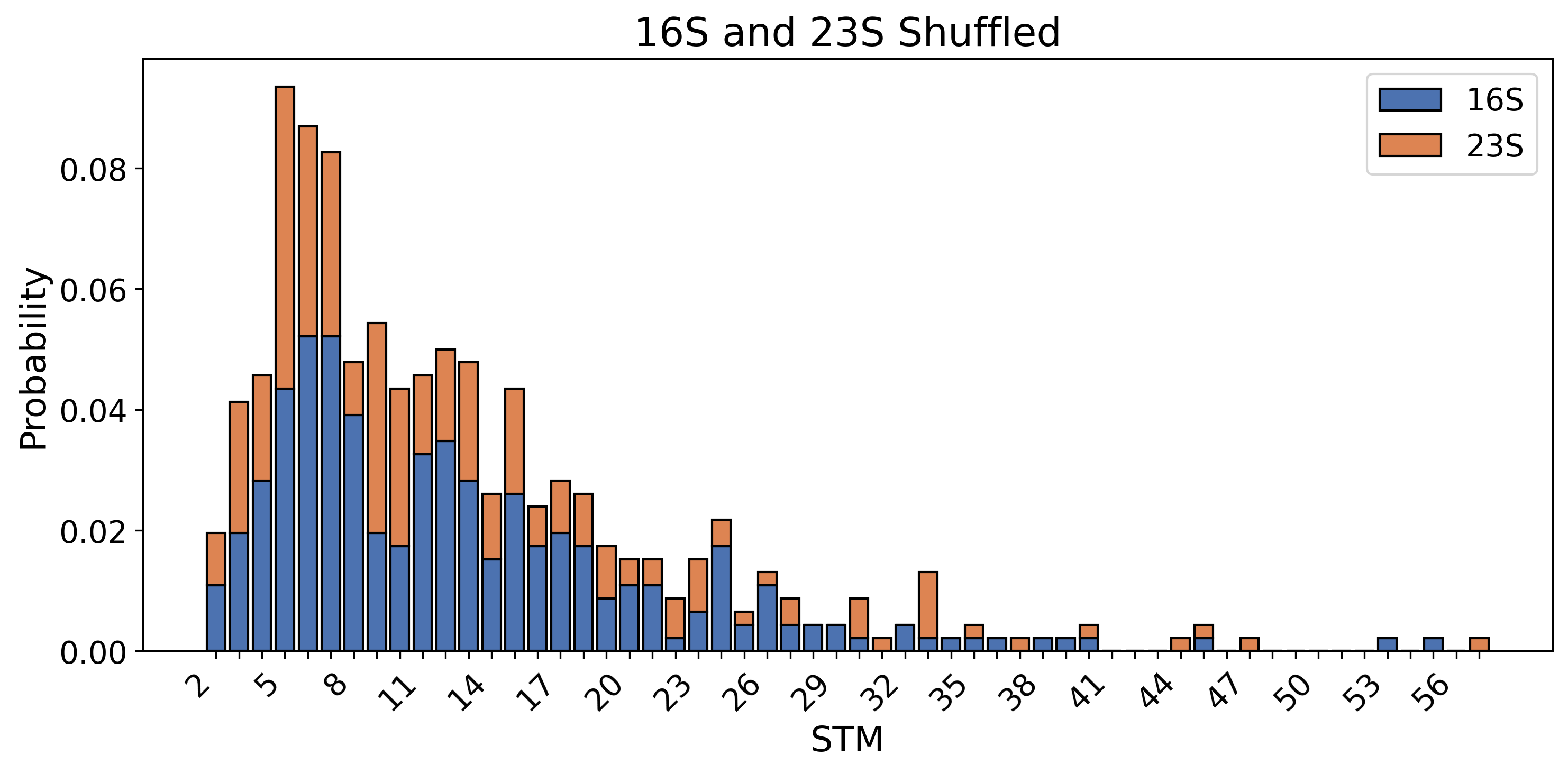}\\
     \includegraphics[width=0.48\textwidth]{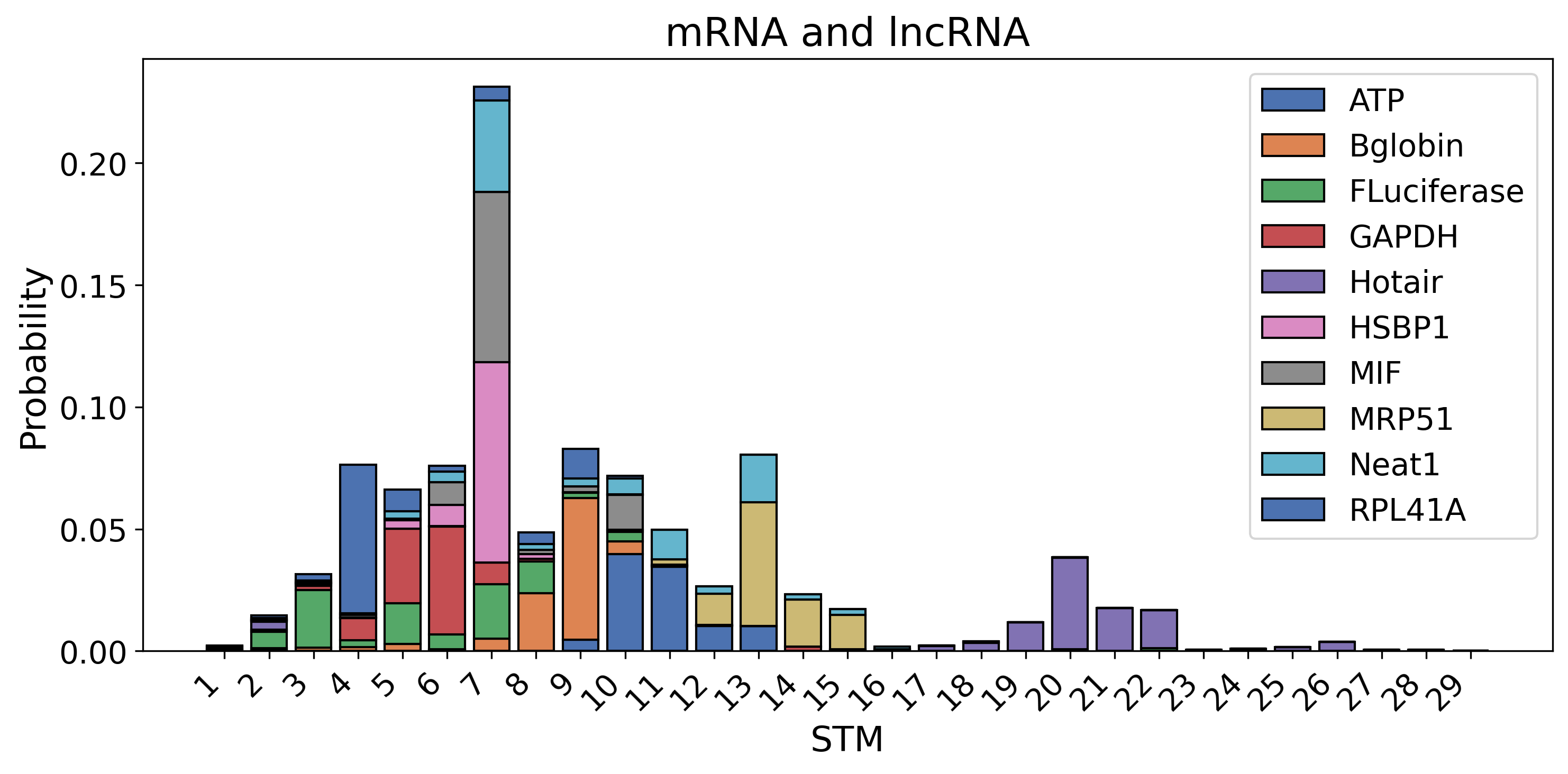}
     \includegraphics[width=0.48\textwidth]{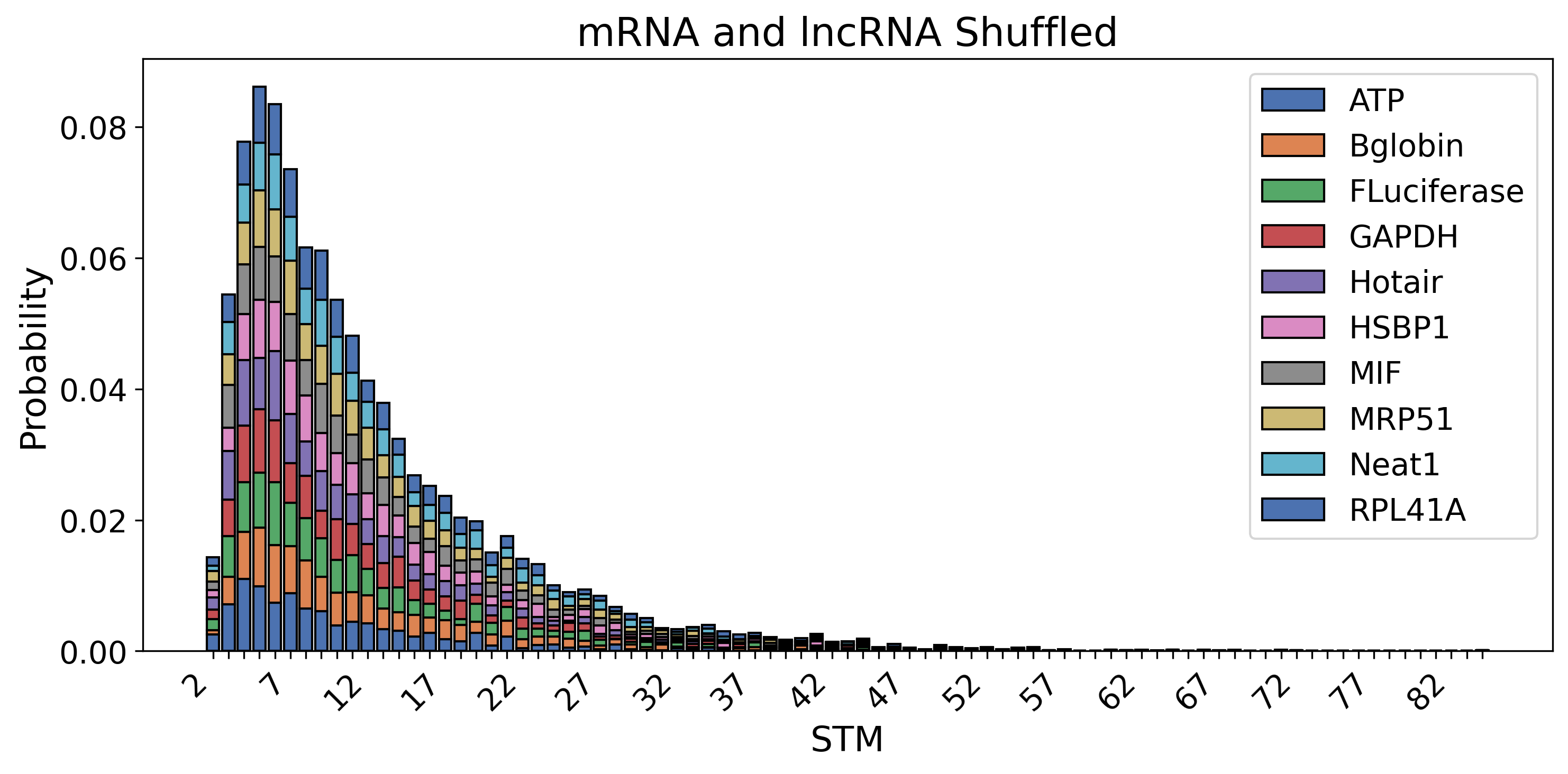}
 \end{center}

\end{appendices}

\newpage

\printbibliography

\end{document}